\documentclass[11pt]{amsart}
\usepackage{amssymb}
\usepackage{graphicx}
\usepackage{hyperref}
\usepackage{mathrsfs}
\usepackage{pb-diagram}
\usepackage{epstopdf}
\usepackage{amsmath,amsfonts,amsthm,enumerate,amscd,latexsym,curves}
\usepackage{bbm}
\usepackage[mathscr]{eucal}
\usepackage{epsfig,epsf,epic}
\usepackage{caption}
\usepackage{subcaption}
\usepackage{multirow}
\usepackage{color}
\usepackage[all]{xy}
\usepackage{fourier}
\usepackage{tikz-cd}
\usepackage{calc}
\usepackage{float}
\usepackage{enumitem}
\usepackage{comment}

\newtheorem{thm}{Theorem}[section]
\newtheorem{lem}[thm]{Lemma}

\newtheorem{cor}[thm]{Corollary}
\newtheorem{prop}[thm]{Proposition}
\newtheorem{definition}[thm]{Definition}
\newtheorem{example}[thm]{Example}
\newtheorem{remark}[thm]{Remark}

\numberwithin{equation}{section}

\newcommand{\act}{\mathrm{act}}
\newcommand{\tru}{\mathrm{tru}}
\newcommand{\R}{\mathbb{R}}
\newcommand{\N}{\mathbb{N}}
\newcommand{\cA}{\mathcal{A}}                       
\newcommand{\im}{\mathrm{Im}}

\newcommand{\SH}{\operatorname{SH}}
\newcommand{\CF}{\operatorname{CF}}
\newcommand{\HF}{\operatorname{HF}}
\newcommand{\RFH}{\operatorname{RFH}}
\newcommand{\RFC}{\operatorname{RFC}}

\newcommand{\Spec}{\operatorname{Spec}}
\newcommand{\set}[1]{\left\{ #1\right\}}            
\newcommand{\norm}[1]{\left\lVert#1\right\rVert}    

\begin{document}

\title[Barcode entropy for Reeb flows]{Barcode Entropy for Reeb Flows on contact manifolds with Liouville fillings}

\author{Elijah Fender}
\address[Elijah Fender]{Center for Geometry and Physics\\ Institute for Basic Science (IBS)\\ Pohang 37673, Korea}
\email{fender@ibs.re.kr}

\author{Sangjin Lee}
\address[Sangjin Lee]{Center for Geometry and Physics\\ Institute for Basic Science (IBS)\\ Pohang 37673, Korea}
\email{sangjinlee@ibs.re.kr}

\author{Beomjun Sohn}
\address[Beomjun Sohn]{Seoul National University\\ Department of Mathematical Sciences\\ Research Institute in Mathematics\\ 08826 Seoul, South Korea}
\email{bum2220@snu.ac.kr}

\begin{abstract}
    We study the topological entropy of Reeb flows on contact manifolds with Liouville fillings. With the theory of persistence modules, we define $\SH$-barcode entropy from the symplectic homology of a filling. We prove that the $\SH$-barcode entropy is independent of the choice of the filling and that the barcode entropy provides a lower bound for the topological entropy of the Reeb flow.
\end{abstract}

\maketitle
\setcounter{tocdepth}{1}
\tableofcontents

\section{Introduction}
\label{section introduction}


{\em Persistence modules}, which originated in the study of topological data analysis \cite{Zomorodian-Carlsson05}, were introduced to the study of Hamiltonian dynamics by Polterovich and Shelukhin \cite{Polterovich-Shelukhin16}. After their work, persistence modules have been extensively used in symplectic and contact topology. 
For example, \cite{Usher-Zhang, Stevenson18, Zhang19, Chor-Meiwes21, Kislev-Shelukhin21, Allais22, Kawamoto22, Shelukhin22a, Allais23, Dietzsch23} used persistence modules constructed from various Floer homologies to study Hamiltonian dynamics on symplectic manifolds. In \cite{Stojisavljevic-Zhang21, Shelukhin22b, Shelukhin22c}, the authors studied persistence modules constructed from symplectic homologies of Liouville domains and their applications.
Moreover, \cite{Buhovsky-Humiliere-Seyfaddini21, jannaud21, joly21, LeRoux-Seyfaddini-Viterbo21} studied $C^0$-symplectic topology by defining persistence modules for Hamiltonian homeomorphisms on surfaces.

In the current paper, we will consider persistence modules coming from filtered symplectic homologies of Liouville domains to study the topological entropy of Reeb flows on their contact boundaries. Before stating our results, we begin by offering a brief background and motivation for our study.

Various works have provided examples of contact structures on each of which every Reeb flow has positive topological entropy.
For example, it is shown by Frauenfelder and Schlenk \cite{Frauenfelder-Schlenk06} and Macarini and Schlenk \cite{Macarini-Felix11} that every Reeb flow of the standard contact structure on the unit cotangent bundle of a certain Riemannian manifold, more precisely energy hyperbolic in the sense of \cite{Macarini-Felix11}, has positive topological entropy.
Let us recall that every contact form supporting the standard contact structure of a unit cotangent bundle can be realized as the standard contact form on the boundary of a fiberwise star-shaped domain in the cotangent bundle.
In this setting, the authors studied some exponential growth rate of filtered wrapped Floer homologies on fiberwise star-shaped domains in order to provide a lower bound for the topological entropy of the Reeb flows.
They showed that if this exponential growth rate is positive for one Reeb flow, then the same holds for every Reeb flow.
Using this observation, they proved the positivity of topological entropy for every Reeb flow of the standard contact structure on the unit cotangent bundle of an energy hyperbolic manifold. Similar or related works also appeared in \cite{Alves16a, Alves16b, Meiwes18, Alves-Meiwes19, Alves19, Alves-Colin-Honda19, Dahinden20, Dahinden21, Alves-Pirnapasov22}.



Recently, in order to study the topological entropy of Hamiltonian diffeomorphisms, {\c C\. inel\. i}, Ginzburg, and G\"urel \cite{Cineli-Ginzburg-Gurel21} introduced a Floer-theoretic entropy, called barcode entropy, from a perspective of persistence homology and Lagrangian Floer theory. Then, \cite{Cineli-Ginzburg-Gurel21} showed the barcode entropy bounded above by the topological entropy of Hamiltonian diffeomorphisms.
Extending this approach, Ginzburg, G\"urel, and Mazzucchelli \cite{Ginzburg-Gurel-Mazzucchelli22} defined barcode entropy of geodesic flows on unit-tangent bundles of Riemannian manifolds from the Morse homology of the energy functional on the free loop spaces.
Then, \cite{Ginzburg-Gurel-Mazzucchelli22} proved that the barcode entropy low-bounds the topological entropy of geodesic flows.

There are examples of contact structures that admit a Reeb flow with zero topological entropy, for example, the unit cotangent bundle of $S^2$.
On such contact structures, the lower bounds of the topological entropy of Reeb flows given by \cite{Frauenfelder-Schlenk06, Macarini-Felix11}, which is the exponential growth rate of filtered wrapped Floer homology, should be zero. However, it is known that every contact structure admits a Reeb flow with positive topological entropy.
This means that the exponential growth rate considered in \cite{Frauenfelder-Schlenk06, Macarini-Felix11} may not fully detect the topological entropy of such Reeb flows.

On the other hand, for the case of $S^2$, it is known that there are geodesic flows with positive barcode entropy from \cite[Theorem C]{Ginzburg-Gurel-Mazzucchelli22}. We note that the geodesic flows of Riemannian manifolds are Reeb flows on their unit cotangent bundles. In this example, the barcode entropy in \cite{Ginzburg-Gurel-Mazzucchelli22} may detect more information about the topological entropy of such Reeb flows. This example motivates us to generalize the notion of barcode entropy to Reeb flows.

In this paper, we define a barcode entropy for Reeb flows by generalizing the barcode entropy in \cite{Cineli-Ginzburg-Gurel21, Ginzburg-Gurel-Mazzucchelli22}. More precisely, we define a Floer-theoretic entropy of Liouville-fillable contact manifolds by applying the persistence module theory to filtered symplectic homology. We call the Floer-theoretic entropy {\em symplectic homology barcode entropy}, or simply {\em $\SH$-barcode entropy}. With the $\SH$-barcode entropy, we prove the following results:
\vspace{-2mm}
\begin{itemize}
	\item $\SH$-barcode entropy of a contact manifold is independent of the choice of the filling.
	\item $\SH$-barcode entropy provides a lower bound for the topological entropy of the Reeb flow.
\end{itemize}
\vspace{-2mm}
More detailed results will be given in the next subsection.

\subsection{Results}
Let $(Y,\alpha)$ be a closed non-degenerate contact manifold. Throughout this paper, we assume that $(Y,\alpha)$ can be realized as the boundary of a Liouivlle domain $(W, \lambda)$ with vanishing first Chern class.
We construct a persistence module from the filtered symplectic homology of $(W,\lambda)$.
Then, by the {\em Structure Theorem}, a well-known result in persistence module theory, we transform this persistence module into a {\em barcode}, which is a multiset of intervals. 
Each interval, an element of a barcode, is called a bar. We denote by $B_{\SH}(W, \lambda)$ the barcode constructed from the filtered symplectic homology of $(W,\lambda)$. We define {\em symplectic homology barcode entropy}, or simply {\em $\SH$-barcode entropy}, as the exponential growth rate of the number of not-too-short bars in this barcode under the action filtration level.

Because of its definition, $\SH$-barcode entropy seems to be dependent on the choice of the filling. When a contact manifold admits two different Liouville fillings with vanishing first Chern classes, the corresponding two symplectic homologies can differ greatly. For example, the symplectic homology of a 2-dimensional disc vanishes, while the symplectic homology of a higher genus surface with a single boundary is infinite-dimensional. Even if two different Liouville fillings have the same symplectic homology, the respective barcodes can be considerably different. Nevertheless, we show that the $\SH$-barcode entropy depends only on the contact manifold $(Y,\alpha)$. More precisely,
\begin{thm}[= Theorem \ref{invariance: main thm}]
	\label{thm independence of filling}
    Let $(Y,\alpha)$ be a closed non-degenerate contact manifold with Liouville fillings  $(W_1, \lambda_1)$ and $(W_2, \lambda_2)$ whose first Chern classes vanish. Then, the $\SH$-barcode entropy of $(Y,\alpha)$ defined by the symplectic homology of $(W_1, \lambda_1)$ and that of $(W_2, \lambda_2)$ coincides.
\end{thm}


The key ingredient of the proof is the existence of a constant $C=C(Y,\alpha)$>0 with
\[B_{\SH}(W_1,\lambda_1)_{+}^{\le C}=B_{\SH}(W_2,\lambda_2)_{+}^{\le C},\]
where $B_{\SH}(W, \lambda)^{\le C}_+$ denotes the multiset of bars with length at most $C$ and with positive left-endpoint in the barcode constructed from the symplectic homology of a Liouville domain $(W,\lambda)$. The constant $C(Y,\alpha)$ only depends on the contact boundary and is independent of the fillings. Using this observation, we will show that the exponential growth rates from two barcodes $B_{\SH}(W_1, \lambda_1), B_{\SH}(W_2, \lambda_2)$ are the same. For a complete proof, see Section \ref{subsection fillings}.

Since the definition of $\SH$-barcode entropy is motivated by \cite{Cineli-Ginzburg-Gurel21, Ginzburg-Gurel-Mazzucchelli22}, it is natural to expect that the results obtained in those papers can be extended to the setting of $\SH$-barcode entropy. Here, we prove the following Theorem which is an analogue of \cite[Theorem A]{Cineli-Ginzburg-Gurel21} and \cite[Theorem A]{Ginzburg-Gurel-Mazzucchelli22}.

\begin{thm}[=Theorem \ref{thm vs topological entropy formal}]
    \label{thm vs topological entropy} 
    Let $(Y,\alpha)$ be a closed non-degenerate contact manifold with some Liouville filling whose first Chern class vanishes. Then, the $\SH$-barcode entropy of $(Y,\alpha)$ is less than or equal to the topological entropy of the Reeb flow of $(Y,\alpha)$.
\end{thm}

The proof of Theorem \ref{thm vs topological entropy} in this paper follows the same lines as in \cite{Cineli-Ginzburg-Gurel21, Ginzburg-Gurel-Mazzucchelli22}. Recall that a symplectic homology can be derived as the direct limit of Hamiltonian Floer homologies of a sequence of Hamiltonian functions $\set{H_i}_{i\in \N}$. We understand the generators of Hamiltonian Floer homology of $H_i$ as the intersections of two Lagrangian submanifolds,
\[\text{the diagonal  } \Delta_W := \{(p,p) \in W \times W | p \in W\} \text{  and the graph  } \Gamma_{H_i}:= \left\{\left(p,\phi_{H_i}(p)\right) | p \in W\right\},\]
in the product space $W\times W$, where $\phi_{H_i}$ is the Hamiltonian diffeomorphism generated by the Hamiltonian function ${H_i}$.

Using the concept of {\em Lagrangian tomograph} from \cite{Cineli-Ginzburg-Gurel21, Cineli-Ginzburg-Gurel22a}, along with the {\em Crofton inequality}, we relate the number of not-so-short bars with the volume of the graph $\Gamma_{H_i}$. Finally, with Yomdin Theorem \cite{Yomdin}, we show the exponential growth rate of volume of the graphs of $H_i$ is at most the topological entropy. This finishes the proof of Theorem \ref{thm vs topological entropy}. For details, see Section \ref{section SH barcode entropy vs the topological entropy of the Reeb flow}.

\subsection{Further directions} 
In this subsection, we briefly describe possible future directions. 

One possible direction for further research is to give a lower bound of the $\SH$-barcode entropy as an analogue of Theorem B of \cite{Cineli-Ginzburg-Gurel21, Ginzburg-Gurel-Mazzucchelli22}. More precisely, the authors that the barcode entropy is at least equal to the topological entropy of the flow restricted to a compact invariant hyperbolic set for Hamiltonian diffeomorphisms and geodesic flows respectively. The proof relies crossing energy theorem \cite[Theorem 6.1]{Ginzburg-Gurel18} and \cite[Proposition 6.6]{Ginzburg-Gurel-Mazzucchelli22}. If one can prove the Reeb-version of the crossing energy theorem, it will be possible to show that the barcode entropy of a given Reeb flow is at least the topological entropy of the flow restricted to a compact invariant hyperbolic set.




Another possible direction is to study (semi-)continuity properties of the $\SH$-barcode entropy. For the topological entropy, Newhouse \cite{Newhouse89} proved that it is upper semi-continuous on spaces of smooth flows with respect to the $C^\infty$-topology. This result directly implies the upper semi-continuity of the topological entropy of Reeb flows with respect to the $C^\infty$-topology. However, for the barcode entropy, the upper semi-continuity is even unclear on the space of contact forms with respect to the $C^\infty$-topology.

Lastly, one can ask whether the $\SH$-barcode entropy is unbounded above on the space of contact forms with contact volume one. In \cite{Abbondandolo-Alves-Saglam-Schelnk21}, Abbondandolo, Alves, Sa{\u g}lam, and Schlenk proved that for any positive real number $r$, there exists a contact form such that the topological entropy of the corresponding Reeb flow is $r$ and the contact volume is one. Together with Theorem \ref{thm vs topological entropy}, the barcode entropy can be arbitrarily small even with the fixed volume. However, it is yet unclear whether barcode entropy can be also arbitrarily large.

\subsection{Acknowledgment}
We appreciate Viktor Ginzburg for explaining key ideas of \cite{Cineli-Ginzburg-Gurel21} in a seminar talk and a personal conversation. The second and third named authors thank Jungsoo Kang and Otto van Koert for suggestions and valuable comments on the first draft. third named author thanks Matthias Meiwes for helpful discussions and his encouragement.

During this work, the first and second named authors were supported by the Institute for Basic Science (IBS-R003-D1). The research of third named author was supported by
Samsung Science and Technology Foundation SSTF-BA1801-01 and
National Research Foundation of Korea grant NRF-2020R1A5A1016126 and RS-2023-00211186.

\section{Preliminary: Persistence module} 
\label{section persistent homology}
In this section, we review basic definitions/facts about persistence modules.
We will review something we need in the later part of the paper, but there are many references in which a reader can find more details.
For example, see \cite{Bauer-Lesnick15, Bubenik-Vergili18, Polterovich-Rosen-Samvelyan-Zhang, Usher-Zhang, Polterovich-Shelukhin-Stojisavljevic17}, etc. 

\subsection{Non-Archimedean normed vector space}
\label{subsection non-Archimedean normed vector space}

In Section \ref{subsection non-Archimedean normed vector space}, we focus on the notion of {\em non-Archimedean vector space} and then explain its basic notions/properties.
For detailed explanations, the reader can refer to \cite[Sections 2 and 3]{Usher-Zhang}.

Throughout Section \ref{section persistent homology}, we fix a field $\mathbb{k}$ and consider the vector spaces over $\mathbb{k}$. 
\begin{definition}
	\label{def non-Archimedean vector space} 
	\mbox{}
	\begin{enumerate}
		\item 	A {\bf valuation} $\nu$ on $\mathbb{k}$ is a function $\nu : \mathbb{k} \to \mathbb{R} \cup \{\infty\}$ such that 
		\begin{enumerate}
			\item[(V1)] $\nu(a) = \infty$ if and only if $a =0$, 
			\item[(V2)] for any $a, b \in \mathbb{k}$, $\nu(ab) = \nu(a) + \nu(b)$, and
			\item[(V3)] for any $a, b \in \mathbb{k}$, $\nu(a+b) \geq \min \{\nu(a), \nu(b)\}$.
		\end{enumerate}
		\item A valuation $\nu$ is {\bf trivial} if $\nu(a)=0$ for all nonzero $a \in \mathbb{k}$.
		\item Let $\mathbb{k}$ be a field equipped with a valuation $\nu$.
		A {\bf non-Archimedean normed vector space} over $\mathbb{k}$ is a pair $(V,\ell: V \to \mathbb{R}\cup\{-\infty\})$ of a $\mathbb{k}$-vector space $V$ and a function $\ell$ on $V$ such that
		\begin{enumerate}
			\item[(F1)] $\ell(x) = \infty$ if and only if $x =0$, 
			\item[(F2)] for any $a \in \mathbb{k}$ and $x \in V$, $\ell(ax) = \ell(x) - \nu(a)$, and 
			\item[(F3)] for any $x, y \in V$, $\ell(x+y) \leq \max \{\ell(x), \ell(y)\}$.
		\end{enumerate}
	\end{enumerate}
\end{definition}
\begin{remark}
	In the current paper, we assume that every field is equipped with the trivial valuation if we do not specify another valuation.
\end{remark}

Now, we define the notion of {\em orthogonality} on a non-Archimedean vector space.
\begin{definition}
	\label{def orthogonality}
	Let $(V, \ell)$ be a non-Archimedean vector space over $\mathbb{k}$. 
	\begin{enumerate}
		\item A finite ordered collection $\left(v_1, \dots, v_r\right)$ of elements of $V$ is said to be {\bf orthogonal} if, for all $a_1, \dots, a_r \in \mathbb{k}$, the following holds:
		\[\ell\left(\sum_{i=1}^r a_i v_i\right) = \max\left\{\ell(a_iv_i) | i = 1, \dots, r \right\}.\]
		\item An {\bf orthogonalizable space} $(V,\ell)$ is a finite-dimensional non-Archimedean normed vector space $(V,\ell)$ over $\mathbb{k}$ such that there exists an orthogonal basis for $V$.
	\end{enumerate}
\end{definition}

We will use Lemmas \ref{lem svd} and \ref{lem one to one map} in later sections. 
It could be a good idea to skip Lemmas \ref{lem svd},\ref{lem one to one map} and their proofs for the first read and come back when it is needed. 

\begin{lem}[Theorem 3.4 of \cite{Usher-Zhang}]
	\label{lem svd}
	Let $(C, \ell_C)$ and $(D, \ell_D)$ be orthogonalizable $\mathbb{k}$-spaces and let $A: C \to D$ be a linear map with rank $r$. 
	Then, $C$ (resp.\ $D$) admits an orthogonal basis $\{v_1, \dots, v_n\}$ (resp.\ $\{y_1, \dots, y_m\}$) such that 
	\begin{enumerate}
		\item $\{v_{r+1}, \dots, v_n\}$ is an orthogonal basis for $\ker A$, 
		\item $\{w_1, \dots, w_r\}$ is an orthogonal basis for $\mathrm{Im} A$, 
		\item $Av_i = w_i$ for all $i =1, \dots, r$, and
		\item $\ell_C(v_1) - \ell_D(w_1) \leq \dots \leq \ell_C(v_r) - \ell_D(w_r)$.
	\end{enumerate}
\end{lem}
\begin{proof}
	See the proof of \cite[Theorem 3.4]{Usher-Zhang}. We omit the proof.
\end{proof}

\begin{lem}
	\label{lem one to one map}
	Let $(V, \ell)$ be a finite-dimensional non-Archimedean normed vector space with two different orthogonal bases 
	\[\{v_1, \dots, v_n\} \text{  and  } \{w_1, \dots, w_n\}.\]
	If $\ell(v_1) \leq \ell(v_2) \leq \dots \leq \ell(v_n)$ and $\ell(w_1) \leq \ell(w_2) \leq \dots \leq \ell(w_n)$, then 
	\[\ell(v_i) = \ell(w_i),\]
	for all $i =1, \dots, n$. 
	In particular, there is a one-to-one map between two bases preserving the non-Archimedean norm. 
\end{lem}
\begin{proof}
	Let $V_{< r}$ be a subspace of $V$ defined as follows:
	\[V_{< r} := \{v \in V | \ell(v) < r\}.\]
	
	Since $V$ is finite dimensional, there is a unique finite sequence $(r_0 < r_1 < \dots < r_k)$ of real numbers such that 
	\begin{itemize}
		\item $0 = V_{< r_0} \subsetneq V_{< r_1} \subsetneq \dots \subsetneq V_{< r_k} =V$, and 
		\item for any $r \in [r_i, r_{i+1})$, $V_{< r_i} = V_{< r}$.
	\end{itemize}
	We note that the sequence $(r_0, \dots, r_k)$ is an invariant of $(V, \ell)$, not depending on the choice of bases. 
	
	Since the sequence determines $\ell(v_i)$ and $\ell(w_i)$, Lemma \ref{lem one to one map} holds. 
\end{proof}

\subsection{Persistence module}
\label{subsection persistence module}
Let us assume that we have a finite-dimensional, orthogonalizable non-Archimedean normed vector space $(V, \ell)$ equipped with a linear map $\partial: V \to V$ satisfying that
\begin{itemize}
	\item $\partial \circ \partial =0$, and 
	\item for all $v \in V$, $\ell(\partial v) \leq \ell(v)$.  
\end{itemize}
The first condition means that $\partial$ can be seen as a differential map. 
In other words, the {\em homology} vector space of $(V, \partial)$, denoted by $H(V,\partial)$, is defined as 
\[H(V,\partial) := \ker(\partial) / \im(\partial).\] 
Moreover, thanks to the non-Archimedean vector space structure on $V$ and the second condition on $\partial$, $H(V,\partial)$ admits an interesting structure.  
The interesting structure, called the {\em persistence module} structure, is defined as follows: 

\begin{definition}
	\label{def persistence module} 
	A {\bf persistence module} is a pair $(V,\pi)$, where $V$ is a collection of vector spaces $\{V_t\}_{t \in \mathbb{R}}$ and $\pi$ is a collection of linear maps $\{\pi_{s,t}:V_s \to V_t\}_{s \leq t \in \R}$ satisfying the following conditions:
	\begin{enumerate}
		\item For any $s \leq t \leq r$, one has $\pi_{s,r} = \pi_{t,r} \circ \pi_{s,t}$, and 
		\item for all $ t \in \mathbb{R}$, $\pi_{t,t}:V_t \to V_t$ is the identity map.
	\end{enumerate}
\end{definition}

In the main body of the paper, we will see symplectic homology and Hamiltonian Floer homology as persistence modules, and we will investigate their properties. 
In order to prepare the main part, we introduce definitions and properties related to the notion of persistence modules briefly in the rest of Section \ref{section persistent homology}.
We refer the reader to \cite{Bubenik-Vergili18, Polterovich-Rosen-Samvelyan-Zhang} for more information.

First, we give examples of persistence modules.
\begin{example}
	\label{example persistence modules}
	\mbox{}
	\begin{enumerate}
		\item The first example is the example mentioned at the beginning of this subsection. 
		Let us assume that the given triple $(V, \ell, \partial)$ satisfies the above conditions. 
		Let $V^{<a}$ be the subspace of $V$ consisting of all vectors $v$ such that $\ell(v)<a$. 
		Then, the second condition on $\partial$ implies that the restriction of $\partial$ on $V^{<a}$ can be seen as a linear map from $V^{<a}$ to itself. 
		Thus, the following $H^{<a}$ is well-defined:
		\[H^{<a} := \ker(\partial|_{V^{<a}}) / \im(\partial|_{V^{<a}}).\]
		Moreover, for any $a \leq b$, there exists an inclusion map $\pi_{a,b}:V^{<a} \to V^{<b}$. 
		It is easy to observe that $\pi_{a,b}$ is a chain map and induces a linear map from $H^{<a}$ to $H^{<b}$.
		Let $\pi_{a,b}$ denote the induced map on $H^{a}$ again.  
		Now, the following is a persistence module:
		\[H(V):= \left(\{H^{<a}\}_{a \in \R}, \{\pi_{a,b}\}_{a \leq b \in \R}\right).\]
		\item The simplest persistence modules are the interval modules. 
		Every non-empty interval $I \subset \R$ defines a persistence module $\mathbb{k}I$ as follows: 
		\[\mathbb{k}I_t := \begin{cases}
			\mathbb{k} \text{  if  } t \in I, \\
			0 \text{  otherwise}.
		\end{cases}
		\pi_{s,t} := \begin{cases}
			id_\mathbb{k} \text{  if  } s, t \in I, \\
			0 \text{  otherwise}.
		\end{cases}\]
		\item Let $\{I_i | i \in \text{  an index set  } J\}$ be a collection of intervals $I_i \subset \R$.
		Let $\{\left(\pi_i\right)_{s,t}\}_{s \leq t}$ be the collection of linear maps for the interval persistence module $\mathbb{k}I_i$ for all $i \in J$. 
		Then, this collection defines a persistence module $\bigoplus_{i \in J} \mathbb{k}I_i$ as follows: 
		\[\left(\bigoplus_{i\in J} \mathbb{k}I_i\right)_t := \bigoplus_{i\in J} \left(\mathbb{k}I_i\right)_t, \pi_{s,t} := \bigoplus_{i\in J} \left(\pi_i\right)_{s,t}. \]		  
	\end{enumerate}
\end{example} 

Before going further, we define a special class of persistence modules.
\begin{definition}
	\label{def pfd persistence module}
	A {\bf point-wise finite dimensional} persistence module, or simply {\bf pfd} persistence module, is a persistence module $(V, \pi)$ such that, for any $t \in \mathbb{R}$, $V_t$ is of finite dimension,  
\end{definition}

We note that pfd persistence modules are special because of the {\em Structure Theorem}.
The Structure Theorem proves that every pfd persistence module is equivalent to a direct sum of interval modules.
We state the Structure Theorem in Theorem \ref{thm Structure Theorem}.
Before stating it, we first define the notion of equivalence between persistence modules. 

\begin{definition}
	\label{def equivalece}
	Let $(V, \pi)$ and $(V', \pi')$ be two persistence modules.
	\begin{enumerate}
		\item A {\bf morphism} $A: (V, \pi) \to (V', \pi')$ is a family of linear maps 
		\[\{A_t: V_t \to V'_t\}_{t \in \mathbb{R}},\]
		such that the following diagram commutes for all $s \leq t$:
		\[\begin{tikzcd}
			V_s \ar[r, "\pi_{s,t}"] \ar[d, "A_s"'] & V_t \ar[d, "A_t"] \\
			V'_s \ar[r, "\pi'_{s,t}"] & V'_t 
		\end{tikzcd}\]
		\item A morphism $A: (V, \pi) \to (V', \pi')$ is an {\bf isomorphism} if there is a morphism $B: (V', \pi') \to (V, \pi)$ such that $A \circ B$ and $B \circ A$ are the identity morphisms on the corresponding persistence module.
		\item Two persistence modules $(V, \pi)$ and $(V', \pi')$ are {\bf isomorphic} if there is an isomorphism $A: (V, \pi) \to (V', \pi')$. 
	\end{enumerate}
\end{definition}
We note that the composition of morphisms and the identity morphism are naturally defined, thus we used them in Definition \ref{def equivalece} without defining them. 

Now, we can state the Structure Theorem. 
The Structure Theorem is proven by Gabriel \cite{Gabriel72} for a special case and by Crawley-Boevey \cite{Crawley-Boevey15} for the general case. 
See \cite[Theorem 2.7]{Bubenik-Vergili18} for more details. 

\begin{thm}[Structure Theorem]
		\label{thm Structure Theorem}
		Let $(V, \pi)$ be a persistence module.
		If $(V,\pi)$ is a pfd persistence module, i.e., $V_t$ is finite dimensional for each $t \in \mathbb{R}$, then $(V,\pi)$ is isomorphic to a direct sum of interval modules. 	
\end{thm}

Thanks to Theorem \ref{thm Structure Theorem}, we can define the following:
\begin{definition}
	\label{def finite type}
	A pfd persistence module $(V,\pi)$ is 
 \begin{itemize}
     \item  a persistence module of {\bf finite type} if $(V,\pi)$ is equivalent to a direct sum of finitely many interval modules. Equivalently, for all but a finite number of points $t\in \mathbb{R}$ there exists a neighborhood $U$ of $t$ such that $\pi_{s,r}$ is an isomorphism for $s<r$ in $U$.
    \item a persistence module of {\bf locally finite type} if $(V,\pi)$ is equivalent to a direct sum of countably many interval modules and the real endpoints of the intervals form a closed discrete subset of $\R$. Equivalently, for all but a close discrete subset of $\R$, $t\in \mathbb{R}$ there exists a neighborhood $U$ of $t$ such that $\pi_{s,r}$ is an isomorphism for $s<r$ in $U$.
\end{itemize}
\end{definition}

The special classes of persistence modules, i.e., pfd persistence modules and persistence module of (locally) finite types, play an important role in this paper. 
The following remark describes the role roughly:

\begin{remark}
	\label{rmk summary of barcode entropy}
	As mentioned before, we will see a symplectic homology and a Hamiltonian Floer homology as persistence modules. 
	From symplectic geometry, one can show that these homologies are persistence modules of (locally) finite type. This enables us to count the number of intervals satisfying some conditions. One of our main characters, the {\em barcode entropy}, is defined as the exponential growth rate of the number of {\em not-too-short} intervals of the Floer homologies, when a sequence of Hamiltonian Floer homologies is given. 
	The formal definition will be given in Section \ref{section two barcode entropy}.
\end{remark}

We note that there are many other classes of persistence modules. 
See \cite[Section 3.1]{Bubenik-Vergili18}.
We introduce two other classes that we will use in Lemma \ref{lem complete space}.

\begin{definition}
	\label{def cid qtame}
	\mbox{}
	\begin{enumerate}
		\item A {\bf countably interval decomposable}, or simply {\bf cid} persistence module, is a persistence module equivalent to a direct sum of countably many interval modules. 
		We let $\boldsymbol{(cid)}$ denote the set of cid persistence modules. 
		\item A {\bf q-tame} persistence module $(V, \pi)$ is a persistence module such that $\pi_{a,b}$ has a finite rank for any $a < b$. 
		We let $\boldsymbol{(qtame)}$ denote the set of q-tame persistence modules. 
	\end{enumerate}
\end{definition}

We end Section \ref{subsection persistence module} by defining two operations generating a persistence module from another persistence module. 

\begin{definition}
	\label{def two operations}
	\mbox{}
	\begin{enumerate}
		\item Let $f: \mathbb{R} \to \mathbb{R}$ be a strictly increasing function. 
		Let $(V,\pi)$ be a persistence module.
		The {\bf action reparametrization of $\boldsymbol{V}$ with respect to $\boldsymbol{f}$}, denoted by $\boldsymbol{\mathrm{act}(V,f)}$ is a persistence module defined as follows:
		\[\act(V,f) := \left(\{\act(V,f)_t := V_{f(t)}\}_{t \in \R}, \{f(\pi)_{a,b} = \pi_{f(a),f(b)}\}_{a\leq b \in R}\right).\]
		\item Let $V_1$ and $V_2$ be two barcodes. 
		We say {\bf $\boldsymbol{V_2}$ agrees with $\boldsymbol{V_1}$ after (action) reparametrization} if there exists a strictly increasing function $f: \mathbb{R} \to \mathbb{R}$ such that 
		\[V_1 = \act(V_2,f).\]
		\item Let $(V,\pi)$ be a persistence module, and let $T$ be a real number. 
		Then, the {\bf truncation of $\boldsymbol{V}$ at $\boldsymbol{T}$}, denoted by $\boldsymbol{\mathrm{tru}(V,T)}$, is a persistence module 
		\[\tru(V,T) := \left(\{\tru(V,T)_t\}_{t \in \R}, \{\tru(\pi,T)_{a,b}\}_{a \leq b \in \R} \right),\] 
		defined as follows:
		\begin{align*}
			\tru(V,T)_t = \begin{cases}
				V_t \text{  if  } t <T \\ 0 \text{  otherwise}
			\end{cases}, 
			\tru(\pi,T)_{a,b} = \begin{cases}
				\pi_{a,b} \text{  if  } b <T \\ 0 \text{  otherwise}
			\end{cases}.
		\end{align*}
	\end{enumerate} 
\end{definition}

\subsection{Interleaving distance}
\label{subsection interleaving distance}
In Section \ref{subsection persistence module}, we introduced the definition of persistence modules and related notions. 
In the current subsection, we discuss the distance between two persistence modules.

In the literature, there exists a distance function which is called {\em interleaving distance}. 
In order to define the interleaving distance, we need the following definition:
\begin{definition}
	\label{def shift}
	Let $(V, \pi)$ be a persistence module and let $\delta \in \mathbb{R}$. 
	\begin{enumerate}
		\item We define {\bf $\boldsymbol{\delta}$-shift} of $(V, \pi)$ as the persistence module $(V[\delta], \pi[\delta])$ defined as follows:
		\[\left(V[\delta]\right)_t := V_{t+\delta}, \left(\pi[\delta]\right)_{s,t} = \pi_{s+\delta, t+\delta}.\]
		\item For $\delta>0$, we define {\bf the $\boldsymbol{\delta}$-shift morphism} as the following morphism: 
		\[\Phi^\delta_{(V,\pi)}: (V,\pi) \to (V[\delta], \pi[\delta]), \left(\Phi^\delta_{(V,\pi)}\right)_t = \pi_{t,t+\delta}.\]
		\item Let $F: (V,\pi) \to (V',\pi')$ be a morphism between two persistence modules. 
		Then, for $\delta >0$, let us denote the morphism commuting the following diagram by $\boldsymbol{F[\delta]}$:
		\[\begin{tikzcd}
			(V, \pi) \ar[r, "F"] \ar[d, "\Phi^\delta_{(V,\pi)}"'] & (V', \pi') \ar[d, "\Phi^\delta_{(V',\pi')}"]  \\
			(V[\delta], \pi[\delta]) \ar[r, "F\lbrack\delta\rbrack"] & (V'[\delta], \pi'[\delta])
		\end{tikzcd}\]
	\end{enumerate}
\end{definition} 

By using the $\delta$-shifts of persistence modules, we define a distance between two persistence modules in the following way:
\begin{definition}
	\label{def interleaving distance}
	\mbox{}
	\begin{enumerate}
		\item Let $\delta>0$. 
		Two persistence modules $(V,\pi)$ and $(V', \pi')$ are {\bf $\boldsymbol{\delta}$-interleaved} if there exist two morphisms $F: (V,\pi) \to (V'[\delta], \pi'[\delta])$ and $G: (V',\pi') \to (V[\delta], \pi[\delta])$ such that the following diagrams commute:
		\[\begin{tikzcd}
			(V, \pi) \ar[r,"F"] \ar[rr, bend right=20, "\Phi^{2\delta}_{(V,\pi)}"'] & (V'[\delta], \pi'[\delta]) \ar[r, "G\lbrack \delta \rbrack"] & (V[2\delta], \pi[2\delta])
		\end{tikzcd} 
		\begin{tikzcd}
			(V', \pi') \ar[r,"G"] \ar[rr, bend right=20, "\Phi^{2\delta}_{(V',\pi')}"'] & (V[\delta], \pi[\delta]) \ar[r, "F\lbrack \delta \rbrack"] & (V'[2\delta], \pi'[2\delta])
		\end{tikzcd}\]
		\item The {\bf interleaving distance} between two persistence modules $(V,\pi)$ and $(V', \pi')$ is defined as follows: 
		\[d_{int}\left((V, \pi), (V', \pi')\right) := \inf \left\{ \delta >0 | (V,\pi) \text{  and  } (V', \pi') \text{  are $\delta$-interleaved.}\right\}.\]
	\end{enumerate}
\end{definition} 

We note that the interleaving distance defines a {\em pseudometric} structure on the set of persistence modules, not a metric structure.
In other words, one can find two persistence modules $V$ and $W$ such that $V$ and $W$ are not isomorphic, but $d_{int}(V, W)=0$.
The simplest example is a pair of interval modules $\mathbb{k}[0,1]$ and $\mathbb{k}(0,1)$. 
Then, the interleaving distance between them is zero, but they are not isomorphic.

To obtain a metric structure, we consider the set of equivalence classes of the equivalence relation $V \sim W$ if $d_{int}(V, W)=0$. Then, the interleaving distance function gives a metric structure on the set of (equivalence classes of) persistence modules. 
For more details about the (pseudo)metric structure induced from the interleaving distance, we refer the reader to \cite[Section 2.4]{Bubenik-Vergili18}.

For equivalence classes of cid persistence module, there exists a unique element satisfying the lower semi-continuity property: for every $t\in \R$,
\[ V_t = \varinjlim_{s<t} V_s,\]
where the direct system is given by the persistence map $\pi_{s,r}$ for real numbers $s < r < t$. This is a unique element in the equivalence class that is equivalent to a direct sum of intervals modules in forms of $(a,b], (-\infty, b],$ and $(a, \infty)$ for some real numbers $a <b$. We let {\textbf{(lower)}} denote the set of persistence modules with lower semi-continuity properties.

We note that every Persistence module constructed from Floer theories will be countably interval decomposable q-tame persistence module with lower semi-continuity properties. For this class of persistence modules, interleaving distance gives a metric structure. Thus, we can consider a {\em Cauchy sequence} of persistence modules. Lemma \ref{lem complete space} cares about Cauchy sequences of persistence modules. 
\begin{lem}[Theorem 4.25 of \cite{Bubenik-Vergili18}]
	\label{lem complete space}
	Every Cauchy sequence in $(cid) \cap (qtame) \cap (lower)$ has a unique limit point. 
	In other words, the space $(cid) \cap (qtame) \cap (lower)$ is complete. 
\end{lem}
\begin{proof}
	See \cite[Theorem 4.25]{Bubenik-Vergili18} and its proof. 
\end{proof}

\subsection{Barcodes and the bottleneck distance}
\label{subsection barcodes and the bottleneck distance}
We recall that by the Structure Theorem (Theorem \ref{thm Structure Theorem}), every pfd persistence module is isomorphic to a direct sum of interval modules.
Since we consider an equivalence class of persistence module, not a persistence module itself, every pfd persistence module type is isomorphic to a {\em barcode} defined as follows:
\begin{definition}
	\label{def barcodes}
	\mbox{} 
	\begin{enumerate}
			\item A {\bf barcode} $\boldsymbol{B}$ is a set of intervals in $\R$, i.e., it is a collection 
			\[B = \left\{\text{  an interval  } I_i \subset \R | i \in \text{  an index set  } J\right\}.\] 
			\item The intervals in a barcode are called {\bf bars} of the barcode.
		\end{enumerate}
\end{definition}

For convenience, we identify a barcode $B = \{I_i | i \in J\}$ with a persistence module 
\[\bigoplus_{i \in J} \mathbb{k}I_i.\]

The purpose of introducing barcodes is to define another distance function on the set of pfd persistence modules.
More precisely, we define a {\em bottleneck} distance on the set of barcodes. 
In the later part of this subsection, we will state the Isometry Theorem, claiming that the interleaving distance and the bottleneck distance are the same.
However, even though the two distance functions are the same, they provide different ways of measuring distances between persistence modules, and we will utilize both ways of measuring in the paper.  

To define the bottleneck distance, we need the following preparations:
\begin{itemize}
	\item Given an interval $I$ and a positive number $\delta$, let $I^{-\delta}$ denote the interval obtained from $I$ by expanding by $\delta$, for example, if $I=[a,b]$, then
	\[I^{-\delta}= [a-\delta, b+\delta].\]
	\item For a barcode $B$ and a positive number $\epsilon$, let $B_{\epsilon}$ the set of all bars from $B$ of length greater than $\epsilon$. 
	\item A {\em matching} between two barcodes $B_1=\left\{I_i | i \in J_1\right\}, B_2=\left\{I'_j | j \in J_2\right\}$ is a bijection $\mu: \bigsqcup_{i \in J_1'} I_i \to \bigsqcup_{j \in J_2'} I'_j$, where $J_1' \subset J_1, J_2'\subset J_2$, and where $\bigsqcup$ means a disjoint union.
	In this case, $\bigsqcup_{i \in J_1'} I_i$, or equivalently a barcode $B_1':= \{I_i | i \in J_1'\}$, is the coimage of $\mu$, and $\bigsqcup_{j \in J_2'} I'_j$, or equivalently a barcode $B_2':=\{I_j' | j \in J_2'\}$ is the image of $\mu$. 
\end{itemize}

\begin{definition}
	\label{def bottleneck}
	\mbox{}
	\begin{enumerate}
		\item A {\bf $\boldsymbol{\delta}$-matching} between two countable barcodes $B$ and $C$ is a matching $\mu:B \to C$ such that 
		\begin{itemize}
			\item the coimage of $\mu$ contains $B_{2\delta}$, 
			\item the image of $\mu$ contains $C_{2\delta}$, and 
			\item if a bar $I$ in $B$ and a bar $J$ in $C$ satisfy that $\mu(I) = J$, then $I \subset J^{-\delta}, J \subset I^{-\delta}$.  
		\end{itemize}
		\item The {\bf bottleneck distance} between two countable, proper barcodes $B$ and $C$, denoted by $\boldsymbol{b_{bot}(B,C)}$, is defined to be the infinmum over all $\delta$ such that ther is a $\delta$-matching between $B$ and $C$. 
	\end{enumerate}
\end{definition}

Now, we state the Isometry Theorem, claiming that the interleaving and bottleneck distances are the same for finite barcodes. 

\begin{thm}[Isometry Theorem]
	\label{thm isometry theorm}
	Let $V$ and $W$ be two pfd persistence modules, and let $B(V)$ and $B(W)$ be barcodes such that $V$ and $B(V)$ (resp.\ $W$ and $B(W)$) are isomorphic to each other as persistence modules. 
	Then, 
	\[d_{int}(V,W) = d_{bot}\left(B(V), B(W)\right).\]
\end{thm}
\begin{proof}
	The Isometry Theorem is given in \cite{Bauer-Lesnick15}.
	See \cite[Section 3 and Proof of Theorem 3.5]{Bauer-Lesnick15} and references therein.
\end{proof}

\section{Two barcode entropy}
\label{section two barcode entropy}
The main idea of the paper is to apply the theory of persistence modules to symplectic homologies. 
In other words, we would like to construct persistence modules from symplectic homologies, and would like to investigate the constructed persistence modules.

In Section \ref{section two barcode entropy}, we focus on the part of constructing persistence modules. 
We start the section by briefly reviewing symplectic homologies.
And we also explain what persistence modules are related to symplectic homologies.
Finally, we can define the main characters of the paper, which are called {\em barcode entropy}.

\subsection{Setting and notations}
\label{subsection setting and notations}
We set the situation we are interested in and notations on the situation in this subsection. 

Let $W$ be a Liouville domain.
In other words, $W$ is a manifold with boundary, which is equipped with a one-form $\lambda$ such that $d\lambda$ is a symplectic two-form on $W$. 
Then, the boundary $\partial W$ is a contact manifold such that the restriction of $\lambda$ on $\partial W$ is a contact one-form. 
For simplicity, we set $\alpha:= \lambda|_{\partial W}$. 
We also assume that $\alpha$ is a {\em non-degenerate} contact one-form. For the later use, let $R_\alpha$ denote the Reeb vector field corresponding to $\alpha$. Throughout the paper, we further assume the first Chern class of $W$ vanishes.

Let $\widehat{W}$ be the completion of the Liouville domain $W$.
It means that 
\begin{itemize}
	\item as a set $\widehat{W} = W \cup \left( \partial W \times [1,\infty)_r\right)$ (the subscription $r$ means $[1,\infty)$-factor is coordinated by $r$) where $\partial W$ and $\partial W \times \{1\}$ are identified by the trivial identification, 
	\item the Liouville one-form $\lambda$ extends to a Liouville one-form on $\widehat{W}$, denoted by $\lambda$ again for convenience, and
	\item  on $\partial W \times [1,\infty)$, $\lambda = r \alpha$. 
\end{itemize}

For the later use, we set notations in the rest of Section \ref{subsection setting and notations}.

Let $t>1$ be a fixed real number. 
Then, $\boldsymbol{W_t}$ is a compact subset of $\widehat{W}$ defined as  
\[W_t := W \cup \left(\partial W \times [1,t]\right).\]

In order to define the symplectic homology of $\widehat{W}$, we need to consider smooth $1$-periodic families of Hamiltonian functions. 
In other words, we consider a (time-dependent) Hamiltonian function $H: S^1 \times \widehat{W} \to \mathbb{R}$, where $S^1$ is identified with $\mathbb{R}/\mathbb{Z}$. 
Since $\widehat{W}$ is a symplectic manifold, there is a Hamiltonian vector field for $H$.
Let $\boldsymbol{X_H}$ denote the Hamiltonian vector field, i.e., 
\[d \lambda\left(X_H, Y\right) = -dH(Y) \text{  for all vector field  } Y \text{  on  } \widehat{W}.\]
We also set a notation for the time $t$-map of $X_H$.
Let $\boldsymbol{\phi^t_H}$ denote the time $t$-map of $X_H$.
Sometimes, we let $\boldsymbol{\phi_H}$ denote the time $1$-map of $X_H$, i.e., $\phi^1_H$, for convenience.

For technical reasons, we will consider Hamiltonian functions with some nice properties defined as follows: 
\begin{definition}
	\label{def linear at infinity}
	Let $H$ be a (time-dependent) Hamiltonian function on $\widehat{W}$.
	\begin{enumerate}
		\item We say that $H$ is {\bf linear at infinity} if there exist two real numbers $r_0 >1$ and $T$ such that 
		\begin{itemize}
			\item $T$ is not a period of any closed Reeb orbits of $\left(\partial W, \alpha\right)$, and
			\item for any $(t, r, x) \in S^1 \times \left(\partial W \times [r_0, \infty) \right)$, $H(t, r, x) = Tr - C$ for some constant $C \in \mathbb{R}$.  
		\end{itemize}
		The constant $T$ satisfying the above two conditions is called {\bf slope of $\boldsymbol{H}$}.
		\item We say that $H$ is {\bf non-degenerate} if every fixed point of $\phi_H$ does not admit $1$ as an eigenvalue for its linearized Poincare return map.
	\end{enumerate}
\end{definition}

\subsection{Hamiltonian Floer homology and the corresponding persistence module}
\label{subsection Hamiltonian Floer homology and the corresponding persistence module}
We start the main part of Section \ref{section two barcode entropy} by reviewing the notion of Hamiltonian Floer homology on the Liouville domain. For more details, we refer the reader to \cite{Salamon99, Seidel08, Audin-Damian14}.

Let $H: S^1 \times \widehat{W} \to \R$ be a time-dependent Hamiltonian such that $H$ is non-degenerate and linear at infinity.
We denote by $\mathcal{P}_1(H)$ the set of 1-periodic orbits of the Hamiltonian flow of $H$. 
Since $H$ is linear at infinity, 1-periodic orbits or equivalently fixed points of $\phi_H$ are points of $W$. By non-degeneracy of $H$, the set of fixed points is isolated. We conclude that $\mathcal{P}_1(H)$ is a finite set. Then, one can define a finite-dimensional vector space and an action functional on the vector space from $H$ as follows:

\begin{definition}
	\label{def action functional on CF(H)}
	\mbox{}
\begin{enumerate}
	\item Let $\boldsymbol{\CF(H)}$ denote the vector space defined as
	\[\CF(H) := \bigoplus_{\gamma \in \mathcal{P}_1(H)} \mathbb{Z}_2\left<\gamma\right>.\] 
	\item The {\bf action functional on $\boldsymbol{\CF(H)}$}, denoted by $\boldsymbol{\cA_H}$, is defined as follows: 
	\begin{itemize}
		\item for any $\gamma \in \mathcal{P}(H)$, 
		\[\cA_H(\gamma) = \int_\gamma \lambda - \int_0^1H\left(t,\gamma(t)\right)dt,\]
		\item for any $a_i \in \mathbb{Z}_2$ and $\gamma_i \in \mathcal{P}(H)$ such that $\gamma_i \neq \gamma_j$ for all $i \neq j$, 
		\[\cA_H\left(\sum_i a_i \gamma_i\right) := \max_i \{\cA_H(\gamma_i)| a_i \neq 0\}.\]
	\end{itemize}
	\item For any $a \in \mathbb{R}$ we set a subspace $\CF^{<a}(H) \subset \CF(H)$ as
    \[\CF^{<a}(H) = \{x \in \CF(H) | \cA_H(x) < a\}.\]
\end{enumerate}
\end{definition}

\begin{remark}
	\label{rmk non-Archimedean normed vector space}
	By the second bullet of (2), $(\CF(H), \cA_H)$ is an orthogonalizable non-Archimedean normed vector space.
\end{remark}

Now, we introduce a differential map $\partial: \CF(H) \to \CF(H)$, i.e., a linear map such that $\partial \circ \partial =0$.
For introducing that, we need to discuss a choice of almost complex structure. 

Let $J$ be a smooth $S^1$-family of $\left(d\lambda\right)$-compatible almost complex structure. 
We further assume that $J$ is {\em SFT-like on $r\ge r_0$} for some constant $r_0 > 1$.
In other words, $J$ satisfies following conditions on $r \ge r_0$:
\begin{itemize}
	\label{SFT-like}
	\item $J(\xi)=\xi$ for the contact structure $\xi=\text{ker}(\alpha)\subset T\left(\partial W \times \{r\}\right)$.
	\item $J^*\hat\lambda=dr$, i.e. $J(\partial_r)=R$.
	\item $J$ is invariant under the natural action by $\mathbb{R}$-translation.
\end{itemize}
\noindent We denote $\mathcal{J}_{r_0}^{\text{SFT}}$ the set of $(d\lambda)$-compatible almost complex structures SFT-like in $r\ge r_0$.

For $x,y \in \mathcal{P}_1(H)$, we denote by $\mathcal{M}(x,y;H,J)$ the moduli space of Floer cylinders from $x$ to $y$, or equivalently, the moduli space of smooth solutions $u:\mathbb{R}\times S^1\to \widehat{W}$ of the following equation:
\[\partial_s u + J(t,u)(\partial_t u + X_H(t,u)) =0, \lim_{s\to -\infty}u(s,\cdot)=x, \lim_{s\to +\infty}u(s,\cdot)=y,\]
where $(s,t)\in \mathbb{R}\times S^1$.
We recall that $X_H(t,p)$ is defined to be the Hamiltonian vector field of $H_t$ at point $p$ in Section \ref{subsection setting and notations}. In our convention, the Floer equation is the gradient flow line for the Hamiltonian action functional $\mathcal{A}_H$.

We note that if $x \neq y$, $\mathcal{M}(x,y;H,J)$ admits a free $\mathbb{R}$-action.
The $\mathbb{R}$-action is defined by the $\R$-translation on $s$-coordinate. 
Let us consider the quotient set $\mathcal{M}(x,y;H,J)/\mathbb{R}$.

For $x, y \in \mathcal{P}_1(H)$, we set $n(x,y) \in \mathbb{Z}_2$ as follows:
If $x \neq y$, and if the quotient set $\mathcal{M}(x,y;H,J)/\mathbb{R}$ is a finite set, then $n(x,y)$ is defined to be the parity of the number of elements in $\mathcal{M}(x,y;H,J)/\mathbb{R}$.
If either of the above conditions does not hold, then $n(x,y)=0$.
Then, one can define a linear map $\partial=\partial_{H,J}: \CF(H)\to\CF(H)$ satisfying that
\[\partial y := \sum_{x \in \mathcal{P}_1(H;\beta)} n(x,y)x \text{  for all  } y \in \mathcal{P}(H).\]

For a generic choice of $J\in \mathcal{J}_{r\ge r_0}^{\text{SFT}}$, $\partial \circ \partial = 0$, for more explanation see \cite[Section 3.2]{Salamon99}. We'll say such $J$ is $H-$regular, or simply \textbf{regular}. Thus, for such $J$, the pair $(\CF(H), \partial_{H,J})$ forms a chain complex. 
For the chain complex and its homology, we set notations: 
\begin{definition}
	\label{def CH(H,J)}
	\mbox{}
	\begin{enumerate}
		  \item A pair $(H,J)$ is called a {\bf Floer data} for a Liouville manifold    $(\widehat{W}, \lambda)$ if the following hold: 
		  \begin{itemize}
			\item $H: S^1 \times \widehat{W} \to \R$ is a non-degenerate, (time-dependent) Hamiltonian function linear at infinity.
			\item $J$ is an almost complex structure in $\mathcal{J}_{r_0}^{\text{SFT}}$ for some $r_0 > 1$.
		\end{itemize}
	    \item For a Floer data $(H,J)$, the pair $(\CF(H), \partial_{H,J})$ is simply denoted by $\boldsymbol{\CF(H,J)}$.
		\item The homology of $\,\CF(H,J)$ is denoted by $\boldsymbol{\HF(H,J)}$, i.e.,
		\[\HF(H,J):= \ker \partial_{H,J} / \im \partial_{H,J}.\] 
	\end{enumerate}
\end{definition}

Since $0 \le E(u)=\int \norm{\partial_s u}^2=\mathcal{A}_H(x)-\mathcal{A}_H(y)$ for $u\in \mathcal{M}(x,y;H,J)$, Floer cylinder from $x$ to $y$ cannot exist if $\cA_H(x) < \cA_H(y)$. It means that the restriction of $\partial$ onto the subspace $\CF^{<a}(H)$ gives a linear automorphism of $\CF^{<a}(H)$ for any $a \in \mathbb{R}$.
Thus, the following definition makes sense:
\begin{definition}
	\label{def HF<a(H,J)}
	For any $a \in \mathbb{R}$, $\boldsymbol{\HF^{<a}(H,J)}$ is defined to be the homology of $(\CF^{<a}(H), \partial|_{\CF^{<a}(H)})$, i.e., 
		\[\HF^{<a}(H,J):= \ker\left(\partial|_{\CF^{<a}(H)}:\CF^{<a}(H) \to \CF^{<a}(H)\right) \Big/ \im\left(\partial|_{\CF^{<a}(H)}:\CF^{<a}(H) \to \CF^{<a}(H)\right).\]
\end{definition}

Now, we would like to discuss the fact that $\HF(H,J)$ and $\HF^{<a}(H,J)$ are independent of the choice of $J$.
More precisely, for two pairs $(H, J_1)$ and $(H, J_2)$ satisfying the above conditions, we will show that there is a linear isomorphism between $\HF(H,J_1)$ and $\HF(H,J_2)$. (resp.\ $\HF^{<a}(H,J_1)$ and $\HF^{<a}(H,J_2)$ for any $a \in \mathbb{R}$.)

To show that, let us assume that we have two pairs $(H^\pm, J^\pm)$ satisfying the above conditions so that $\HF^{<a}(H^\pm,J^\pm)$ is defined for all $ a \in \mathbb{R}$. 
Moreover, we assume that $H^+(t, p)\le H^-(t,p)$ for all $t \in S^1$ and $p\in \widehat{W}$. 
Because of the assumption, one can choose a smooth homotopy of Hamiltonians $\{H_s\}_{s \in \mathbb{R}}$ such that 
\[H_s = H^- \text{  for  }\forall s\le -1,\, H_s=H^+ \text{  for  }\forall s\ge 1, \text{  } \partial_s H_s \le 0 \text{  for  } \forall s \in \mathbb{R}.\]

Similarly, let $\set{J_s}_{s\in\R}$ be a smooth homotopy of time-dependent almost complex structures such that 
\[J_s = J^- \text{  for  }\forall s\le -1, J_s=J^+, \text{  for  }\forall s\ge 1, \text{  for  } \forall s \in \mathbb{R}, J_s \text{  is SFT-like on  } r \geq r_s \text{  for some  } r_s>1.\]
The smooth family $\{J_s\}_{s\in\mathbb{R}}$ exists because the space of $\mathcal{J}^{\text{SFT}}_{r\ge r_0}$ is contractible.

Let $x^\pm \in \mathcal{P}_1(H^\pm)$.
We will consider the moduli space of Floer cylinders from $x^-$ to $x^+$, or equivalently, the moduli space of smooth solutions $u:\R \times S^1 \to \widehat{W}$ satisfying
\begin{equation}
\label{equation: continuation cylinder}
\partial_s u + J_s(t,u)(\partial_t u - X_{H_s}(t,u)) =0, \lim_{s\to -\infty}u(s,\cdot)=x^-, \lim_{s\to +\infty}u(s,\cdot)=x^+.
\end{equation}
Let us denote the moduli space by $\mathcal{M}(x^-,x^+;\{H_s\},\{J_s\})$.

Now, we can define $m(x^-,x^+)$ as follows:
If $\mathcal{M}(x^-,x^+;\{H_s\},\{J_s\})$ is a finite set, then $m(x^-,x^+)$ is defined to be the parity of the number of elements in $\mathcal{M}(x^-,x^+;\{H_s\},\{J_s\})$. 
If $\mathcal{M}(x^-,x^+;\{H_s\},\{J_s\})$ is not a finite set, then $m(x^-,x^+)$ is defined to be zero. 

Then the continuation homomorphism $\Phi_{(H^+,J^+),(H^-,J^-)}:\CF(H^+,J^+)\to \CF(H^-,J^-)$ is defined to be the linear map satisfying
\[\Phi_{(H^+,J^+),(H-+,J^-)}(x^+) := \sum_{x^-\in \mathcal{P}_1(H^-)} m(x^-,x^+)x^-.\]
Again for generic choice of $\{J_s\}$, $\Phi_{(H^+,J^+),(H-+,J^-)}$ is a chain map. Thus, $\Phi_{(H^+,J^+),(H-+,J^-)}$ induces a linear map from $\HF(H^+,J^+)$ to $\HF(H^-,J^-)$. 
Moreover, by using the homotopy of homotopy argument, $\Phi_{(H^+,J^+),(H-+,J^-)}$ on the homology level does not depend on the choice of $\{H_s\}$ and $\{J_s\}$, for more explanation see \cite[Lemma 3.12]{Salamon99}.

For $u:\mathbb{R}\times S^1\to \widehat{W}$ solving \ref{equation: continuation cylinder}, 
\begin{equation}
\label{equation: Energy of continuation cylinder}
    E(u)=\mathcal{A}_{H^+}(x^+)-\mathcal{A}_{H^-}(x^-)+\int \partial_s H_s(t, u(s,t)).
\end{equation}
Since $\partial_s H_s$ is non-positive at every point, $m(x^-,x^+)=0$ for all $\cA_{H^+}(x^+)<\cA_{H^-}(x^-)$. Then, the restriction of $\Phi_{(H^+,J^+),(H-+,J^-)}$ on $\HF^{<a}(H^+,J^+)$ induces a linear map 
\[\Phi^{<a}_{(H^+,J^+),(H^-,J^-)}:\HF^{<a}(H^+,J^+) \to \HF^{<a}(H^-,J^-).\]

Now, we consider three pairs $(H_i, J_i)$ for $i =1, 2, 3$ such that 
\begin{itemize}
	\item $H_i$ is a 1-periodic non-degenerate Hamiltonian function that is linear at infinity, 
	\item $J_i$ is $H_i$-regular $(d\lambda)$-compatible almost complex structure SFT-like of $r \geq r_i$ for some $r_i >1$, and 
	\item $H_1(t,p) \leq H_2(t,p) \leq H_3(t,p)$ for all $(t, p) \in S^1 \times \widehat{W}$. 
\end{itemize}
Then, by choosing a proper family of smooth functions $\{H_s\}$ between $H_i$ and $H_{i+1}$ and by choosing a proper family of almost complex structure $\{J_s\}$, one can prove that
\[\Phi_{(H_1,J_1), (H_3, J_3)}=\Phi_{(H_2,J_2), (H_3,J_3)}\circ \Phi_{(H_1,J_1), (H_2,J_2)}\]
nn the homology level. 

Finally, we set $H_1 = H_2 = H_3 = H$ and $J_1 = J_3 \neq J_2$. 
Then, by choosing the trivial choice, $\Phi_{(H_1,J_1), (H_3,J_3)}$ is the identity morphism on the homology level.
It proves that $\Phi_{(H_1,J_1),(H_2,J_2)}$ and $\Phi_{(H_2,J_2),(H_3,J_3)}$ are linear isomorphisms, i.e., $\HF(H,J_1)$ and $\HF(H,J_2)$ are the same vector space.
We note that the above arguments work for $\HF^{<a}(H,J)$ too. 

Now, we can define the Hamiltonian Floer homology of $H$ if $H$ is a non-degenerate Hamiltonian function that is linear at infinity, without mentioning a specific almost complex structure $J$. 
Also, we can define the continuation homomorphism without mentioning specific almost complex structures $J_1, J_2$.
$\Phi_{H^-,H^+} : {\HF}(H^-) \to {\HF}(H^+)$ without mentioning a specific almost complex structure $J^-, J^+$.
\begin{definition}
	\label{def HF(H)}
	\mbox{}
	\begin{enumerate}
		\item Let $H$ be a non-degenerate Hamiltonian function that is linear at infinity. 
		The {\bf Hamiltonian Floer homology of $\boldsymbol{H}$} (resp.\ {\bf filtered Hamiltonian Floer homology of $\boldsymbol{H}$ at $a \in \mathbb{R}$}) is the vector space $\HF(H)$ (resp.\ $\HF^{<a}(H)$) defined as follows:
		\[\HF(H):=\HF(H,J), \quad \HF^{<a}(H):=\HF^{<a}(H,J)\]
		for any $H$-regular almost complex structure $J$.
        \item Let $H_1, H_2$ be non-degenerate Hamiltonian functions that are linear at infinity. We further assume $H_1(t,p) \le H_2(t,p)$ for all $t\in S^1$ and $p \in \widehat{W}$. The {\bf continuation homomorphism from $\boldsymbol{H_1}$ to $\boldsymbol{H_2}$} (resp.\ {\bf continuation homomorphism from $\boldsymbol{H_1}$ to $\boldsymbol{H_2}$ at $a \in \mathbb{R}$}) is the linear map $\Phi_{H_1, H_2}:\HF(H_1)\to\HF(H_2)$ (resp.\ $\Phi^{<a}:\HF^{<a}(H_1)\to\HF{<a}(H_2)$) defined as follows:
        \[\Phi_{H_1, H_2}:=\Phi_{(H_1,J_1),(H_2,J_2)},\quad \Phi^{<a}_{H_1, H_2}:=\Phi^{<a}_{(H_1,J_1),(H_2,J_2)}\]
        for any $H_i$-regular almost complex structure $J_i$ for $i=1,2$.
    \end{enumerate}
\end{definition}

Before discussing the definition of symplectic homology, we would like to point out that the filtered Hamiltonian Floer homologies of $H$ can be seen as a persistence module. 
We recall that a persistence module consists of two families, one is a family of vector spaces and the other is a family of linear maps. 
Since $\{\HF(H)^{<a}\}_{a \in \R}$ gives a family of vector spaces, it is enough to give a family of linear maps between the filtered Hamiltonian Floer homologies. 

We note that if $a < b$, then $\CF^{<a}(H,J) \subset \CF^{<b}(H,J)$ by definition. 
The inclusion map induces a linear map $\iota_H^{a,b}: \HF^{<a}(H)\to \HF^{<b}(H).$
Then, then the pair $\left(\{\HF^{<a}(H)\}_{a\in\R}, \{\iota_H^{a,b}\}_{a\leq b \in \R}\right)$ satisfies the conditions in Definition \ref{def persistence module}. 
Thus, one can define the following:
\begin{definition}
	\label{def persistence module for H}
	Let $H$ be a non-degenerate Hamiltonian function which is linear at infinity. 
	The {\bf persistence module for $H$}, denoted by $\boldsymbol{B(H)}$, is defined to be a persistence module consisting of 
	\begin{itemize}
		\item a family of vector spaces $\left\{B(H)_a:=\HF^{<a}(H)\right\}_{a \in \R}$, and
		\item a family of linear maps $\left\{\iota_{H}^{a,b}:B(H)_a \to B(H)_b\right\}_{a \leq b \in \R}$. 
	\end{itemize}
\end{definition}

\begin{lem}
	\label{lem persistence module} 
	Let $H: \widehat{W} \to \R$ be a non-degenerate Hamiltonian function such that $H$ is linear at infinity.
	Then, $B(H)$ is a persistence module of finite type with the lower semi-continuity property.
\end{lem}
\begin{proof}
	Because $H$ is non-degenerate, $\mathcal{P}_1(H)$ is a finite set.
	In other words, $\CF(H)$ is a finite-dimensional vector space. 
	It is enough to prove that $B(H)$ is a persistence module of finite type. The lower semi-continuity property also follows from the lower-semi continuity property on the chain level.
\end{proof}
By the Structure Theorem (Theorem \ref{thm Structure Theorem}) guarantees that $B(H)$ is equivalent to a barcode consisting of finitely many bars.

We end Section \ref{subsection Hamiltonian Floer homology and the corresponding persistence module} by proving Lemmas \ref{lem morphism} and \ref{lem sup norm}, which we will need in the later parts of the paper.

\begin{lem}
	\label{lem morphism} 
	Let $F, G: S^1 \times \widehat{W} \to \R$ be non-degenerate Hamiltonian functions such that 
	\begin{itemize}
		\item $F$ and $G$ are linear at infinity, and 
		\item $F(t,p) \leq G(t,p)$ for all $(t,p) \in S^1 \times \widehat{W}$. 
	\end{itemize}
	Then, a family of linear maps 
	\[\left\{\Phi^{<a}_{F,G}:B(F)_a=\HF^{<a}(F) \to B(G)_a=\HF^{<a}(G)\right\}_{a \in \R}\] 
	is a morphism between $B(F)$ and $B(G)$. 
\end{lem}
\begin{proof}
	We note that $\Phi^{<a}_{F,G}$ is defined as the restriction of $\Phi_{F,G}$. 
	We also note that $\iota_F^{a,b}$ and $\iota_G^{a,b}$ are induced from the inclusion maps. Since the inclusion map and continuation homomorphism commutes at the chain level, the following equation holds:
	\begin{equation}
		\label{naturality}
  \Phi^{<b}_{F,G}\circ\iota_{F}^{a,b}=\iota_{G}^{a,b}\circ\Phi^{<a}_{F,G}.
	\end{equation}
	It completes the proof.
\end{proof}

\begin{lem}
	\label{lem sup norm}
	\mbox{}
	\begin{enumerate}
		\item Let $F:S^1 \times \widehat{W} \to \R$ be a non-degenerate Hamiltonian function which is linear at infinity.
		For any constant $c \in \R$, $B(F+c)$ is a $c$-shift of $B(H)$, i.e., $B(F+c)_a = \left(B(F)[c]\right)_a$ for all $a \in \R$. 
		\item Let $F, G: S^1 \times \widehat{W} \to \R$ be non-degenerate Hamiltonian functions of the same slope such that $F$ and $G$ are linear at infinity, especially outside of $W_{r_0}$, for the same $r_0 >1$.
		Then, $B(F)$ and $B(G)$ are $\|F-G\|_{\sup, W_{r_0}}$-interleaved, where 
		\[\|H\|_{\sup, W_{r_0}}:= \sup\left\{|H(t,x)| \Big\vert (t,x) \in S^1 \times W_{r_0}\right\}.\]
	\end{enumerate}
\end{lem}
\begin{proof}
	For the first item, we would like to point out that the chain complexes for $F$ and $F+c$ are generated by the same generators, i.e., $\mathcal{P}_1(F) = \mathcal{P}_1(F+c)$.
	Moreover, it is easy to observe that if $J$ is an almost complex structure satisfying all necessarily conditions, then the induced differential maps on $CF(F)$ and $CF(F+c)$ are the same.  
	
	By Definition \ref{def action functional on CF(H)}, for any $v \in CF(F) = CF(F+c)$, $\cA_F(v) = \cA_{F+c}(v) + c$. 
	Thus, $CF^{<a}(F+c) = CF^{<a+c}(F)$ for any $a \in \R$. 
	It concludes that $HF^{<a}(F+c) = HF^{a+c}(F)$, i.e., $B(F+c)_a = \left(B(F)[c]\right)_a$ for all $a \in \R$.
	
	For the second item, for convenience, let $\delta$ be the supreme norm of $F-G$ on $W_{r_0}$, i.e.,
	\[\delta:= \|F-G\|_{\sup, W_{r_0}}.\]
	Then, it is easy to observe that 
	\[G(t,x) \leq F(t,x) + \delta, F(t,x) \leq G(t,x) +\delta \text{  for all  } (t,x) \in S^1 \times \widehat{W}.\]
	Thanks to Lemma \ref{lem morphism}, there are morphisms of persistence modules 
	\[\Phi_{F,G+\delta} : B(F) \to B(G+\delta), \Phi_{G,F+\delta}: B(G) \to B(F+\delta).\]
	Moreover, since $\delta \geq 0$, it is easy to observe that 
	\[\Phi^{<a}_{F,F+2\delta} = \iota_F^{a, a+2\delta}, \Phi^{<a}_{G,G+2\delta} = \iota_G^{a, a+2\delta} \text{  for all  } a \in \R.\]
	
	Finally, we see the following diagrams 
	\begin{equation*}
		\begin{tikzcd}
			B(F)_a \ar[r,"\Phi^{<a}_{F,G+\delta}"] \ar[rr, bend right=20, "\iota_F^{a, a+2\delta}"'] & B(G+\delta)_a \ar[r, "\Phi^{<a}_{G+\delta,F+2\delta}"] & B(F+2\delta)_a
		\end{tikzcd}, \\
		\begin{tikzcd}
			B(G)_a \ar[r,"\Phi^{<a}_{G,F+\delta}"] \ar[rr, bend right=20, "\iota_G^{a, a+2\delta}"'] & B(F+\delta)_a \ar[r, "\Phi^{<a}_{F+\delta,G+2\delta}"] & B(G+2\delta)_a
		\end{tikzcd},
	\end{equation*}
	for all $a \in \R$.
	
	Now, the first item and Equation \eqref{naturality} complete the proof of the second item.
\end{proof}

\begin{remark}
    More precisely, two persistence modules $B(F)$ and $B(G)$ are $\norm{F-G}_\text{Hofer}$-interleaved through the continuation homomorphism. The proof uses inequality \eqref{equation: Energy of continuation cylinder} and is identical to the closed symplectic manifold case. But still, Lemma \ref{lem sup norm} is enough for the statements in this work.
\end{remark}

\subsection{A barcode entropy from the symplectic homology}
\label{subsection a barcode entropy from the symplectic homology}
In this subsection, we briefly review the definition of symplectic homology.
Then, we define a barcode entropy from symplectic homology. 
For more details on symplectic homology, we refer the reader to \cite{Floer-Hofer94, Viterbo99, Seidel08}. Since we deal with symplectic homology also for the non-contractible orbits, we strongly recommend the reader to also see \cite{Biran-Polterovich-Salamon03, Weber06}.

As mentioned in \cite{Viterbo99}, the symplectic homology of $(\widehat{W}, \lambda)$ is defined to be a direct limit of a direct system. We define the direct system first. 

We consider the following set:
\[\mathcal{D}:=\left\{\text{a nondegenerate Hamiltonian  }H | H(t,p) <0 \text{  for all  } (t,p) \in S^1 \times W\right\}.\]
Then, one can define a partial order on $\mathcal{D}$ as follows:
\[H_1 \prec H_2 \iff H_1(t,x) \le H_2(t,x) \text{  for all  } (t,x)\in S^1\times \widehat{W}.\]
We note that if $H_1 \prec H_2$, then there exist a continuation homomorphism from $H_1$ to $H_2$:
\[\Phi_{H_1,H_2}: \HF(H_1) \to \HF(H_2) \text{  and  } \Phi^{<a}_{H_1,H_2}: \HF^{<a}(H_1) \to \HF^{<a}(H_2) \text{  for all  } a \in \R.\]

Then, we define the following:
\begin{definition}
	\label{def symplectic homology}
	\mbox{}
	\begin{enumerate}
		\item The {\bf symplectic homology of $\boldsymbol{(W, \lambda)}$}, denoted by $\boldsymbol{\SH(W,\lambda)}$, is defined to be the following direct limit:
		\[\SH(W,\lambda)=\varinjlim_{(\mathcal{D}, \prec)} \HF(H).\]
		\item A {\bf filtered symplectic homology of $\boldsymbol{(W,\lambda)}$ at $\boldsymbol{a \in \R}$}, denoted by $\boldsymbol{\SH^{<a}(W,\lambda)}$, is defined to be the following direct limit:
		\[\SH^{<a}(W,\lambda)=\varinjlim_{(\mathcal{D}, \prec)} \HF^{<a}(H).\]
	\end{enumerate}
\end{definition}
  
By Equation \eqref{naturality} and the universal property of direct limits, for any $a < b \in \R$, $\{\iota^{a,b}_H\}_{(H,J) \in \mathcal{D}}$ induces maps 
\[\iota^{a,b}:\SH^{<a}(W,\lambda)\to \SH^{<b}(W),\]
such that the following diagram commutes for any $H$:
\[\begin{tikzcd}\HF^{<a}(H) \arrow{r}{\iota_H^{a,b}} \arrow{d}& \HF^{<b}(H) \arrow{d}\\
	\SH^{<a}(W,\lambda) \arrow{r}{\iota^{a,b}} & \SH^{<b}(W,\lambda).\end{tikzcd}\]
\begin{remark}
Symplectic homology $\SH(W,\lambda)$ can be computed as the direct limit of a direct system 
\[\left(\left\{\SH^{<a}(W, \lambda)\right\}_{a \in \R}, \left\{\iota^{a,b}\right\}_{a<b \in \R}\right),\]
i.e. $\SH(W,\lambda) =\varinjlim_{a \in \R} \SH^{<a}(W,\lambda)$. It can be proved by the commutativity of direct limits,
\begin{align*}
\SH(W,\lambda):=\varinjlim_{H\in\mathcal{D}}\HF(H)
=\varinjlim_{H\in\mathcal{D}}\varinjlim_{a\in\R}\HF^{<a}(H)
=\varinjlim_{a\in\R} \varinjlim_{H\in\mathcal{D}}\HF^{<a}(H)
=:\varinjlim_{a\in\R}\SH^{<a}(W,\lambda).
\end{align*}
For real number $b\in \R$, taking a direct limit among $a<b$, the commutativity of direct limits also shows the lower semi-continuity property for each $b$.
\end{remark}

Similar to Definition \ref{def persistence module for H}, we define a persistence module corresponding to $\SH(W, \lambda)$. 
\begin{definition}
	\label{def persistence module for SH}
	The {\bf persistence module for $\boldsymbol{\SH(W,\lambda)}$}, denoted by $\boldsymbol{B_{SH}(W,\lambda)}$, is defined to be a persistence consisting of 
	\begin{itemize}
		\item a family of vector spaces $\left\{B_{SH}(W,\lambda)_a:=\SH^{<a}(H)\right\}_{a \in \R}$, and
		\item a family of linear maps $\left\{\iota^{a,b}:B_{SH}(W,\lambda)_a \to B_{SH}(W,\lambda)_b\right\}_{a <b \in \R}$. 
	\end{itemize}
\end{definition}

Now, we can prove Lemma \ref{lem direct limit of persistence modules 1}.
\begin{lem}
	\label{lem direct limit of persistence modules 1}
	The persistence module for $\SH(W, \lambda)$ is a direct limit of a direct system $\left\{B(H) | (H,J) \in \left(\mathcal{D}, \prec\right)\right\}$, i.e., 
	\[B_{SH}(W,\lambda) = \varinjlim_{(\mathcal{D}, \prec)}B(H).\]
\end{lem}
\begin{proof}[Sketch of the proof]
	One can directly show that $B_{SH}(W,\lambda)$ satisfies the universal property because for all $a \in \R$, $B_{SH}(W,\lambda)_a$ is defined as a direct limit. We omit the details.
\end{proof}

We would like to point out the following: 
assuming the non-degeneracy of $(Y,\alpha)$, $B_{SH}(W,\lambda)$ is a pfd persistence module, i.e., for all $a \in \R$, $B_{SH}(W,\lambda)_a = \SH^{<a}(W,\lambda)$ is finite dimensional. 
Thus, by applying the Structure Theorem (Theorem \ref{thm Structure Theorem}), one can assume that $B_{SH}(W, \lambda)$ is equivalent to a barcode. 
Now, we can define one of the main characters of our paper. 

\begin{definition}
	\label{def SH barcode entropy}
	\mbox{}
	\begin{enumerate}
		\item Let $B$ be a barcode. For any positive real number $\epsilon$, $\boldsymbol{n_\epsilon(B)}$ is defined as 
		\[\mathbb{Z}_{\geq 0} \cup \{\infty\} \ni n_\epsilon(B):= \text{  the number of bars in  } B \text{  whose length is greater than  } \epsilon.\]
		\item The {\bf $\boldsymbol{\epsilon}$-symplectic homology barcode entropy of $\boldsymbol{(W,\lambda)}$}, denoted by $\boldsymbol{\hbar^{\SH}_\epsilon(W,\lambda)}$, is defined as 
		\[\hbar^{\SH}_\epsilon(W,\lambda):= \lim_{T \to \infty} \frac{1}{T} \log \Big(n_\epsilon\big(\tru(B_{\SH}(W, \lambda), T)\big) \Big).\]
		\item The {\bf symplectic homology barcode entropy of $\boldsymbol{(W,\lambda)}$}, denoted by $\boldsymbol{\hbar^{\SH}(W,\lambda)}$, is defined as
		\[\hbar^{\SH}(W,\lambda):= \lim_{\epsilon \searrow 0} \hbar^{\SH}_\epsilon(W, \lambda).\]
	\end{enumerate}
\end{definition}

The ($\epsilon$-)symplectic homology barcode entropy, which is not clear to be finite from its definition, is finite because it is bounded above by the topological entropy of the Reeb flow on the contact boundary, see theorem \ref{thm vs topological entropy formal}. Recall that the topological entropy is always finite for smooth flows. Also, symplectic homology barcode entropy is not a zero function, see Remark \ref{rmk infinite bars}.

Yet ($\epsilon$-)symplectic homology barcode entropy is defined as the entropy of Liouivlle domains. Later, in section \ref{subsection fillings}, we will prove that the two symplectic homology barcode entropy agree when the contact boundary are the same. After that, we can define symplectic homology barcode entropy of $(Y,\alpha)$, or shortly $\SH$-barcode entropy of $(Y,\alpha)$, for contact manifolds with Liouville fillings.

\begin{remark}[Alternative defintion]
    For concordance with \cite{Ginzburg-Gurel-Mazzucchelli22}, we introduce another definition for $\hbar^{\SH}_\epsilon(W,\lambda)$. Let $b_\epsilon(B_{\SH}(W,\lambda), T)$ be the number of bars in $B$, with their length exceeding $\epsilon$ and the left-endpoint being no more than $T$. Then, by definition of truncation,
    \[b_\varepsilon(B_{\SH}(W,\lambda),T-\varepsilon)\le
    n_\varepsilon(\tru(B_{\SH}(W,\lambda),T)
    \le b_\varepsilon(B_{\SH}(W,\lambda),T).\]
    Therefore, the exponential growth rates of $b_\varepsilon(B_{\SH}(W,\lambda),T)$ and
    $n_\epsilon\big(\tru(B_{\SH}(W, \lambda), T))$ under $T$ agree.
\end{remark}

\begin{remark}[Infinite bars]
\label{rmk infinite bars}
In \cite{Meiwes18}, Meiwes studied the positivity of $\,\Gamma(\SH(W,\lambda))$, the exponential growth rate of the filtered symplectic homology, see \cite[Definition 4.2]{Meiwes18}. From the barcode point of view, this number can be interpreted as the exponential growth rate under $T\in \R$ of the number of infinite bars having left-endpoints no more than $T$. This number can be smaller than any ($\varepsilon$-)symplectic homology barcode entropy. However, the invariant is more `rigid' than the barcode entropy in the following sense.

For a given Liouville domain $(W, \lambda)$ and let $(W_f,\lambda)$ be a Liouville domain defined by
\[W_f=\widehat{W} \setminus \set{(r,x)|r>f(x), x\in \partial W}, \lambda_f=\lambda|_{W_f}.\]
Note that the contact boundary of $W_f$ is $(Y, f\alpha)$. Then, by \cite[Corollary 9.3]{Meiwes18}, $\,\Gamma(\SH(W, \lambda))$ is positive if and only if $\, \Gamma(\SH(W_f,\lambda_f))$ is positive. With this fact, the author found examples of Liouville domains with positive $\Gamma(\SH(W,\lambda))$, see \cite[Chapter12]{Meiwes18} for examples including 2n-dimensional Ball with exotic Weinstein structure. These examples also assure the existence of Liouville-fillable contact structures on each of which every contact form has positive barcode entropy.

\end{remark}

\subsection{Hamiltonian Floer homologies of degenerate Hamiltonian functions}
\label{subsection Hamiltonian Floer homologies of degenerate Hamiltonian functions}
In Section \ref{subsection Hamiltonian Floer homologies of degenerate Hamiltonian functions}, we only considered {\em non-degenerate} Hamiltonian functions. 
In this subsection, we introduce the Hamiltonian Floer homologies of {\em degenerate} Hamiltonian functions when the Hamiltonian functions are linear at infinity, in order to define another barcode entropy in Section \ref{subsection barcode entropy of a Hamiltonian function}.

Throughout this and the following subsections, let $H: S^1 \times \widehat{W} \to \R$ be a Hamiltonian function such that  $H$ is linear at infinity, especially outside of $W_{r_0}$ for some $r_0 >1$. We would like to point out that $H$ does not need to be non-degenerate in this and the following subsections. 
Now, we define the Hamiltonian Floer homology of $H$ as similar to Definition \ref{def symplectic homology}.

We recall that in Definition \ref{def symplectic homology}, $\SH(W,\lambda)$ is defined to be a direct limit.
Similarly, we define the Hamiltonian Floer homology of $H$ as a direct limit. 
To do so, we need to fix a direct system first. 

We consider the following set:
\[\mathcal{D_H}:=\left\{\text{a nondegenerate Hamiltonian  }G | G(t,p) < H(t,p) \text{  for all  } (t,p) \in S^1 \times W\right\}.\]
Then, one can define a partial order on $\mathcal{D}$ as follows:
\[H_1 \prec H_2 \iff G_1(t,x) \le H_2(t,x) \text{  for all  } (t,x)\in S^1\times \widehat{W}.\]
We note that if $H_1 \prec H_2$, then there exist a continuation homomorphism from $H_1$ to $H_2$:
\[\Phi_{H_1,H_2}: \HF(H_1) \to \HF(H_2) \text{  and  } \Phi^{<a}_{H_1,H_2}: \HF^{<a}(H_1) \to \HF^{<a}(H_2) \text{  for all  } a \in \R.\]

Now, one can define the Hamiltonian Floer homology of $H$, even if $H$ is degenerate, as follows:
\begin{definition}
	\label{def Hamiltonian Floer homology}
	\mbox{}
	\begin{enumerate}
		\item The {\bf Hamiltonian Floer homology of $\boldsymbol{H}$}, denoted by $\boldsymbol{\HF(H)}$, is defined to be the following direct limit:
		\[\HF(H):=\varinjlim_{(\mathcal{D}_H, \prec)} \HF(F).\]
		\item A {\bf filtered Hamiltonian Floer homology of $\boldsymbol{H}$ at $\boldsymbol{a \in \R}$}, denoted by $\boldsymbol{\HF^{<a}(H)}$, is defined to be the following direct limit:
		\[\HF^{<a}(H):=\varinjlim_{(\mathcal{D}_H, \prec)} \HF^{<a}(F).\]
	\end{enumerate}
\end{definition}

We note that if $H$ is non-degenerate, then $H$ is the unique maximal element of $\mathcal{D}_H$. 
Thus, Definitions \ref{def HF(H)} and \ref{def Hamiltonian Floer homology} give the same Hamiltonian Floer homology for non-degenerate $H$. 

As similar to the case of $\SH(W,\lambda)$, one can easily observe that Equation \eqref{naturality}, together with the universal property of direct limits, induces a family of linear maps $\left\{\iota_H^{a,b}\right\}_{a\leq b\in\R}$ such that the following diagram commutes for every $(F,J) \in \mathcal{D}_H$:
\[\begin{tikzcd}\HF^{<a}(F) \arrow{r}{\iota_F^{a,b}} \arrow{d}& \HF^{<b}(F) \arrow{d}\\
	\HF^{<a}(H) \arrow{r}{\iota_H^{a,b}} & \HF^{<b}(H).\end{tikzcd}\]

As usual, we define the corresponding persistence module: 
\begin{definition}
	\label{def persistence module for degenerate H}
	The {\bf persistence module for $H$}, denoted by $\boldsymbol{B(H)}$, is defined to be a persistence module consisting of 
	\begin{itemize}
		\item a family of vector spaces $\left\{B(H)_a:=\HF^{<a}(H)\right\}_{a \in \R}$, and
		\item a family of linear maps $\left\{\iota_{H}^{a,b}:B(H)_a \to B(H)_b\right\}_{a \leq b \in \R}$. 
	\end{itemize}
\end{definition}

Similar to Lemma \ref{lem direct limit of persistence modules 1}, the following Lemma is easy to prove:
\begin{lem}
	\label{lem direct limit of persistence modules 2}
	The persistence module of $H$ is a direct limit of a direct system $\{B(F) | (F,J) \in (\mathcal{D}_H, \prec)\}$, i.e.,
	\[B(H)=\varinjlim_{(\mathcal{D}_H, \prec)}B(F).\]
\end{lem}

We note that since a direct limit of finite dimensional vector spaces could be an infinite-dimensional vector space, we cannot claim that $B(H)$ is a pfd persistence module. 
Thus, we cannot apply the Structure Theorem in order to show that $B(H)$ is equivalent to a barcode. 
In the rest of Section \ref{subsection Hamiltonian Floer homologies of degenerate Hamiltonian functions}, we prove that $B(H)$ is equivalent to a barcode, by utilizing Lemma \ref{lem complete space}. Note that $B(H)$ has the lower semi-continuity property again by the commutativity of direct limits.

In order to prove that, we need to choose a cofinal sequence $\{H_i \in \mathcal{D}_H\}_{i \in \mathbb{N}}$.
Without loss of generality, we assume that the chosen cofinal sequence satisfies the following conditions: 
\begin{enumerate}
	\item[(i)] for all $i \in \mathbb{N}$, $H_i \prec H_{i+1}$, 
	\item[(ii)] for all $i \in \mathbb{N}$, $H_i$ and $H$ coincide each other on $\widehat{W} \setminus W_{r_0}$, and
	\item[(iii)] for all $i \in \mathbb{N}$, $\|H-H_i\|_{\sup, W_{r_0}} \leq \tfrac{1}{2^i}$.
\end{enumerate}

Now, one can prove Lemma \ref{lem direct limit = Cauchy limit}.
\begin{lem}
	\label{lem direct limit = Cauchy limit}
	Let $\left\{H_i\right\}_{i \in \mathbb{N}}$ be a cofinal sequence in $\mathcal{D}_H$ satisfying (i)--(iii). 
	Then, $B(H)$ is the limit of the sequence of persistence modules $\left\{B(H_i)\right\}_{i \in \mathbb{N}}$ with respect to the interleaving distance.  
\end{lem}
\begin{proof}
	We would like to show that for any $\varepsilon >0$, there exists $N_\varepsilon \in \mathbb{N}$ such that $d_{int}(B(H),B(H_i)) < \epsilon$ for all $ i \geq N_\varepsilon$, or equivalently, 
	$B(H)$ and $B(H_i)$ are $\epsilon$-interleaved. 
	In other words, we would like to show the existence of morphisms $\Psi_{\infty,i}, \Phi_{i,\infty}$ for all $i \geq N_\varepsilon$ such that the following diagrams commute
	\begin{equation}
		\label{eqn commuting diagram}
		\begin{tikzcd}
			B(H_i) \ar[r,"\Phi_{i, \infty}"]  \ar[rr, bend right=20, "S_{H_i}^{2\varepsilon}"'] & B(H)[\varepsilon] \ar[r, "\Psi_{\infty, i}\lbrack\varepsilon\rbrack"]  & B(H)[2\varepsilon]
		\end{tikzcd},
		\begin{tikzcd}
			B(H) \ar[r,"\Psi_{\infty, i}"] \ar[rr, bend right=20, "S_H^{2\varepsilon}"'] & B(H_i)[\varepsilon] \ar[r, "\Phi_{i, \infty}\lbrack\varepsilon\rbrack"] & B(H)[2\varepsilon]
		\end{tikzcd}, 
	\end{equation}
	where $S_H^\delta$ and $S_{H_i}^\delta$ are $\delta$-shift morphisms of $B(H)$ and $B(H_i)$ defined in Definition \ref{def shift}.

	For $\varepsilon>0$, because of (iii), there exists a natural number $N_\varepsilon$ such that $\norm{H_j-H_i}_{\infty, W_{r_0}}<\varepsilon$ for all $j > i > N_\varepsilon$. 
	Then, for such $i, j$, $H_j \prec H_i +\varepsilon$ in $\mathcal{D}$. 

	In order to define $\Psi_{\infty, i}$, we observe that there exists a continuation homomorphism 
	\[\Phi_{H_j, H_i+\varepsilon}:B(H_j)\to B(H_i+\varepsilon) = B(H_i)[\varepsilon].\] 
	Moreover, for all $i, j, k$ such that $k \geq j \geq i \geq N_\varepsilon$, there exists a commuting diagram 
		\[\begin{tikzcd}
			B(H_j) \arrow[r, "\Phi_{H_j,H_i+\varepsilon}"] \arrow[d, "\Phi_{H_j,H_k}"'] & B(H_i + \varepsilon) = B(H_i)[\varepsilon] \\
			B(H_k) \arrow[ru, "\Phi_{H_k, H_i+\varepsilon}"']
		\end{tikzcd}.\]
	Then, from the universal property of the direct limit $B(H)$, there exists a morphism $\Psi_{\infty,i}: B(H) \to B(H_i)[\varepsilon]$, which we desired to construct.

	In order to define $\Phi_{i,\infty}$, we note that since $B(H)$ is a direct limit, there is a morphism 
	\[\Phi_{H_i,H}: B(H_i) \to B(H).\]
	Let $\Phi_{i,\infty} := S_H^\varepsilon \circ \Phi_{H_i,H}$. 
	
	Then, the first diagram in \eqref{eqn commuting diagram} commutes, because of the following observation:
	\begin{align*}
		\Psi_{\infty,i}[\varepsilon] \circ \Phi_{i,\infty} =& \Psi_{i,\infty}[\varepsilon] \circ S_H^\varepsilon \circ \Phi_{H_i, H} \left(\because  S_H^\varepsilon \circ \Phi_{H_i,H}\right) \\ 
	=& S_{H_i[\varepsilon]}^\varepsilon \circ \Psi_{\infty,i} \circ \Phi_{H_i,H} \left(\because \text{  Definition \ref{def shift} (3)}\right) \\
	=& S_{H_i[\varepsilon]}^P\varepsilon \circ S_{H_i}^\varepsilon \left(\because \text{  the universal property of the direct limit  } B(H) \right) \\ 
	=& S_{H_i}^{2\varepsilon}.
	\end{align*}

	To prove that the second diagram in \eqref{eqn commuting diagram} commutes, we consider the following diagram:
	\[\begin{tikzcd}
		B(H_i) \ar[rr, "\Phi_{H_i,H_j}"] \ar[dr, "\Phi_{H_i,H}"'] \ar[ddr, bend right=30, "\Phi_{i,\infty}\lbrack\varepsilon\rbrack \circ S_{H_i}^\varepsilon"'] &  & B(H_j) \ar[dl, "\Phi_{H_j,H}"] \ar[ddl, bend left=30, "\Phi_{j,\infty}\lbrack\varepsilon\rbrack \circ S_{H_j}^\varepsilon"] \\
		                                & B(H) \ar[d, dashed, "\exists! F"]& \\
		                                & B(H)[2\varepsilon] &
	\end{tikzcd}\]
	
	We note that from the universal property that $\Psi_{\infty, i}$ satisfies, one can observe that both of $F= S_H^{e\varepsilon}$ and $F= \Phi_{i,\infty}[\varepsilon] \circ \Psi_{\infty, i}$ make the diagram commuting. 
	Because of the uniqueness of $F$, it proves that the second diagram in \eqref{eqn commuting diagram} commutes.
\end{proof}

Corollary \ref{cor equivalent to a barcode} follows Lemmas \ref{lem complete space}, \ref{lem persistence module}, and \ref{lem direct limit = Cauchy limit}.
\begin{cor}
	\label{cor equivalent to a barcode}
	For a Hamiltonian $H$ satisfying the conditions in Section \ref{subsection Hamiltonian Floer homologies of degenerate Hamiltonian functions}, $B(H)$ is equivalent to a barcode. 
\end{cor}
\begin{proof}
	Lemma \ref{lem direct limit = Cauchy limit} gives a sequence of persistence modules $B(H_i)$, which converges to $B(H)$. 
	Moreover, for all $i \in \mathbb{N}$, $H_i \in \mathcal{D}$. 
	Thanks to Lemma \ref{lem persistence module}, $B(H_i)$ is of finite type for all $i$. 
	Then, $B(H_i)$ is a cid and q-tame persistence module.
	Thus, Lemma \ref{lem complete space} guarantees that $B(H)$ is a cid persistence module. 
	In other words, $B(H)$ is equivalent to a barcode consisting of countably many intervals. This barcode further satisfies the following: For every interval $I$, there are finitely many bars in $B$ having $I$ as a subset.
\end{proof}

\subsection{Barcode entropy of a convex radial Hamiltonian}
\label{subsection barcode entropy of a Hamiltonian function}
Now, we define another barcode entropy, different from Definition \ref{def SH barcode entropy}.  For that, we consider {\em convex radial} Hamiltonian functions defined below.

\begin{definition}
	\label{def convex radial}
	An autonomous Hamiltonian $H: \widehat{W}\to \R$ is called {\bf convex radial} of slope $T$ if there exists a convex function $h:[1,\infty)\to [0,\infty)$ such that
	\begin{itemize}
		\item $H|_{W} \equiv 0$ and $H(r,p) = h(r)$ for all $(r,p) \in [1,\infty) \times \partial W$,
		\item there are constants $r_0>1, C>0, T>0$ satisfying $h(r) = Tr - C$ for all $r \ge r_0$,
		\item $h^{\prime\prime}(r)>0$ for $r\in [1,\infty)$ such that $h^\prime(r) \in \set{0}\cup \text{Spec}(R_\alpha)$,
		\item and the constant $T$ is not a period of any closed Reeb orbits of $(\partial W, \alpha)$.
	\end{itemize}
\end{definition}

When there is no confusion, we will simply denote the Hamiltonian by $H=h(r)$. For convex radial Hamiltonian $H=h(r)$, the Hamiltonian vector fields $X_H$ is given as $h^\prime(r)R_\alpha$ for $[1,\infty)\times Y$ and zero for $W$. Thus, the Hamiltonian 1-map of $H$ restricted to each hypersurface $\set{r}\times \partial W$ is the time $0$-map to the time $t$-map of the Reeb flows generated on $(Y,\alpha)$ and identity at the interior. For this reason, we consider convex radial Hamiltonians to mimic the Reeb flow. Later in \ref{subsection equivalence of two barcode entropy}, we will prove that the barcode $B(H)$ also mimics the Barcode $B_{\SH}(W, \lambda)$.

By direct computation, a 1-periodic orbit of $X_H$ is either a constant solution on $W$ or a non-constant solution on $(1,r_0]\times Y$ in forms of $\set{r_\gamma}\times\gamma$. Here $\gamma$ denotes some closed Reeb orbit and $r_\gamma$ is a unique real number satisfying $h^\prime(r_\gamma)=T_\gamma$, the period of $\gamma$. The Hamiltonian action of the 1-periodic orbit is zero for a constant solution or $r_\gamma \cdot h^\prime(r_\gamma)-h(r_\gamma)$. Note that the action only depends on the period of the corresponding closed Reeb orbit. Based on this observation, we define the reparametrization function $s_h:[0, T]\to[0, C]$ that converts the period to the Hamiltonian action.

\begin{lem}
\label{def reparameterization function}
    For the convex radial function $H=h(r)$ with slope $T$, define the {\textbf reparametrization function $s_h$} as the function satisfying the following: For $t\in [0,T]$, $s_h(t)= r h^\prime (r) - h(r)$ for $h^\prime(r)=t$. Then, the reparametrization function is well-defined and monotone increasing bijective function from $[0,T]$ to $[0,C]$.
\end{lem}
\begin{proof}
For $1 \le r_1 < r_2$ with
\begin{equation*}
    [rh^\prime(r)-h(r)]_{r_1}^{r_2}=\int_{r_1}^{r_2} (rh^\prime(r)-h(r))^\prime\,dr = \int_{r_1}^{r_2} rh^{\prime\prime}(r)\,dr
\end{equation*}
With the convexity of $h$, we conclude
\begin{equation}
\label{eqn reparametrization function difference}
(h^\prime(r_2)-h^\prime(r_1))
\le [rh^\prime(r)-h(r)]_{r_1}^{r_2}
\le r_0 (h^\prime(r_2)-h^\prime(r_1))
\end{equation}
When $h^\prime(r_1)=h^\prime(r_2)$, the equation implies $r_1 h^\prime(r_1)-h(r_1)=r_2h^\prime(r_2)-h(r_2)$. Thus, the reparametrization function $s_h$ is well defined. The function is also a strictly increasing continuous function by the equation. By checking $s_h(0)=1\cdot 0 - 0 = 0$ and $s_h(T)=r_0 T - h(r_0)= C$, we conclude that the reparametrization function $s_h$ is a monotone increasing bijective function from $[0,T]$ to $[0,C]$.
\end{proof}

In order to calculate the barcode $B(H)$ of a convex radial Hamiltonian $H=h(r)$, we will use precise perturbations with controlled behaviors. Denote the set of $\mathcal{P}_\alpha^{<T}$ the set of closed Reeb orbits of Reeb vector field $R_\alpha$ with a period less than $T$. Let $f:\widehat{W}\to\R$ be an extension of some Morse function $f_0:W\to[-1,0]$ which is supported in the small neighborhood $W_{1+{r_h}}$ where $h^\prime(r_h)<T_{\min}$, where $T_{\min}$ is the period of the shortest closed Reeb orbits. We further assume there are no critical points in $r_0 \in (1/4, 1+r_h)$.
Next, let $g$ be a 1-periodic Hamiltonian satisfying:
\begin{itemize}
	\item $g_t=g(t,-):\widehat{W}\to [-1,0]$ supported in 
    \[ V_h=\coprod_{\gamma \in \mathcal{P}^{<T}_{\alpha}} V_{h,\gamma},\]
    where $V_{h,\gamma}$ is the small closed neighborhood of $\set{r_\gamma}\times \im(\gamma)$ with $V_{h,\gamma}\subset (1+r_h, r_0)\times Y$
	\item for sufficiently small $\delta>0$, the set of 1-periodic orbits of $X_{H+\delta g}$ in $(1+r_h, r_0)\times Y$ are precisely
	\[\set{\,\hat{\gamma}, \check{\gamma}~:~\gamma\text{ is a closed Reeb orbit with period less than $T$}}\]
	with $\im(\hat{\gamma})=\im(\check{\gamma})=\set{r_\gamma}\times \im(\gamma)$ for $r_\gamma$ with $h^\prime(r_\gamma)=T_\gamma$, the period of $\gamma$, and
	\[\cA_{H+\delta g}(\hat{\gamma})=\cA_{H}\left(\set{r_\gamma}\times \gamma\right)+\delta=s_h(T_\gamma)+\delta, \quad\cA_{H+\delta g}(\check{\gamma})=\cA_{H}\left(\set{r_\gamma}\times \gamma\right)=s_h(T_\gamma),\]
	\item 1-periodic orbits $\hat{\gamma}$ and $\check{\gamma}$ are non-degenerate,
    \item for any $J\in \mathcal{J}^{\text{SFT}}_{r_0}$, $\mathcal{M}(\check{\gamma},\hat{\gamma};H+\delta g,J)$ consists of two Floer cylinders, and 
    \item both two Floer cylinders in $\mathcal{M}(\check{\gamma},\hat{\gamma};H+\delta g,J)$ stays inside the support $(1+r_h, r_0)\times Y$.
\end{itemize}
Such $g$ can be constructed by using some Morse function defined on each periodic orbit for $X_{H}$, see \cite[Proposition 2.2.]{Cieliebak-Floer-Hofer-Wysocki96} for details.
\begin{definition}
    We call the triplet $(h,f,g)$ is a {\bf convex function $h$ with Morse-perturbation data $(f,g)$}, if it satisfied the aforementioned conditions.
\end{definition}
For the triplet $(f,g,h)$, there exists $\delta_0=\delta_0(f,g,h)$ such that for every $0<\delta<\delta_0$, 1-periodic orbits of $X_{H+\delta(f+g)}>0$ are either constant solutions corresponding critical points of $f|_W$ or non-constant solutions $\hat{\gamma}$, $\check{\gamma}$ from each Reeb orbit $\gamma$. The constant solution has an action between $[0,\delta]$. For the triplet $(f,g,h)$, the Floer cylinder with extremely small energy is also controllable. 

\begin{lem}
    \label{convex radial: no differential}
    Let $\set{J_\delta}_{\delta<\delta_0}$ be a subset of $\mathcal{J}^{\text{SFT}}_{r_0}$ such that
    \begin{itemize}
        \item $J_\delta$ is $H+\delta(g+f)$-regular, i.e. $\partial_\delta:=\partial_{H+\delta(f+g), J_\delta}$ is well-defined and $\partial_\delta \circ \partial_\delta = 0$ and   
        \item $\lim_{\delta\to 0}J_\delta=J_0$ for some $J_0 \in \mathcal{J}^{\text{SFT}}_{r_0}$ in $C^\infty$-topology. 
    \end{itemize}
    Then, there exists a constant $\delta_1=\delta_1(f,g,h,\set{J_\delta})$ such that, for every $\delta<\delta_1$, $\left<\partial_\delta x_\delta, y_\delta\right>=0$ for $x_\delta$, $y_\delta$ corresponding to the Reeb orbits with same period.
\end{lem}
\begin{proof}
    For the proof, we follow \cite[Chapter 3.]{Cieliebak-Floer-Hofer-Wysocki96}. For $x_\delta, y_\delta \in \set{\hat{\gamma}, \check{\gamma}}$, there are precisely two Floer cylinders in $\mathcal{M}(\check{\gamma},\hat{\gamma};H+\delta(g+f), J_\delta)$ and other moduli spaces are the empty set. Therefore, $\left<\partial_\delta x_\delta, y_\delta\right>=0$ always holds for $\delta<\delta_0$.

    Assume, there exists a sequence positive real number $\set{\delta_i}_{i\in N}$ converging to zero with $\left<\partial_{\delta_i} x_{\delta_i}, y_{\delta_i}\right>=1$ for some $x_{\delta_i}$, $y_{\delta_i}$ each corresponding to the Reeb orbit $\gamma^x_{\delta_i}, \gamma^y_{\delta_i}$ with same period. From the above discussion, $\gamma^x_{\delta_i}\ne \gamma^y_{\delta_i}$. Since the number of Reeb orbit with period less than $T$ are finitely, taking a subsequence, we may assume $\gamma^x_{\delta_i}, \gamma^y_{\delta_i}$ does not depend on $\delta_i$.
    
    From the definition of $\partial_{\delta}$, there exists a Floer cylinder $u_i$ in $\mathcal{M}(x_{\delta_i},y_{\delta_i};H+\delta_i(g+f), J_{\delta_i})$. The action of such a cylinder is precisely $\delta_i$. Since $\gamma^x$ and $\gamma^y$ are distinct, after translation, we assume $u_i(0,)$ have intersection with $[1+r_h,r_1]\times Y \setminus V$, the support of $g$. By Gromov-Floer compactness, taking a subsequence, $u$ converges in $C^{\infty}_{\text{loc}}$ to some $u_0$. Since $u_i(0,)$ converges $u(0,)$, the loop $u(0,)$ also has intersection with $[1+r_h,r_1]\times Y \setminus V$.

    Since ${E(u_i)}=\delta_i\to 0$, the energy of $u$ is zero. By definition, $|\partial_s u|=0$ for all points and $u$ is independent of $s$ and $u(s,-):S^1\to \widehat{W}$ must be the 1-periodic solution of $H$. However, since the loop $u(0,)$ intersect $[1+r_h,r_1]\times Y \setminus V$, $u(0,-)$ is not a 1-periodic orbit which gives a contradiction. Therefore, there exists a constant $\delta_1=\delta_1(f,g,h,\set{J_\delta})$ such that, for every $\delta<\delta_1$, $\left<\partial_\delta x_\delta, y_\delta\right>=0$ for $x_\delta$, $y_\delta$ corresponding to the Reeb orbits with same period.
\end{proof}

With the help of Lemma \ref{convex radial: no differential}, we can characterize the Barcode of a convex radial Hamiltonian. Before 

\begin{lem}
	\label{lem convex radial}
	Let $(W, \lambda)$ be a Liouville domain and let $H=h(r)$ be a convex radial Hamiltonian of slope $T$. Then, the persistence module $B(H)$ is a lower semi-continuous persistence module of finite type with
    \vspace{-3mm}
    \begin{itemize}
        \item $\dim H^*(W)$ many bars having left-endpoint as zero and
        \item $2\times$(number of closed Reeb orbit with period $t$) many bars having endpoints of $s_h(t)$ for $t\in {\Spec}^{<T} (R_\alpha)$.
    \end{itemize}
\end{lem}

\begin{proof} For a convex function $h$ with Morse-perturbation data $(f,g)$, consider a family of Hamiltonians $\set{H_\delta := H+\delta(f+g)}_{\delta<\delta_0)}$ and a family of 1-periodic almost complex structures in $\set{J_\delta}_{\delta<\delta_0}\subset \mathcal{J}^{\text{SFT}}_{r_0}$ such that every $J_\delta$ is $H_\delta$-regular and there exists $J_0$ such that $J_\delta$ converges in $J_0$ in $C^\infty$-topology.

Note that $\{H_\delta\}$ gives a cofinal sequence of $\mathcal{D}_H$. For each $\delta<\delta_1$, where $\delta_1=\delta_1(f,g,h,\set{J_s})$ from Lemma \ref{convex radial: no differential}, the chain complex $\CF(H_\delta)$ consists of
\begin{itemize}
	\item (number of critical points of $f$) many generators having action between $[0,\delta]$,
	\item (number of closed Reeb orbits with period $t$) many generators with action $s_h(t)$ and
	\item (number of closed Reeb orbits with period $t$) many generators with action $s_h(t)+\delta$.
\end{itemize}
	
We first prove that $B(H)$ is a persistence module of finite type. For every $a \in \R$, $\dim \HF^{<a}(H_\delta)$ is uniformly bounded above by (number of critical points of $f$) + 2$\cdot$(number of closed Reeb orbits with period less than $T$). Especially when $a \le 0$, $\HF^{<a}(H_\delta)=0$. By taking a direct limit for each $a$, we conclude $\HF^{<a}(H)$ is a finite-dimensional vector space for every $a\in \R$ and especially $0$ for $a \le 0$.
	
For any $a\le b$ with $[a,b)\cap s_h\left(\set{0}\cup \Spec^{<T}(R_\lambda)\right)=\phi$, $\iota^{a,b}_{H_\delta}$ is an isomorphism for sufficiently small $\delta$ because $\CF^{<a}(H_\delta)=\CF^{<b}(H_\delta)$. Taking a direct limit, $\iota^{a,b}_{H_\delta}$ is also an isomorphism. Therefore, the lower semi-continuity holds for all $a \in \R$ and the continuity holds except for the finite set $s_h\left(\set{0}\cup \Spec^{<T}(R_\lambda)\right)$. Therefore, the persistence module $B(H)$ is the lower semi-continous persistence module with finite type.
	
Thus, $B(H)$ is the barcode of finite type with bars in forms of $(a,b]$ or $(a,\infty)$. Note that the endpoints of bars in $B(H)$ are the elements of the set $s_h\left(\set{0}\cup \Spec^{<T}(R_\lambda)\right).$ Since
\[\norm{H-H_\delta}_{\sup} = \delta \cdot \norm{f+g}_{\sup} \le \delta,\]
there exists $\delta$-matching from $B(H_\delta)$ to $B(H)$. Assuming that $3\times\delta$ is smaller than the minimum difference between the elements of the set $s_h\left(\set{0}\cup \Spec^{<T}(R_\lambda)\right)$, the every bars in $B(H)$ are longer than $3\delta$. Therefore, the $\delta$-matching from $B(H_\delta)$ to $B(H)$ is a surjection.

We now prove there are $\dim (H^*(W))$ many bars having zero as a left-endpoint. Since $f$ is a Morse function on $W$, for $\delta<a<s_h(T_{\min})$, $\HF^{<a}(H_\delta)=H^*(W)$. Since there are no generators having action between $(\delta, s_h(T_{\min}))$, $\iota^{a,b}_{H_\delta}=\textrm{id}$ for $a,b \in (\delta, s_h(T_{\min})]$. Therefore, we have precisely $\dim (H^*(W))$ many bars containing $(\delta, s_h(T_{\min})]$ for $B(H_\delta)$. From the $\delta$-matching, there are precisely $\dim (H^*(W))$ many bars containing $(2\delta, s_h(T_{\min})-\delta]$ for $B(H)$. Since $0$ is the only possible left-endpoint for these bars, and since bars in forms of $(0,a]$ or $(0,\infty)$ must contain $(2\delta, s_h(T_{\min})-\delta]$, there are precisely $\dim (H^*(W))$ many bars having 0 as the left-endpoint.
	
Recall that there are (number of closed Reeb orbits with period $t$) many generators with action $s_h(t)$ and $s_h(t)+\delta$. Therefore, there are precisely (number of closed Reeb orbits with period $t$) many bars having endpoint as $s_h(t)+\delta$ and $s_h(t)$ each. By Lemma \ref{convex radial: no differential}, there are no $(s_h(t), s_h(t)+\delta]$ in $B(H_\delta)$. Therefore, the bar having some endpoint larger than $\delta$ is longer than $2\delta$. Through $\delta$ matching, these bars correspond to some bars in $B(H)$ while the endpoints $s_h(t)$ and $s_h(t)+\delta$ can only corresponds to $s_h(t)$. Conversely, these are the only bars that can match with some bar having $s_h(t)$ as one of its endpoints. Therefore, $B(H)$ has $2\times$(number of closed Reeb orbit with period $t$) many bars having endpoints of $s_h(t)$.
\end{proof}

We define another barcode entropy as follows: 
Let $H$ be a given convex radial Hamiltonian $H$ of slope $T$. 
Since $H$ is a convex radial Hamiltonian, there exist constants $r_0 >1$ and $C$ such that outside of $W_{r_0}$, $H= Tr -C$. 
Now, we choose a sequence $\{a_i\}_{i \in \N}$ satisfying 
\begin{itemize}
	\item $a_i$ satisfies that $\left(i-\tfrac{1}{2^i}\right) < a_i < \left(i+\tfrac{1}{2^i}\right)$, and
	\item $a_iT$ is not a period of any closed Reeb orbit of $\alpha$.
\end{itemize} 
Because of the second condition, $H_i := a_iH$ is a convex radial Hamiltonian function. 

Thanks to Lemma \ref{lem convex radial}, for all $i \in \mathbb{N}$, $B(H_i)$ is a persistence module of finite type. 
Thus, the following definition makes sense:
\begin{definition}
	\label{def barcode entropy of a Hamiltonian function}
	Let $\{H_i\}_{i \in \mathbb{N}}$ be the sequence of Hamiltonian functions which is constructed above.  
	\begin{enumerate}
		\item The {\bf $\boldsymbol{\epsilon}$-barcode entropy of the sequence $\boldsymbol{\{H_i\}}$}, denoted by $\boldsymbol{\hbar_\epsilon(\{H_i\})}$ is defined as 
		\[\hbar_\epsilon(\{H_i\}):= \limsup_{i \to \infty} \frac{1}{a_iT} \log n_\epsilon\left(\tru(B(H_i),a_iC)\right).\]
		\item The {\bf barcode entropy of the sequence $\{H_i\}$}, denoted by $\boldsymbol{\hbar(\{H_i\})}$, is defined as  
		\[\hbar(\{H_i\}):=\lim_{\epsilon \searrow 0} \hbar_\epsilon(\{H_i\}).\]
	\end{enumerate}
\end{definition}
We note that $\tru$ is defined in Definition \ref{def two operations} and $n_\epsilon$ is defined in Definiton \ref{def SH barcode entropy} (1).

\begin{remark}
	\label{rmk independency on the sequence}
	\mbox{}
	\begin{enumerate}
		\item We note that by Definition \ref{def barcode entropy of a Hamiltonian function},$\hbar(\{H_i\})$ seems dependent on the choice of $\{H_i\}$. 
		However, we will show that $\hbar(\{H_i\})$ is independent of $\{H_i\}$ because $\hbar(\{H_i\}) = \hbar^{SH}(W,\lambda)$, in the next section.
		\item We also note that, for any barcode $B$, if a bar $b$ in $\tru(B, T)$ has length $>\epsilon$, then the left end point of $b$ should be smaller than or equal to $T-\epsilon$. 
	\end{enumerate}
\end{remark}

\section{Comparison of barcode entropy}
\label{section comparison of barcode entropy}
We defined two different notions of barcode entropy in Section \ref{section two barcode entropy}.
In Section \ref{subsection equivalence of two barcode entropy},
we compare these two different barcode entropy and we show that they are the same.
The equivalence will be proven in Section \ref{subsection equivalence of two barcode entropy}, and we can call it {\em the} barcode entropy of symplectic homology. 
Moreover, we also prove that the barcode entropy of symplectic homology is an invariant of boundary contact manifold in Section \ref{subsection fillings}.
In the last subsection, we will compare our barcode entropy of symplectic homology with other barcode entropy of a Liouville domain, which are defined by applying a similar idea to the positive symplectic homology and the Rabinowitz Floer homology.

\subsection{Equivalence of two barcode entropy}
\label{subsection equivalence of two barcode entropy}
In this subsection, we prove that two barcode entropy defined in Definitions \ref{def SH barcode entropy} and \ref{def barcode entropy of a Hamiltonian function} coincide each other. 

In other words, we prove the following:
\begin{thm}
	\label{thm equivalence}
	In the setting of Section \ref{subsection setting and notations}, for any sequence of convex radial Hamiltonian functions $\{H_i\}_{i \in \N}$ defining $\hbar(\{H_i\})$, 
	\[\hbar(\{H_i\}) = \hbar^{SH}(W, \lambda).\]
\end{thm}
\begin{proof}
	Let $H=h(r)$ be a convex radial function with the linear part $Tr - C$. 
	Our strategy for proving Theorem \ref{thm equivalence} is to show that 
	\begin{gather}
		\label{eqn goal 0}
		n_{r_0\epsilon} \left(\tru(B(H),C)\right) \leq n_{\epsilon} \left(\tru(B_{SH}(W,\lambda),T)\right) \leq n_{\epsilon} \left(\tru(B(H),C)\right).
	\end{gather}
	
	If inequalities in \eqref{eqn goal 0} hold for any $H$, then one can apply the inequalities for each member of the sequence $\{H_i:=a_i H_1\}$ (see Section \ref{subsection barcode entropy of a Hamiltonian function} for the construction and notation of $\{H_i\}$.)
	Then, one has 
	\[n_{r_0\epsilon} \left(\tru(B(H_i),a_iC)\right) \leq n_{\epsilon} \left(\tru(B_{SH}(W,\lambda),T)\right) \leq n_{\epsilon} \left(\tru(B(H),a_i C)\right).\]
	By measuring the exponential growths of the above terms as $i$ increases, one has 
	\[\hbar_{r_0\epsilon}(\{H_i\}) \leq \hbar^{SH}_\epsilon(W, \lambda) \leq \hbar_\epsilon(\{H_i\}).\] 
	And, when one takes the limit as $\epsilon \searrow 0$, one can prove Theorem \ref{thm equivalence}.
	
	We note that \eqref{eqn goal 0} compares the number of bars in (truncations of) barcodes $B(H)$ and $B_{SH}(W,\lambda)$. 
	Since it is not easy to compare $B(H)$ and $B_{SH}(W,\lambda)$ directly, in our proof of \eqref{eqn goal 0}, we use four-steps argument. 
	In the firs step, we fix a sufficiently large $K$, then we compare $B_{SH}(W,\lambda)$ and $B(KH)$, instead of $B_{SH}(W,\lambda)$ and $B(H)$.
	In the second step, we choose another convex radial function $G$ acting like $H$ on $W_{r_0}$ and acting like $KH$ outside of $W_{r_0+1}$.
	Then, we compare $B(KH)$ and $B(G)$.
	The third step compares $B(G)$ and $B(H)$, and the last step proves \eqref{eqn goal 0} from the first three steps. 
	\vskip0.2in
	
	\noindent{\em Step 1.} For the first step, we recall that since $\alpha$ is a non-degenerate contact $1$-form, there exist finitely many closed Reeb orbits whose period $<T$. 
	Let the Reeb orbits of periods $<T$ have $\ell_1< \cdots < \ell_m$ as their periods, and let $\ell_{m+1}$ be the smallest period $>T$ of a closed Reeb orbit.

	

From Lemma \ref{def reparameterization function}, each convex radial Hamiltonian $H=h(r)$ has an assigned monotone increasing bijective function $s_h:[0,T]\to [0,C]$. For each $i\in \N$, we extend the reparametrization function of $a_iH$, $s_{a_i h}:[0,a_i T]\to [0,a_i C]$, to a strictly increasing bijective function $s_{i}:\R \to \R$ by defining $s_{i}(x)=x$ for all $x<0$ and $s_{i}(x)=x+a_i(C-T)$. By convexity of $h$, we have the followings:
\begin{itemize}
		\item $s_{i}(t)$ decreases as $i$ increases, and 
		\item $\lim_{i \to \infty}s_{i}(t) = t$.
	\end{itemize}
 
Thus, for any sufficiently large $k$, we have 
\[0 < \ell_1 < s_{k}(\ell_1) < \ell_2 < \dots < \ell_m < s_{k}(\ell_m) < T.\]
Moreover, for any $k$ satisfying the above, one can observe the following:
\begin{itemize}
	\item If $\ell_i < a \leq b \leq \ell_{i+1}$ (resp.\ $s_{k}(\ell_i) < a \leq b \leq s_{k}(\ell_{i+1})$,) $\iota^{a,b}$ (resp.\ $\iota_{H_k}^{a,b}$) is an isomorphism. 
In other words, $\tru(B_{SH}(W,\lambda),t)$ (resp.\ $\tru(B(H_k),s_{k}(t))$) satisfies the {\em lower semi-continuity} in the sense of \cite[Definition 1.1]{Polterovich-Rosen-Samvelyan-Zhang}. 
		\item For any $a >0$, $\SH^{<a}(W, \lambda) = \lim_{k \to \infty} \HF^{<a}(H_k)$ where the limit means the direct limit. 
		Moreover, thanks to Lemma \ref{lem convex radial}, $\SH^{<t}(W,\lambda) = \HF^{<t}(H_k)$ if $s_{k}(\ell_i) < t< \ell_{i+1}$ for some $i$.  
		\item For any $t_i$ such that $t_0 <0 < t_1 < \ell_1 < t_2< \ell_2 < \dots < t_{m+1} < \ell_{m+1}$, the following diagram commutes:
		\begin{equation}
			\label{equivalence: SH to HF}
			\begin{tikzcd}&\HF^{<s_{k}(t_0)}(H_k) \arrow{r}{\iota_{H_k}^{t_0,t_1}} \arrow{d}{\cong}& \HF^{<s_{k}(t_1)}(H_k) \arrow{d}{\cong}\arrow{r} &\cdots\arrow{r}\arrow{d} &\HF^{<s_{k}(t_m)}(H_k)\arrow{d}{\cong} \ar[r, "\cong"] \ar[d] & \HF^{<s_{k}(T)}(H_k) \arrow{d}{\cong}\\ &\SH^{<t_0}(W,\lambda) \arrow{r}{\iota_{\SH}^{t_0,t_1}} & \SH^{<t_1}(W,\lambda) \arrow{r} &\cdots \arrow{r} &\SH^{<t_m}(W,\lambda) \ar[r, "\cong"] & \SH^{<T}(W, \lambda)\end{tikzcd}.
		\end{equation} 
	\end{itemize}
	
	Now, we fix a sufficiently large $K$. 
	The above observations and Lemma \ref{lem convex radial} explain how $\tru(B(H_K),s_{K}(t))$ and $\tru(B_{SH}(W,\lambda),t)$ change as $t <T$ increases. 
	Moreover, one has 
	\begin{enumerate}
		\item[(A)] $\act\left(\tru\big(B(H_K),s_{K}(T)\big),(s_{K})^{-1}\right) = \tru\left(B_{SH}(W,\lambda),T\right)$.
	\end{enumerate} 
	\vskip0.2in
	
	\noindent{Step 2.} Now, we would like to compare $B(H)$ and $B(H_K)$ for the fixed $K$ in the first step. 
	For the comparison, we choose another convex radial Hamiltonian function $\bar{H}=\bar{h}(r)$ such that 
	\begin{itemize}
		\item $\bar{H} \equiv H$ on $W_{r_0+1/2}$,
		\item $\bar{H}(r,p)= H_K(r,p) - C$ on $r \ge r_0+1$, for some constant $C>0$ and 
		\item $\bar{H} \leq H_K$ for any point on $\widehat{W}$.
	\end{itemize}
	
	Let $H_t := (1-t) \bar{H} + H_k$ for $t \in [0,1]$. 
	Then, $H_t \equiv (1-t)\bar{h} + t a_K h$ on $\partial W \times [1,\infty)_r$ and one can observe the following:
	\begin{itemize}
		\item[(i)] From Lemma \ref{lem convex radial}, one can find a small positive number $\delta_0$ such that every bar in $B(H_t)$ has length $> 2 \delta_0$. 
		We note that $\delta_0$ depends on the slope of $H_t$, i.e., $a_K T$ for now. 	
		\item[(ii)] There exists $\delta>0$ such that if $|t_1 - t_2| <\delta$, then 
		\[\|H_{t_1} - H_{t_2}\|_{\infty,W_{r_0+1}} < \delta_0.\]
		\item[(iii)] For any $t_1, t_2$, 
		\[d_{bot}\left(B(H_{t_1}), B(H_{t_2})\right) = d_{int}\left(B(H_{t_1}), B(H_{t_2})\right) \leq \|H_{t_1} - H_{t_2}\|_{\infty,W_{r_0+1}}.\]
		The first equality holds because of the Isometry Theorem (Theorem \ref{thm isometry theorm}), and the inequality holds because of Lemma \ref{lem sup norm}. 
		We note that Lemma \ref{lem sup norm} is proven for non-degenerate Hamiltonian functions.
		However, it is not hard to prove Lemma \ref{lem sup norm} for the convex radial Hamiltonian functions.
		\item[(iv)] If $d_{bot}\left(B(H_{t_1}), B(H_{t_2})\right) <\delta_0$, then there exists a $\delta_0$-matching from $B(H_{t_1})$ to $B(H_{t_2})$. 
		In other words, there is a matching from $B(H_{t_1})_{2\delta_0}$ to $B(H_{t_2})_{2\delta_0}$. 
		We note that by (i), $B(H_{t_i})_{2\delta_0} = B(H_{t_i})$ for $i=1, 2$.
		Moreover, 
		\[\act\left(B(H_{t_1}), (s_{t_1})^{-1}\right) = \act\left(B(H_{t_2}), (s_{t_2})^{-1}\right),\]
		where $s_{t_i}$ is the extension of reparametrization function for $(1-t_1)\bar{h} + t_1 Kh$ and $(1-t_2)\bar{h} + t_2 Kh$.
	\end{itemize}
	
	We choose $N \in \N$ such that $\tfrac{1}{N} < \delta$.
	Then, we compare $H_{\tfrac{i}{N}}$ and $H_{\tfrac{i+1}{N}}$ for $i = 0, \dots, N-1$. 
	As a result, we have 
	\begin{enumerate}
		\item[(B)] $\act\left(B(H), s_g^{-1}\right) = \act\left(B(H_k), s_{k}^{-1}\right).$
	\end{enumerate}
	\vskip0.2in
	
	\noindent{\em Step 3.} We compare $B(H)$ and $B(\bar{H})$. For this purpose, we strongly control the Morse perturbation $(f,g)$ and $(\bar{f}, \bar{g})$ and subset of almost complex structures $\set{J_\delta}, \set{\bar{J}_\delta}$ to satisfies the following:
    \begin{itemize}
        \item $f=\bar{f}$ in $\widehat{W}$, $\bar{g}-g$ is supported in $(r_0+1/2, r_0+1)\times Y$, $J_\delta=\bar{J}_\delta$ in $S^1\times W_{r_0+1/2}$,
        \item $J_\delta$ is $H_\delta=H+\delta(f+g)$-regular, $\bar{J}_\delta$ is $\bar{H}_\delta=\bar{H}+\delta(\bar{f}+\bar{g})$-regular, and
        \item $\set{J_\delta},\set{\bar{J}_\delta}$ have common limit $J\in\mathcal{J}_{r_0}^{\text{SFT}}$ on $C^\infty$-topology.
    \end{itemize}

    By the property of Morse perturbation, for sufficiently small $\delta$, the Floer chain complex $\CF(H_\delta)$ is identical to the sub-complex $\CF^{<s_h(T)}(\bar{H}_\delta)$ consists of the 1-periodic orbits with action less than $s_h(T)=s_{\bar{h}}(T).$ Since $H$ is linear in $r\ge r_0$, the maximal principle implies that the Floer cylinder for the Floer-data $(H_\delta, J_\delta)$ does not intersect the hypersurface $r=r_0$.
    
    Conversely, we claim the Floer cylinder for the Floer data $(\bar{H}_\delta, \bar{J}_\delta)$ stays does not intersect the hypersurface $r=r_0$ for sufficiently small $\delta$. Assume there exists a sequence of real numbers $\set{\delta_i}_{i \in \N}$ converging to zero and a sequence of Floer cylinder $\set{u_i}_{i \in \N}$ where each $u_i$ intersects $r=r_0$. After translation, we assume $u_i(0,-)$ also intersects $r=r_0$. By Gromov-Floer compactness, $u_i$ converges to some Floer cylinder $u$ solving the Floer equation for $(\bar{H}, J)$ in $C^\infty_{\text{loc}}$-topology. Since we assumed $u_i(0,-)$ intersects $r=r_0$ for every $i$, $u(0,-)$ intersects $r=r_0$.

    Since $\bar{H}$ is a radial function, the maximum principle holds for $u$ and there exists a sequence $\set{s_k}$ with $s_k\to \infty$ such that $u(s_k, t_k)$ has $r$-coordinate at least $r_0$. Note that $\mathcal{A}_H(u(s_k,-))=\lim \mathcal{A}_H(u_i(s_k, -))<s_h(T)$. Following \cite[Proposition 6.5.7.]{Audin-Damian14}, $u(s_k, -)$ converges to some 1-periodic orbit of $X_{\bar{H}}$ in $C^\infty$ topology. From the limit, this periodic orbit also has action at most $s_{\bar{h}}(T)$ and therefore contained in $r<r_0$. This contradicts that $u(s_k,-)$ is not contained in $r<r_0$ for all $k\in\N$, and the Floer cylinder stays in $r\le r_0$ for sufficiently small $\delta$.
    
    Since the Floer equations for two Floer-data $(H_\delta, J_\delta), (\bar{H}_\delta, \bar{J}_\delta)$ agree on $r\le r_0$, the differentials of $\CF(H_\delta, J_\delta)$ and $\CF^{<s_h(T)}(\bar{H}_\delta, \bar{J}_\delta)$ agree for sufficiently small $\delta$. Taking a limit, we conclude
	\begin{enumerate}
		\item[(C)] $\tru\left(B(H),s_h(T)\right) = \tru\left(B(\bar{H}),s_{\bar{h}}(T)\right)$.
	\end{enumerate}
	We note that $H \equiv \bar{H}$ on $W_{r_0}$. 
	\vskip0.2in
	
	\noindent{\em Step 4.} From (A)--(C), we have the following:
	\begin{align*}
		\tru\left(B_{SH}(W,\lambda),T\right) =& \act\left(\tru\big(B(H_K),s_{K}(T)\big), (s_{K})^{-1}\right) \text{  ($\because$ (A))} \\
		=& \tru\left(\act\big(B(H_K), (s_{K})^{-1}\big), T\right) \text{  (By Definition \ref{def two operations})} \\
		=& \tru\left(\act\big(B(\bar{H}), (s_{\bar{h}})^{-1}\big), T\right)  \text{  ($\because$ (B))} \\
		=& \act\left(\tru \big(B(\bar{H}),s_{\bar{h}}(T)\big), s_{\bar{h}}^{-1}\right) \text{  (By Definition \ref{def two operations})} \\
		=& \act\left(\tru \big(B(H),s_h(T),s_h^{-1}\big)\right) \text{  ($\because$ (C))} \\
		=& \tru\left(\act \big(B(H),s_h^{-1}\big), T\right) \text{  ($\because s_{\bar{h}}= s_h$ on $[0,T]$)} \\
		=& \act\left(\tru \big(B(H), s_h(T) =C\big), s_h^{-1}\right) \text{  (By Definition \ref{def two operations})}.
	\end{align*}
	Thus, there is one-to-one relation between bars of $\tru \big(B(H), C\big)$ and $\tru\left(B_{SH}(W,\lambda),T\right)$.
	Moreover, the lengths of a bar in $\tru \big(B(H), C\big)$ and the corresponding bar in $\tru\left(B_{SH}(W,\lambda),T\right)$ are different by $s_h^{-1}$. 
	
	We note that as we observed in the first step, every bars in $\tru\left(B_{SH}(W,\lambda),T\right)$ has its endpoints at either $0, \infty$ or a period of a closed Reeb orbit.
	Let us assume that there exists a bar $\mathbb{b}$ in $\tru\left(B_{SH}(W,\lambda),T\right)$ whose the right end point is $\ell_j$ and the left end point is $\ell_i$ 
	Then, the bar in $\tru \big(B(H), C\big)$ corresponding to $\mathbb{b}$ has two end points at $s_h(\ell_j)$ and $s_h(\ell_i)$. By inequalities in \eqref{eqn reparametrization function difference}, the length of the bar $(s_h(\ell_i), s_h(\ell_i)]$ satisfies
	\[\ell_j - \ell_i \leq s_h(\ell_j) - s_h(\ell_i) \leq  r_0(\ell_j - \ell_i).\]
	It completes the proof of \eqref{eqn goal 0}.
\end{proof}

From Theorem \ref{thm equivalence}, one can observe that the barcode entropy of a sequence of Hamiltonian functions, $\hbar(\{H_i\})$, is independent of the choice of sequence.
In the rest of the paper, we simply use the term the {\em $\SH$-barcode entropy} to indicate either of $\hbar^{SH}(W, \lambda)$ or $\hbar(\{H_i\})$.

\subsection{$\SH$-Barcode entropy of different fillings}
\label{subsection fillings}
The $\SH$-barcode entropy $\hbar^{SH}(W,\lambda)$ is defined to be an invariant of a Liouville domain $(W,\lambda)$. 
In this subsection, we prove that $\hbar^{SH}(W,\lambda)$ is an invariant of the boundary contact manifold and its non-degenerate contact $1$-form $\alpha:= \lambda|_\partial W$. 
More precisely, we prove Theorem \ref{invariance: main thm}.

\begin{thm}
	\label{invariance: main thm}
    Let $(Y,\alpha)$ be a non-degenerate contact manifold with Liouville fillings  $(W_1, \lambda_1)$ and $(W_2, \lambda_2)$ whose first Chern classes vanish. Then, the $\SH$-barcode entropy of $(Y,\alpha)$ defined by the symplectic homology of $(W_1, \lambda_1)$ and that of $(W_2, \lambda_2)$ coincides, i.e,
    \[\hbar^{\SH}(W_1, \lambda_1)=\hbar^{\SH}(W_2, \lambda_2).\]
    Moreover, there exists a constant $C(Y,\alpha)$ depending on $(Y, \alpha)$ such that, for any $\epsilon \le C(Y)$, the following holds:
	\[\hbar_\epsilon^{\SH}(W_1, \lambda_1)=\hbar_\epsilon^{\SH}(W_2, \lambda_2).\]
\end{thm}

%

\begin{remark}
	We note that for a filling $(W,\lambda)$ of $(Y,\alpha)$, one can find a (closed) collar neighborhood of $\partial W$, $N(\partial W)$ such that 
	\[\left(N(\partial W), \lambda|_{N(\partial W)}\right) \simeq \left([\tfrac{1}{2}, 1]_r \times Y, r\alpha\right).\]
	As usual, the subscript $r$ in $[\tfrac{1}{2},1]_r$ means the coordinate of the factor.
	For more details, we refer the reader to \cite[Chapter 11]{Cieliebak-Eliashberg12}.
	In the rest of this subsection, we will identify the collar neighborhood of $\partial W$ with $\left([\tfrac{1}{2}, 1]_r \times Y, r\alpha\right)$.
\end{remark}

Theorem \ref{invariance: main thm} is a corollary of Proposition \ref{invariance: short bars SH vers} that we state and prove below. 
In order to state Proposition \ref{invariance: short bars SH vers}, let us define the following notation first. 
\begin{definition}
	Let $C$ be a positive real number. 
	For a barcode $B$, let $\boldsymbol{B^{<C}_+}$ denote the set of bars in $B$ satisfying that
	\begin{itemize}
		\item the length of the bar is at most than $C$, and 
		\item the left-endpoint of the bar is bigger than $0$. 
	\end{itemize} 
\end{definition}

\begin{prop}
	\label{invariance: short bars SH vers}
	In the setting of Theorem \ref{invariance: main thm}, one can find a constant $C(Y,\alpha)>0$ such that
	\[B_{\SH}(W_1, \lambda_1)^{<C(Y,\alpha)}_+=B_{\SH}(W_2, \lambda_2)^{<C(Y,\alpha)}_+.\]
\end{prop}

Before proving Porposition \ref{invariance: short bars SH vers}, we note that the constant $C(Y, \alpha)$ in Proposition \ref{invariance: short bars SH vers} will be defined as 
\[C(Y, \alpha):=\sup_{J \in \mathcal{J}(Y,\alpha)}C(J),\]
where $\mathcal{J}(Y,\alpha)$ denotes the set of $(d\alpha)$-compatible almost complex structure on $\ker\alpha$ and $C(J)$ is given by the following lemma:

\begin{lem}[Lemma 1 of \cite{Oancea06}]
	\label{Invariance: Gromov monotonicity}
	For any $J\in \mathcal{J}(Y, \alpha)$, there exists a constant $C(J)>0$ such that, if $u$ is a $\hat{J}$-holomorphic curve having its boundary components on both of $\{\tfrac{1}{2}\} \times Y$ and $\{1\} \times Y$, then, $E(u) \ge C$.
\end{lem}
We note that we allow $C(Y,\alpha)$ to be infinite. 
In that case, Proposition \ref{invariance: short bars SH vers} means $B_{\SH}(W_1,\lambda_1) = B_{\SH}(W_2, \lambda_2)$.

\begin{proof}[Sketch of the proof of Proposition \ref{invariance: short bars SH vers}]
	We sketch the proof of Proposition \ref{invariance: short bars SH vers} here. 
	The detailed proof will appear after proving Lemmas \ref{Invariance: Crossing Energy}--\ref{invariance: short bars are preserved General}.
	
	First, we will prove a Hamiltonian Floer homology version of Proposition \ref{invariance: short bars SH vers}.
	More precisely, for a pair of convex radial Hamiltonian functions $H_j:\widehat{W} \to \mathbb{R}$ of slope $T$ such that 
	\[H_1|_{[1,\infty) \times \partial W_1} \equiv H_2|_{[1,\infty) \times \partial W_2}=h(r),\]
	as functions defined on 
	\[[1,\infty) \times \partial W_1 \simeq [1,\infty) \times Y \simeq [1,\infty) \times \partial W_2,\]
	we will prove that 
	\begin{gather}
		\label{eqn invariance 1}
		B(H_1)^{<C(Y,\alpha)}_+ = B(H_2)^{<C(Y,\alpha)}_+.
	\end{gather}
	
	After proving the Hamiltonian Floer homology version of it, we can prove Proposition \ref{invariance: short bars SH vers} in a similar way we did in Section \ref{subsection equivalence of two barcode entropy}. 
\end{proof}

As mentioned in the above sketch, we first consider a pair of convex radial functions $H_j:\widehat{W}_j \to \R$ agrees to each other on the cylindrical end part $[1, \infty) \times Y$. 
In order to handle their Hamiltonian Floer homology, we choose their Morse-perturbation data $(f_j, g_j)$ for $j = 1, 2$ as follows: 
We choose $(f_j,g_j)$ such that $f_1 \equiv f_2$ and $g_1 \equiv g_2$ on $[\tfrac{1}{2}, \infty) \times Y$. 
Then, for all sufficiently small $\delta >0$, there is a one-to-one relation between the set of non-constant periodic orbits of Hamiltonian vector fields of $H_1 + \delta (f_1+g_1)$ and $H_2 + \delta(f_2 + g_2)$. 
Moreover, the actions values of a pair of periodic orbits which are related to each other by the above one-to-one relation should agree.

Let us fix a sequence of positive real number $\{\delta_i\}_{i \in \N}$ such that $\delta_i$ converges to $0$ as $i \to \infty$. 
We also fix two sequence of almost complex structures 
\[\left\{J_{1,i} \in \mathcal{J}^{\text{SFT}}_{r_0}(W_1,d\lambda_1)\right\}_{i \in \N}, \left\{J_{2,i} \in \mathcal{J}^{\text{SFT}}_{r_0}(W_2,d\lambda_2)\right\}_{i \in \N},\] 
such that   
%
%
%
%
\begin{itemize}
	\item for each $i \in \N$, the pair $\left(H_j+\delta_i(f_j+g_j), J_{j,i}\right)$ is regular,
	\item $J_{1,i} = J_{2,i}$ in ${S^1\times[1/2, \infty)\times Y}$ for every $i\in \N$, and
	\item there exists  $J_{j,\infty} \in \mathcal{J}^{\text{SFT}}_{r_0}(W_j,d\lambda_j)$ such that $\set{J_{j,i}}_{i\in N}$ converges to $J_{j,\infty}$, and such that $J_{j,\infty}(t,r,p)\equiv\hat{J}_0(r,p)$ if $r \ge \tfrac{1}{2}$. 
\end{itemize}    
Together with the above chosen choices, we state and prove Lemma \ref{Invariance: Crossing Energy}.

\begin{lem}
	\label{Invariance: Crossing Energy}
	For any chosen data satisfying the above conditions, there exists a natural number $N$ and a constant $C(J_0)>0$ satisfying the following:
	If $i >N$, then every Floer cylinder $u$ satisfying that 
	\begin{itemize}
		\item $u$ connects non-constant periodic orbits of $H_j+\delta_i(f_j + g_j)$, and 
		\item $u$ intersects $\{\tfrac{1}{2}\} \times Y$,
	\end{itemize}
	should have energy bigger than or equal to $C(J_0)$, i.e., 
	\[E(u) > C(J_0).\]
\end{lem}
\begin{proof}
	Let us assume that there exist infinitely many $i \in \N$ such that there exists a Floer cylinder $u_i$ satisfying 
	\begin{itemize}
		\item $u_i$ connects non-constant periodic orbits of $H_j+\delta_i(f_j + g_j)$,  
		\item $u_i$ intersects $\{\tfrac{1}{2}\} \times Y$, and
		\item $E(u_i) \le C(J_0)$.  
	\end{itemize}
	By a suitable translation, we can assume that $u_i(0,S^1)$ intersects $\{\tfrac{1}{2}\} \times Y$ without loss of generality. 
	By the Gromov-Floer compactness, there exists a subsequence of $u_i$ converging to a Floer cylinder $u_\infty\in\mathcal{M}(x,y;H,J_\infty)$.
	Then, one can check that $u_\infty(0,S^1)$ intersects $\{\tfrac{1}{2}\} \times Y$.
	
	We note that $u_\infty$ is a $J_\infty$-holomorphic curve.
	On the region $[\tfrac{1}{2}, 1) \times Y$, the ($S^1$-family of) almost complex structure $J_\infty$ agrees with $\hat{J}_0$. 
	Thus, $u_\infty$ is a $\hat{J}_0$-holomorphic curve on the region.
	In other words, $u_\infty|_{u_\infty^{-1}([1/2,1] \times Y)}$ is a $\hat{J}_0$-holomorphic curve satisfying the conditions of Lemma \ref{Invariance: Gromov monotonicity}. 
	Thus, $E(u|_{u^{-1}([1/2,1]\times Y)}) > C(\hat{J}_0)$ by Lemma \ref{Invariance: Gromov monotonicity}, and it contradicts that $E(u) = \lim_{i \to \infty} E(u_i) \le C(\hat{J}_0)$.
	It completes the proof. 
\end{proof}

In $[1/2,\infty)\times Y$, the Floer equation for $(H_1+\delta_i(f_1+g_1),J_{1,i})$ is identical with the Floer equation for $(H_2+\delta_i(f_2+g_2),J_{2,i})$. 
Therefore, by Lemma \ref{Invariance: Crossing Energy}, for any sufficiently large $i$, the $\left<\partial_{j,i} x, y\right>$ is independent of $j$ for two Floer boundary maps, if the Hamiltonian action difference between $x$ and $y$ is smaller than $C=C(J_0)$.

\begin{lem}
	\label{invariance: short bars are preserved General}
	Let $(V_1, \ell_1), (V_2,\ell_2), (W,\ell_W)$ be orthogonalizable $\mathbb{k}$-spaces. Consider $(C=V_1 \oplus W, \ell_C=\ell_1 \oplus \ell_W)$(resp. $(D=V_2 \oplus W, \ell_D=\ell_2 \oplus \ell_W)$) endowed with a linear map $\partial_C:C\to C$(resp. $\partial_D:D\to D$) such that $\partial_C^2=0$(resp. $\partial_D^2=0$) and $\partial_C$(resp. $\partial_D$) decreases the action, where 
	\[\ell_i \oplus \ell_W(a,b) = \max \{\ell_i(a), \ell_W(b), \text{  for all  } i = 1, 2.\}\] 
	Let us assume that there there exists a constant $\eta>0$ and $E>0$ such that
	\begin{enumerate}
		\item $\ell_j(v_j)\le E<\ell_W(w)$ for all $w\in W$, $v_j \in V_j$ for $j=1,2$
		\item and $\ell_W(\pi_W (\partial_C w) - \pi_W (\partial_D w))<\ell_W(w)-\eta$ for all $w\in W$. ($W$ is viewed as a subspace of $V_i\oplus W$ and $\pi_W$ means the natural projection to $W$ from $V_i \oplus W$.)
	\end{enumerate}
	Denote $\mathbb{B}_C=\set{x_1^C, \cdots, x_{m_1}^C, y_1^C,\cdots y_{n_1}^C, z_1^C,\cdots, z_{n_1}^C}$ the orthogonal basis of $C$ such that $\partial_C x^C_i=0$, $\partial_C z^C_j=y^C_j$. Let $C^{<\eta}_E$ be the multi-set of intervals $(\ell_C(y^C_j), \ell_C(z^C_j)]$ such that $\ell_C(z^C_j)-\ell_C(y^C_j)<\eta$ and $E < \ell_C(y^C_j)$. Define multi-set $D^{<\eta}_E$ in the same way. Then $C^{<\eta}_E=D^{<\eta}_E$.
\end{lem}
\begin{proof}
	Let $j=i_1, i_2, \cdots, i_k$ be the bars satisfying $\ell_C(z^C_j)-\ell_C(y^C_j)<\eta$ and $E<\ell_C(y^C_j)$. For each $z^C_j$ and $y^C_j$, let $z^W_j$ and $y^W_j$ be the projection on $W$. By orthogonality and assumption (1), the vector space $V_1$ is generated by the elements of $\mathbb{B}_C$ with an evaluation value less than $E$. With some basis change, we now assume $z^C_j$ are the vectors in $W$ and still $\partial_C z^C_j =y^C_j$ with $\ell_C(z^C_j)-\ell_C(y^C_j)<\eta$ for $j=i_1,i_2,\cdots,i_k$.
	
	Let $z_j^D:=z_j^C$ as an element in $W\subset V_2 \oplus W$ and let $y_j^D:=\partial_D z_j^D$. Since $z_j^D$ is an element in $W$, $\ell_D(z_j^D)=\ell_W(z_j^D)=\ell_C(z_j^C)$. We also claim that $\ell_D(y_j^D)=\ell_C(y_j^C)$. More generally, consider $w\in \left< z^C_{i_1}, z^C_{i_2}, \cdots, z^C_{i_k}\right>\subset W$. Then $\partial_C(w)$ is a linear combinations of $y^C_j$ with $j=i_1,i_2,\cdots,i_k$. By the orthogonality and the assumption on $z^C_j$, $\ell_C(\partial_C(w))>\ell_C(w)-\eta$. By assumption (1), $\ell_C(\partial_C(w))=\ell_W(\pi_W\circ(\partial_C w))$. By assumption (2) and the orthogonality, $\ell_C(\partial_C(w))=\ell_D(\partial_D(w))$ which is greater than $E$. Especially, $\ell_D(y_j^D) = \ell_C(y_j^C)>E$.
	
	Since $\left< y_{i_1}^D, y_{i_2}^D,\cdots, y_{i_k}^D\right>$ is the subspace of $\ker \partial_D$, it is enought to show that $\set{z_j^D, y_j^D=\partial z_j^D:j=i_1, i_2, \cdots, i_k}$ is orthogonal with respect to $\ell_D$. Extending this orthogonal set, we can find an orthogonal basis of $D$ containing every $x^D_{i_j}, y^D_{i_j}$. Thus the number of bars in the statement for $D$ is at least $k$. Applying this from $D$ to $C$, we get the equality for two numbers.
	
	\noindent\textit{Orthogonality for $\set{z_j^D, y_j^D=\partial z_j^D:j=i_1, i_2, \cdots, i_k}$ with respect to $\ell_D$}: Every vector in the subspace spanned by this set of vectors is in form of $w_1+\partial_D(w_2)$ where $w_i$ is the vector in the subspace of $W$ spanned by the set $\set{z_j^D:j=i_1,i_2,\cdots,i_k}$. To prove the orthogonality, it is enough to prove the following equation:
	\begin{equation}
		\label{equation: orthogonality for D}
		\ell_D(w_1+\partial_D(w_2)) = \max(\ell_D(w_1), \ell_D(\partial_D(w_2))).    
	\end{equation}
	
	\noindent If the equation \ref{equation: orthogonality for D} holds, for $v=\sum_{n\in I} a_n \cdot z_{i_n}^D + \sum_{m\in J} b_m \cdot y_{i_m}^D$ with non-zero $a_n$ and $b_m$,
	\begin{align*}
		\ell_D(v)
		&= \max\left(\ell_D\left(\sum_{n\in I} a_n \cdot z_{i_n}^D\right), \ell_D\left(\partial_D \sum_{m\in J} b_m \cdot z_{i_m}^D\right)\right)\\
		&= \max\left(\ell_C\left(\sum_{n\in I} a_n \cdot z_{i_n}^C\right), \ell_C\left(\partial_C \sum_{m\in J} b_m \cdot z_{i_m}^C\right)\right)\\
		&= \max\left(\ell_C\left(\sum_{n\in I} a_n \cdot z_{i_n}^C\right), \ell_C\left(\sum_{m\in J} b_m \cdot y_{i_m}^C\right)\right)\\
		&= \max\left(\max_{n\in I}\set{\ell_C(z_{i_n}^C)}, \max_{m\in J}\set{\ell_C(y_{i_n}^C)}\right)\\
		&= \max\left(\max_{n\in I} \set{\ell_D(z_{i_n}^D)}, \max_{m\in J}\set{\ell_D(y_{i_n}^D)}\right).
	\end{align*}
	When $\ell_D(w_1)\ne \ell_D(\partial_D w_2)$, equation \ref{equation: orthogonality for D} is from the assumption that $D$ is orthogonalizable. If there exists $w_1, w_2$ with $\ell_D(w_1) = \ell_D(\partial_D w_2)=e$ and $\ell_D(w_1+\partial_D(w_2))<e$, 
	\[\ell_W (\pi_W\circ (w_1+\partial_D w_2))\le \ell_D(w_1+\partial_D(w_2))<e.\]
	Then by the orthogonality of $W$, 
	\begin{align*}
		e &= \max(\ell_W (\pi_W\circ (w_1+\partial_D w_2)), \ell_W (\pi_W\circ (w_1+\partial_C w_2)))\\
		&= \ell_W (\pi_W (w_1+\partial_D w_2)-\pi_W (w_1+\partial_C w_2))\\
		&= \ell_W (\pi_W (\partial_D w_2) - \pi_W (\partial_C w_2))< \ell_W(w_2) - \eta\quad(\because \text{ Assumption (2)})\\
		&< \ell_C(\partial_C(w_2))=\ell_D(\partial_D(w_2))=e,
	\end{align*}
	and we get the contradiction and the set $\set{z_j^D, y_j^D=\partial z_j^D:j=i_1, i_2, \cdots, i_k}$ is orthogonal with respect to $\ell_D$.
\end{proof}

Now we prove Proposition \ref{invariance: short bars SH vers}.

\begin{proof}[Proof of Proposition \ref{invariance: short bars SH vers}]
	As mentioned in the sketch of the proof, we start with a pair of convex radial Hamiltonian functions $H_j: \widehat{W}_j \to \R$ of the same slope $T$ for $j=1, 2$.
	In other word, we first prove Equation \eqref{eqn invariance 1}.
	
	Let us wrap up the following facts from Section \ref{subsection barcode entropy of a Hamiltonian function}: 
	From Lemma \ref{lem convex radial} and its proof, we can check that, for any sufficiently large $i$, the barcode $B\left(H_j + \delta_i(f_j+g_j)\right)$ satisfies that
	\begin{itemize}
		\item every bar is in the form of $(a,b]$ or $(a,\infty)$,
		\item if a bar has an endpoint greater than $\delta_i$, then the end point should be contained in 
		\[s_h\left(\Spec^{<T}(R_\alpha)\right)\cup \left(s_h\left(\Spec^{<T}(R_\alpha)\right)+\delta\right),\]
		\item there are exactly $\dim(H^*(W))$-many bars containing $(\delta_i, s_h(T_{\min})]$.
		\item for any $t \in \Spec^{<T}(R_\alpha)$, the number of bars having $s_h(t)$ as an endpoint is the same as the number of Reeb orbits of period $t$,
		\item for any $t \in \Spec^{<T}(R_\alpha)$, the number of bars having $s_h(t)+\delta_i$ as an endpoint is the same as the number of Reeb orbits of period $t$,
		\item (Lemma \ref{convex radial: no differential}) there is no bar in the form of $(s_h(t), s_h(t)+\delta_i]$, and
		\item the number of bars being contained $(0,\delta_i]$ is the same as 
		\[\left(\text{the number of critical points of  } f- \dim(H^*(W))\right)/2.\]
	\end{itemize}
	We also proved that the barcode $B(H_j)$ has similar properties because $B(H_j)$ is defined as the direct limit of $B\left(H_j+\delta_i (f_j+g_j)\right)$ as $i \to \infty$.
	More precisely, $B(H_j)$ satisfies that
	\begin{itemize}
		\item every bar is in the form of $(a,b]$ or $(a,\infty)$, where
		\[a, b \in s_h(\set{0}\cup \Spec^{<T}(R_\alpha)),\]
		\item the number of bars whose left-endpoints are $0$ is the same as $\dim(H^*(W))$, and
		\item for any $t \in \Spec^{<T}(R_\alpha)$, the number of bars having $s_h(t)$ as an endpoint is the same as 
		\[2\cdot \text{(the number of Reeb orbits of period  } t).\]
	\end{itemize}
	
	We also note that by choosing sufficiently small $f_j$ and $g_j$ (in the sense of $\|\cdot\|_{\infty, W_{r_0}}$-norm), one can assumption that $\delta_i \|f_j + g_j\|_{\sup, W_{r_0}} \leq \delta_i$ without loss of generality.
	Under the assumption, one can find s a $\delta_i$-matching $\mu_{j,i}: B(H_j) \to B\left(H_j+\delta_i(f_j+g_j)\right)$. 
	Let $i$ be sufficiently large so that $2 \delta_i$ is smaller then the shortest bar in $B(H_j)$. 
	We note that $B(H_j)$ has finitely many bars, one can find the shortest bar in $B(H_j)$. 
	Therefore, the image of $\mu_i$ is $B(H_j)$. 
	Moreover, let us assume that $3\delta_i$ is smaller than 
	\[\min\{|s_h(t_1)-s_h(t_2)|, s_h(T_{\min}) | t_1, t_2 \in \Spec^{<T}(R_\alpha)\}.\]
	Then, every bar in the form of $(s_h(t_1), s_h(t_2)]$ (resp.\ $(s_h(t), \infty)$) in $B(H_j)$ is matched with one of 
	\[(s_h(t_1),s_h(t_2)], (s_h(t_1),s_h(t_2)+\delta_i], (s_h(t_1)+\delta_i,s_h(t_2)], \text{  or  } (s_h(t_1)+\delta_i,s_h(t_2)+\delta_i] (\text{  resp.\ } (s_h(t), \infty) \text{  or  } (s_h(t)+\delta_i, \infty))\] 
	in $B\left(H_j+\delta_i(f_j+g_j)\right)$. 
	Every bar in one of the above forms is longer than $2\delta_i$, thus all bars belong to the co-image of $\mu_i$. 
	%
	
	In order to prove Equation \eqref{eqn invariance 1}, let us choose a bar from $B(H_j)^{<C(J_0)}_+$.
	The above arguments guarantee that for any sufficiently large $i$, there is a corresponding bar in $B\left(H_j +\delta_i (f_j + g_j)\right)$. 
	Moreover, since the length of the corresponding bar is different from that of the original bar in $B(H_j)^{<C(J_0)}_+$ by $\delta_i$, for any sufficiently large $i$, the corresponding bar is also shorter than $C(J_0)$. 
	
	Now, we apply Lemma \ref{invariance: short bars are preserved General} for a sufficiently large $i$, in the following set up:
	\begin{itemize}
		\item $V_j$: the vector space generated by the critical points of the Morse function $f_j$.
		\item $\ell_j$: the Hamiltonian action functional of $\cA_{H_j+\delta_i(f_j+g_j)}$, which is $\delta_i f(p)$ for the constant solution $p \in \text{Crit}f_j$.
		\item $W$: the vector space generated by non-constant periodic orbits of $H_j+\delta_i(f_j+g_j)$.
		\item $\ell_W$: the Hamiltonian action functional of $\cA_{H_j+\delta_i(f_j+g_j)}$, which is given by $\ell_W(\check\gamma)=s_h(t_\gamma)+\delta$, $\ell_W(\hat\gamma)=s_h(t_\gamma)$,
		where $\gamma$ denote the Reeb orbit with period $t_\gamma<T$.
		\item $E=\delta_i$ and $\eta=C(J_0)$.
	\end{itemize}
	We note that the second assumption of Lemma \ref{invariance: short bars are preserved General} holds because of Lemma \ref{Invariance: Crossing Energy} and the arguments following Lemma \ref{Invariance: Crossing Energy}.
	
	By Applying Lemma \ref{invariance: short bars are preserved General}, one can conclude that 
	\[B\left(H_1 + \delta_i (f_1+g_1)\right)_{\delta_i}^{<C(J_0)} = B\left(H_2 + \delta_i (f_2+g_2)\right)_{\delta_i}^{<C(J_0)},\]
	where $B\left(H_j + \delta_i (f_j+g_j)\right)_{\delta_i}^{<C(J_0)}$ is the set of bar in B$\left(H_j+\delta_i(f_j +g_j)\right)$ such that 
	\begin{itemize}
		\item the length of the bar is shorter than $C(J_0)$, and 
		\item the left-endpoint of the bar is bigger or equal to $\delta_i$. 
	\end{itemize}
	
	We note that, by applying $\delta_i$ matching $\mu_{j,i}$ for $j =1,2$, one concludes 
	\[B(H_j)^{<C(J_0)}_+=\mu_{j,i}\left(B\left(H_j + \delta_i (f_j+g_j)\right)_{\delta_i}^{<C(J_0)}\right).\]
	Thus, as multi-sets of intervals, 
	\[B(H_1)^{<C(J_0)}_+ = B(H_2)^{<C(J_0)}_+.\]
	
	So far, we proved Equation \eqref{eqn invariance 1}. 
	Proposition \ref{invariance: short bars SH vers} can be proven from Equation \eqref{eqn invariance 1}, as follows:
	We recall that every bar in $B_{\SH}(W_j,\lambda_j)^{C(J_0)}_+$ is the form of $(t_1,t_2]$ with $t_1, t_2 \in \Spec(R_\alpha)$. 
	For a fixed pair $t_1, t_2 \in \Spec(R_\alpha)$, let us assume that $B_{\SH}(W_j,\lambda_j)^{C(J_0)}_+$ has $n_j$-many $(t_1,t_2]$ bars.
	We note that because of non-degeneracy of $\alpha$, $n_j < \infty$. 
	
	We choose a convex radial Hamiltonian functions $H_j :\widehat{W}_j \to \R$ for $j =1,2$ satisfying the followings:
	\begin{itemize}
		\item The slope $T$ is bigger than $t_2$ and is not a period of any closed Reeb orbits.
		\item $H_j$ is linear outside of $W_{r_0}$ for some $r_0 < \frac{C(Y,\alpha)}{t_2 -t-1}$. 
		\item The pair $H_1$ and $H_2$ satisfies every condition given above, thus Equation \eqref{eqn invariance 1} holds.
	\end{itemize} 
	We recall that $C(Y,\alpha) = \sup_{J \in \mathcal{J}(Y,\alpha)}C(J)$.
	
	One can easily check that the diagram in \eqref{eqn commuting diagram} guarantees that in $B(H_j)$, there exist exactly $n_j$-many bars in the form of $(s_{H_j}(t_1),s_{H_j}(t_2)]$.
	And, one also can easily check that 
	\[s_{H_j}(t_2) - s_{H_j}(t_1) < r_0(t_2 -t_1) < C(Y,\alpha).\]
	Thus, by Equation \ref{eqn invariance 1}, we have $n_1 = n_2$ for any pair $t_1$ and $t_2$. 
	It completes the proof. 
\end{proof}

Now we are ready to prove the main theorem of this subsection, i.e., Theorem \ref{invariance: main thm}.
\begin{proof}[Proof of Theorem \ref{invariance: main thm}]
	First, we fix $\epsilon < C(Y,\alpha)$. 
	For the fixed $\epsilon$,  
	\begin{align}
		\label{eqn invariance 2}
		\begin{split}
			\hbar^{SH}_\epsilon(W_j,\lambda_j)= & \limsup_{T \to \infty} \frac{1}{T} \log \Big(n_\epsilon\big(\tru(B_{SH}(W_j, \lambda_j), T)\big) \Big) \\ 
			= & \limsup_{T \to \infty} \frac{1}{T} \log \Big(n_\epsilon\big(\tru(B_{SH}(W_j, \lambda_j), T+\epsilon)\big) \Big).
		\end{split}
	\end{align}
	We point out that $n_\epsilon\big(\tru(B_{SH}(W_j, \lambda_j), T+\epsilon)$ is the number of bars in $B_{\SH}(W_j,\lambda_j)$ satisfying that 
	\begin{itemize}
		\item[(a)] the length of the bar is larger than or equal to $\epsilon$, and
		\item[(b)] the left-endpoint of the bar is smaller than or equal to $T$. 
	\end{itemize}
	
	The bars in $B_{\SH}(W_j,\lambda_j)$ satisfying (b) can be classified into the following four types:
	\begin{itemize}
		\item Type I: bars longer than $\epsilon$ with right endpoint larger than $T$,
		\item Type II: bars longer than $\epsilon$ with right endpoint smaller than $T$,
		\item Type III: bars shorter than $\epsilon$ with right endpoint larger than $T$,
		\item Type IV: bars shorter than $\epsilon$ with right endpoint smaller than $T$.
	\end{itemize}
	
	Let us denote $n_{j,\epsilon}(X, T)$ by the number of bars in Type X. 
	We note the following two facts:
	\begin{itemize}
		\item because of the facts proven in Section \ref{subsection barcode entropy of a Hamiltonian function},
		\[n_{j,\epsilon}(\text{I},T)+2 \cdot n_{j,\epsilon}(\text{II},T)+n_{j,\epsilon}(\text{III},T)+2 \cdot n_{j,\epsilon}(\text{IV},T)=\dim H^*(W_j)+2 \cdot (\text{ the number of closed Reeb orbits with period < T}).\]
		\item By Proposition \ref{invariance: short bars SH vers}, $n_{j,\epsilon}(\text{III}, T)$ and $n_{j,\epsilon}(\text{IV},T)$ are independent on $j=1,2$. 
	\end{itemize}
	
	By combining the above two facts, one can define a number $n_{\epsilon}(Y,T)$ independent of $j=1,2$ as follows:
	\begin{align*}
		n_{\epsilon}(Y,T):=& n_{j,\epsilon}(\text{I},T)+2 \cdot n_{j,\epsilon}(\text{II},T) - \dim H^*(W_j) \\ 
		=& 2 \cdot (\text{ the number of closed Reeb orbits with period < T}) - n_{j,\epsilon}(\text{III},T)-2 \cdot n_{j,\epsilon}(\text{IV},T).
	\end{align*}
	
	We recall that 
	\begin{gather}
		\label{eqn invariance 3}
		n_\epsilon\big(\tru(B_{SH}(W_j, \lambda_j), T+\epsilon)\big) = n_{j,\epsilon}(\text{I},T) + n_{j,\epsilon}(\text{II},T).
	\end{gather}
	And, one can easily check that
	\begin{gather}
		\label{eqn invariance 4}
		\frac{n_\epsilon(Y,T)+\dim H^*(W_j)}{2} \leq n_{j,\epsilon}(\text{I},T) + n_{j,\epsilon}(\text{II},T) \leq  n_\epsilon(Y,T)+\dim H^*(W_j).
	\end{gather}
	By combining Equations \eqref{eqn invariance 2}--\eqref{eqn invariance 4}, one can conclude that, for $j=1,2$,
	\[\hbar_\epsilon^{\SH}(W_j, \lambda_1):=\limsup_{T\to \infty} \frac{1}{T}\log^+ n_\epsilon(Y,T).\]
	Therefore, $\hbar_\epsilon^{\SH}(W_1, \lambda_1)=\hbar_\epsilon^{\SH}(W_2, \lambda_2)$. 
	Moreover, by Definition \ref{def SH barcode entropy}, we have $\hbar^{\SH}(W_1, \lambda_1)=\hbar^{\SH}(W_2, \lambda_2)$.
\end{proof}

\subsection{$\SH$-Barcodes and $\RFH$-Barcodes} In this subsection, we will introduce different persistence module that can be defined for Liouville domain $(W,d\lambda)$. Then, we will prove that its Barcode entropy agrees with the $\SH$-Barcode entropy $\hbar^{\SH}(W, \lambda)$.

\subsubsection{Persistence module of positive Symplectic homology}
Let $\varepsilon$ be a positive real number smaller than $T_{\min}$, the shortest period of closed Reeb orbits. For every $a>\varepsilon$, the exact sequence of chain complexes
\[0\rightarrow \CF^{<\varepsilon}(H, J)\xrightarrow{\iota} \CF^{<a}(H,J)\xrightarrow{\pi}\CF^{[\varepsilon,a)}(H,J)\rightarrow 0\]
induced by the natural inclusion and projection give rise to the exact triangle
\begin{equation}
\label{Comparision: exact triangle for SH+, Ham vers}
\begin{tikzcd}[column sep=small]
& \HF^{<\varepsilon}(H,J) \arrow[dr, "\iota_*"] & \\
  \HF^{[\varepsilon,a)}(H,J)\arrow[ur, "\delta"]&   & \HF^{<a}(H,J)\arrow[ll, "\pi_*"]
\end{tikzcd}.
\end{equation}
For two Floer data $(H_1, J_1), (H_2, J_2)$, two exact triangles commute with the continuation map $\Phi_{(H_1, J_1), (H_2, J_2)}$. Taking a direct limit, we get an exact triangle
\begin{equation}
\begin{tikzcd}[column sep=small]
\label{Comparision: exact triangle for SH+ with e}
& \SH^{<\varepsilon}(W, \lambda) \arrow[dr, "\iota^a"] & \\
\SH^{[\varepsilon,a)}(W, \lambda)\arrow[ur, "\delta^a"]&
& \SH^{<a}(W, \lambda)\arrow[ll, "\pi^a"]
\end{tikzcd}.
\end{equation}

As mentioned in Previous chapter, $\SH^{<\varepsilon}(W, \lambda)$ equals to $H^*(W)$, the singular cohomology of $W$. With an appropriate choice of cofinal sequence, one can prove that $\SH^{[\varepsilon, a)}{(W, \lambda)}$ is also independent of the choice of $\varepsilon$. Therefore, we define $\SH^{(0,a)}(W,\lambda)$ as $\SH^{[\varepsilon, a)}{(W, \lambda)}$ for some $\varepsilon<a$. The exact triangle \ref{Comparision: exact triangle for SH+ with e} can be re-written as
\begin{equation}
\label{Comparision: exact triangle for SH+}
\begin{tikzcd}[column sep=small]
& H^*(W) \arrow[dr, "\iota^a"] & \\
\SH^{(0,a)}(W, \lambda)\arrow[ur, "\delta^a"]&
& \SH^{<a}(W, \lambda)\arrow[ll, "\pi^a"]
\end{tikzcd}.
\end{equation}

Next, we would like to claim that this exact triangle gives an exact triangle of persistence modules. For $0<a\le b$, chain-level inclusion map induces map between $\iota^{a,b}_{\SH^+}:\SH^{(0,a)}(W,\lambda)\to\SH^{(0,b)}(W,\lambda)$. For $a\le 0$, set $\SH^{(0,a)}(W,\lambda)=0$ and set $\iota^{a,b}_{\SH^+} = 0$ for any $b$ at least $a$.

\begin{definition}
    The persistence module of $\,\SH^+(W, \lambda)$, denoted by $B_{\SH^+}(W,\lambda)$, is defined to be a persistence module consisting of
    \begin{itemize}
        \item a family of vector space $\set{B_{\SH}(W, \lambda)_a:=\SH^{(0,a)}(W, \lambda)}_{a\in \R},$ and
        \item a family of linear maps $\set{\iota^{a,b}_{\SH^+}: B_{\SH}(W, \lambda)_a \to B_{\SH}(W, \lambda)_b}_{a \le b}$.
    \end{itemize}
\end{definition}

Next, we complete the desired exact triangle with $H^*(W)$. Define barcode of $\SH^0$ consists of $\dim (H^*(W))$-many bars of in forms of $(0,\infty)$. More precisely, define $B_{\SH^0}(W,\lambda)_a =H^*(W)$ for $a > 0$ and zero otherwise. Define $\iota^{a,b}_{\SH^0}$ be an  identity map for $0<a\le b$ and zero otherwise. These data again give a persistence module denoted by $B_{\SH^0}(W, \lambda)$.

Since, exact triangles \ref{Comparision: exact triangle for SH+, Ham vers} for $a,b$ commute with the persistence map, exact triangles \ref{Comparision: exact triangle for SH+} also commute with the persistence map. Therefore, we get the exact triangle of persistence modules:
\begin{equation}
\label{Comparision: exact triangle for SH+, barcode vers}
\begin{tikzcd}[column sep=small]
& B_{\SH^0}(W, \lambda) \arrow[dr, "\iota"] & \\
B_{\SH^+}(W, \lambda)   \arrow[ur, "\delta"]&
& B_{\SH}(W, \lambda)   \arrow[ll, "\pi"]
\end{tikzcd},
\end{equation}
where persistence map for each $a > 0$ is defined as $\delta_i, \iota_a, \pi_a$ and zero-map otherwise.

When the contact boundary has a non-degenerate Reeb flow, $B_{SH^{+}}(W, \lambda)$ is also a persistence module of locally finite type. Therefore, we can define $\SH^+$-barcode entropy $\hbar^{\SH^+}(W, \lambda)$.
\begin{remark}
    In \cite{Ginzburg-Gurel-Mazzucchelli22}, the Barcode of geodesic flows is defined for {\em every Riemannian metric} (not necessarily bumpy). By the Viterbo isomorphism, \cite{Viterbo99, Weber06}, with sufficiently field coefficient, ex) $\mathbb{Z}_2$, filtered symplectic homology and filter loop space homology are isomorphic when the metric is bumpy. Therefore, for every bumpy metric $g$ on $M$, 
    \[\hbar_{\SH^+}(D_{1}^*M, \lambda_{\text{can}})=\hbar(M,g).\]
\end{remark}

\begin{prop}
    For Liouville domain $(W, \lambda)$ with a non-degenerate Reeb flow, $\SH$-barcode entropy and $\SH^+$-barcode entropy agrees,
    \[\hbar^{\SH^+}(W, \lambda) = \hbar^{\SH}(W, \lambda).\]
\end{prop}

\begin{proof}
Since $B_{\SH^0}(W, \lambda)$ consists of $\dim H^*(W)$-many $(0,\infty)$ bars, $B_{\SH^0}(W, \lambda)$, $B_{\SH^0}(W, \lambda)$ is a persistence module of finite type. Therefore, all the persistence module in the exact triangle \ref{Comparision: exact triangle for SH+, barcode vers} are moderate in the sense of \cite{Buhovsky-Payette-Polterovich-Polterovich-Shelukhin-Stojisavljevic22}. Then, by \cite[Theorem 3.1]{Buhovsky-Payette-Polterovich-Polterovich-Shelukhin-Stojisavljevic22}, we can compare the number of bars $B_{\SH^+}(W, \lambda)$ and $B_{\SH}(W, \lambda)$.

\begin{thm}\cite[Theorem 3.1]{Buhovsky-Payette-Polterovich-Polterovich-Shelukhin-Stojisavljevic22}
\label{Comparison: comparing the number of bars}
    Let $U\to V\to W$ be an exact sequence of moderate persistence modules. Then for every $\delta>0$, the following inequality holds:
    \[b_{2\delta}(V)\le b_{\delta}(U)+ b_{\delta}(W),\]
    where $n_{\varepsilon}(B)$ denotes the number of bars in $B$ longer than $\varepsilon$.
\end{thm}
\begin{proof}
    For proof, see \cite[Chapter 3]{Buhovsky-Payette-Polterovich-Polterovich-Shelukhin-Stojisavljevic22}
\end{proof}

Applying the Theorem to two exact sequences with $T>0$,
\begin{align*}
\tru(B_{\SH^0}(W, \lambda), T)\xrightarrow{\tru(\iota)} \tru(B_{\SH}(W, \lambda),T) \xrightarrow{\tru(\pi)} \tru(B_{\SH^+}(W, \lambda),T),\\
\tru(B_{\SH}(W, \lambda), T)\xrightarrow{\tru(\pi)} \tru(B_{\SH^+}(W, \lambda),T) \xrightarrow{\tru(\delta)} \tru(B_{\SH^0}(W, \lambda),T),
\end{align*}
we get 
\begin{align*}
n_{2\varepsilon}(\tru(B_{\SH},T))\le n_\varepsilon(\tru(B_{\SH^+},T))+n_\varepsilon(\tru(B_{\SH^0},T))\le n_\varepsilon(\tru(B_{\SH^+},T))+\dim (H^*(W))\\
n_{2\varepsilon}(\tru(B_{\SH^+},T))\le n_\varepsilon(\tru(B_{\SH},T))+n_\varepsilon(\tru(B_{\SH^0},T))\le n_\varepsilon(\tru(B_{\SH},T))+\dim (H^*(W)).
\end{align*}
Therefore,
\begin{align*}
n_{4\varepsilon}(\tru(B_{\SH},T))-\dim (H^*(W))
\le  n_{2\varepsilon}(\tru(B_{\SH^+},T))
\le n_\varepsilon(\tru(B_{\SH},T))+\dim (H^*(W)),\\
n_{4\varepsilon}(\tru(B_{\SH^+},T))-\dim (H^*(W)) \le  n_{2\varepsilon}(\tru(B_{\SH},T)) \le n_\varepsilon(\tru(B_{\SH^+},T))+\dim (H^*(W)),
\end{align*}
and
\[ \hbar^{\SH^+}_{4\varepsilon}(W, \lambda) < \hbar^{\SH}_{2\varepsilon}(W, \lambda) \le \hbar^{\SH^+}_{\varepsilon}(W, \lambda), \quad \hbar^{\SH}_{4\varepsilon}(W, \lambda) < \hbar^{\SH^+}_{2\varepsilon}(W, \lambda) \le \hbar^{\SH}_{\varepsilon}(W, \lambda).\]
Taking $\varepsilon\to 0$, we get $\hbar^{\SH^+}(W, \lambda)=\hbar^{\SH}(W, \lambda)$.
\end{proof}
\subsubsection{Comparison with Rabionwitz Floer homology}
Rabinowitz Floer homology is the Floer homology of Rabinowitz action functional $\cA^F: C^\infty(S^1, \widehat{W})\times \R \to \R$ define by
\[\cA^F(u,\eta) := \int_0^1 u^* \lambda - \eta \int_0^1 F(u(t))dt,\]
where $F$ is a function on $\widehat{W}$ whose zero level set $F^{-1}(0)=\partial W$. The critical point of $\cA^F$ are precisely the constant solution on $\partial W$ and the positive/negative Reeb orbits. Then the Rabinowitz Floer homology $\RFC(W,\lambda)$ is the homology of chain complex $\RFC(W, \lambda)$ that is generated by the critical points of $\cA^F$ and with a differential that is given by counting rigid Rabinowitz Floer gradient trajectories connecting critical points. See \cite{Cieliebak-Frauenfelder09, Albers-Frauenfelder10, Cieliebak-Frauenfelder-Oancea10} for more details on the Rabinowitz Floer homology.

Since the differential $\partial_{\RFH}$ decreases the action, filtered Rabinowitz Floer homology is well-defined. Following \cite[Chapter 10]{Meiwes18}, we define non-negative $\RFH$ persistence module $B_{\RFH^{\ge 0}}(W, \lambda)$ with the following data. For $a>0$, let $\RFH^{[0,a)}(W,\lambda)$ be homology of $(\RFC^{[0,a)}(W,\lambda), \partial_{\RFH})$ and $\iota^{a,b}_{\RFH^\ge 0}$ be a map induced by chain level inclusion map.
\begin{definition}
    The persistence module of $\,\RFH^\ge 0(W, \lambda)$, denoted by $B_{\RFH^{\ge 0}}(W,\lambda)$, is defined to be a persistence module consisting of
    \begin{itemize}
        \item a family of vector space $\set{B_{\RFH^{\ge 0}}(W, \lambda)_a:=\RFH^{(0,a)}(W, \lambda)}_{a\in \R}$ and
        \item a family of linear maps $\set{\iota^{a,b}_{\RFH^{\ge 0}}: B_{\RFH^{\ge 0}}(W, \lambda)_a \to B_{\RFH^{\ge 0}}(W, \lambda)_b}_{a \le b}$.
    \end{itemize}
\end{definition}

\begin{thm}\cite[Proposition 1.4.]{Cieliebak-Frauenfelder-Oancea10}
For $a>0$, there is an exact triangle
\begin{equation}
\label{Comparision: exact triangle for RFH}
\begin{tikzcd}[column sep=small]
& H_*(W) \arrow[dr] & \\
\RFH^{[0,a)}(W, \lambda)\arrow[ur]&
& \SH^{<a}(W, \lambda)\arrow[ll]
\end{tikzcd}.
\end{equation}
Moreover, the triangles are compatible with persistence maps $\iota^{a,b}_{\SH}$, $\iota^{a,b}_{\RFH^{\ge 0}}$ for $a<b$.
\end{thm}

In fact, in \cite{Cieliebak-Frauenfelder-Oancea10}, the exact triangle is given for $a=\infty$ case. However, as mentioned in \cite[Remark 10.5.]{Meiwes18}, the theorem directly follows from the proof. Define $\RFH^{\ge 0}$-barcode entropy of $(W, \lambda)$ according to definition \ref{def SH barcode entropy} by changing $\SH$ into $\RFH^{\ge 0}$. Appyling the theorem \ref{Comparison: comparing the number of bars} for the exact triangle \ref{Comparision: exact triangle for RFH}, we get the following proposition.
\begin{prop}
Let $(W,\lambda)$ be a Liouville domain with a non-degenerate Reeb flow on the contact boundary. Then, $\SH$-barcode entropy and $\RFH^{\ge 0}$-barcode entropy of $(W, \lambda)$ agrees,
\[\hbar^{\SH}(W,\lambda)=\hbar^{\RFH^{\ge 0}}(W,\lambda).\]
\end{prop}

\section{SH barcode entropy vs the topological entropy of the Reeb flow}
\label{section SH barcode entropy vs the topological entropy of the Reeb flow}

For a Liouville domain $(W,\omega = d\lambda)$ such that $\alpha:= \lambda|_{\partial W}$ is a non-degenerate contact one-form on $\partial W$, we defined $\hbar^{SH}(W,\lambda)$ in Section \ref{section two barcode entropy}. 
In this subsection, we compare $\hbar^{SH}(W,\lambda)$ to the topological entropy of the Reeb flow of $\alpha$. 

Let $R$ denote the Reeb vector field of $\alpha$ and let $\mu^t$ (resp.\ $\mu$) denote the time $t$ (resp.\ time $1$) flow of $R$. 
We will define the topological entropy in Section \ref{subsection notations for Section 4}. 
Let $h_{top}(\mu)$ denote the topological entropy of $\mu$.
With the notation, we would like to prove the following theorem in Section \ref{section SH barcode entropy vs the topological entropy of the Reeb flow}.
\begin{thm}
	\label{thm vs topological entropy formal}
	In the above setting, the barcode entropy of $\SH(W,\lambda)$ bounds the topological entropy of $\mu$ from below, i.e., 
	\[\hbar^{SH}(W,\lambda) \leq h_{top}(\mu).\]
\end{thm}

In order to prove Theorem \ref{thm vs topological entropy formal}, we fix our setting and set notations in Section \ref{subsection notations for Section 4}. 
Section \ref{subsection sketch of the proof} will give a sketch of the proof of Theorem \ref{thm vs topological entropy formal}.
In the sketch, one can see that we need preparations for the proof, which are done in Sections \ref{subsection Lagrangian tomograph}--\ref{subsection Crofton type inequality}. 
In the last subsection of Section \ref{section SH barcode entropy vs the topological entropy of the Reeb flow}, we will give a full proof of Theorem \ref{thm vs topological entropy formal}.

\subsection{Notations for Section \ref{section SH barcode entropy vs the topological entropy of the Reeb flow}}
\label{subsection notations for Section 4}
Theorem \ref{thm vs topological entropy formal} compares two entropy, one is the barcode entropy of $\SH(W,\lambda)$ and the other is the topological entropy of the Reeb flow. 

To be more precise, let us define the notion of topological entropy first. 
For a compact space $X$ equipped with a metric and an automorphism $\phi:X \to X$, the topological entropy of $\phi$ is defined as follows: 
\begin{definition}
	\label{def topological entropy}
	\mbox{}
	\begin{enumerate}
		\item We set 
		\[\boldsymbol{\Gamma_k(\phi)}:= \left\{\left(x, \phi(x), \phi^2(x), \dots, \phi^{k-1}(X)\right) \in X \times X \times \dots X =:X^k | x \in X\right\}.\]
		\item A {\bf $\boldsymbol{\epsilon}$-cube} in $X^k$ is a product of balls of radius $\epsilon$ from $X$. 
		\item For a set $Y \subset X^k$, $\boldsymbol{\mathrm{Cap}_\epsilon(Y)}$ denotes the minimal number of $\epsilon$-cubes in $X^k$ needed to cover $Y$.
		\item The {\bf topological $\boldsymbol{\epsilon}$-entropy of $\boldsymbol{\phi}$ $\boldsymbol{h_\epsilon(\phi)}$} is defined as 
		\[h_\epsilon(\phi):= \limsup_{k \to \infty} \tfrac{1}{k} \log \mathrm{Cap}_\epsilon\left(\Gamma_k(\phi)\right).\]
		\item The {\bf topological entropy of $\boldsymbol{\phi}$ $h_{top}(\phi)$} is defined as 
		\[h_{top}(\phi):= \lim_{\epsilon \searrow 0} h_\epsilon(\phi).\]
	\end{enumerate}
\end{definition}

\begin{remark}
	We note that the definition of topological entropy in Definition \ref{def topological entropy} follows that of Bowen \cite{Bowen71a} and Dinaburg \cite{Dinaburg70}. 
	There is another definition of topological entropy given by \cite{Adler-Konheim-McAndrew}, which is defined without using a metric structure. 
	We also note that the topological entropy defined in Definition \ref{def topological entropy} is independent of the choice of metric structures on a compact space $X$. 
\end{remark}

We recall that by Theorem \ref{thm equivalence}, $\hbar^{SH}(W,\lambda)$ can be computed as $\hbar(\{H_i\})$ where $\{H_i\}_{i \in \N}$ is a sequence of Hamiltonian functions satisfying conditions described in Section \ref{subsection barcode entropy of a Hamiltonian function}. 
Because of the computational convenience, we fix a specific sequence of Hamiltonian $\{H_i\}_{i \in \N}$ in Section \ref{section SH barcode entropy vs the topological entropy of the Reeb flow}. 

For simplicity, we assume that the fixed sequence $\{H_i\}_{i \in \N}$ satisfies the following condition unless otherwise stated:
\begin{enumerate}
	\item[($\star$)] For any $i \in \N$, $H_i = i \cdot H_1$. 
\end{enumerate}
One can easily check that the assumption ($\star$) holds if and only if there exists a positive number $T$ such that $T$ is not a period of any closed Reeb orbit for all $i \in \N$. 
If there exists a such $T$, then one can choose a convex radial Hamiltonian function $H$ of slope $T$, then set $H_i = i \cdot H$. 
We note that the arguments in below subsections work for any sequence $\{H_i\}$ after a small modification, even if the assumption ($\star$) does not hold. 
Thus, there is no harm in assuming ($\star$).

For the fixed sequence, we use the following notation:
\begin{itemize}
	\item The convex radial Hamiltonian functions are linear outside of $W_{r_0}$ for a constant $r_0$. 
	\item $H_1 = Tr -C$ on $\widehat{W}\setminus W_{r_0}$, or more generally, $H_i = iTr - iC$ for any $i \in \N$.
\end{itemize}

Since we fix a specific sequence, we can simplify the notations defined in Definition \ref{def barcode entropy of a Hamiltonian function} for simplicity. 
Thus, we use the following notations:
\begin{gather}
	\label{eqn new notations}
	\begin{split}
		& b_\epsilon(i) := n_\epsilon\left(\tru\big(B(H_i),iC\big)\right), \\
		& \hbar_\epsilon := \limsup_{i \to \infty} \frac{1}{iT}\log b_\epsilon(i) = \hbar_\epsilon(\{H_i\}), \\
		& \hbar := \lim_{\epsilon \searrow 0} \hbar_\epsilon = \hbar(\{H_i\}).
	\end{split}
\end{gather}

So far, we introduced new notations which simplify the already defined notations.
We define the following two notations which did not appear before. 
\begin{itemize}
	\item We will consider the product space $\widehat{W} \times \widehat{W}$ and the diagonal subspace 
	\[\Delta:= \left\{(x,x) \in \widehat{W} \times \widehat{W} | x \in \widehat{W}\right\}.\]
	\item For a (time-dependent) Hamiltonian function $F: S^1 \times \widehat{W} \to \R$, we choose a Lagrangian submanifold $\Gamma_H$ of $\widehat{W}\times \widehat{W}$ as follows: 
	First we consider the Hamiltonian vector field $X_H$ and its time $1$-flow $\phi_H$.
	Then, $\Gamma_H$ is given as the graph of $\phi_H$, i.e.,
	\[\Gamma_H := \left\{\left(x,\phi_H(x)\right) | x \in \widehat{W}\right\}.\]
\end{itemize}

With the above notations, Lemma \ref{lem intersection points} is easy to prove. 
\begin{lem}
	\label{lem intersection points}
	Let $H$ be a non-degenerate Hamiltonian function on $\widehat{W}$. 
	\begin{enumerate}
		\item $\Delta$ and $\Gamma_H$ intersect transversally. 
		\item Moreover, if $H$ is linear at infinity, then there is a one to one correspondence 
		\[O_H : \mathcal{P}_1(H) \stackrel{\sim}{\to} \Delta \cap \Gamma_H.\]
		We recall that $\mathcal{P}_1(H)$ is the set of $1$-periodic orbits of $X_H$.  
	\end{enumerate}	
\end{lem}

\subsection{Sketch of the proof}
\label{subsection sketch of the proof}
In Section \ref{subsection sketch of the proof}, we roughly explain our strategy of the proof of Theorem \ref{thm vs topological entropy formal}. 
We note that the goal of this subsection is to briefly describe the strategy and we misuse some terms for simplicity. 

First, We recall that $\phi_{H_1}$ means the time $1$ Hamiltonian flow of $H_1$, $\mu^t$ means time $t$ flow of the Reeb vector field, and $T$ means the slope of $H_1$. 
By definition, it is easy to observe that $h_{top}\left(\phi_{H_1}|_{W_{r_1}}\right) = h_{top}\left(\mu^T\right) = T h_{top}\left(\mu\right)$.
We restrict $\phi_{H_1}$ on $W_{r_1}$ since in Definition \ref{def topological entropy}, we only define the topological entropy for automorphisms of compact spaces. 
Then, we compare $\hbar^{SH}(W,\lambda) = \hbar(\{H_i\})$ with $h_{top}(\phi_{H_1})$, instead of $h_{top}(\mu)$.

From classical results, for example \cite{Yomdin}, it is easy to observe that $h_{top}(\phi_{H_1})$ is bounded from below by the exponential volume growth of $\Gamma_{H_n}$ as $n \to \infty$. 
Thus, we care about the volume of $\Gamma_{H_n}$. 
In order to deal with the volume of $\Gamma_{H_n}$, we fix a positive constant $\epsilon$, and we fix a collection of Lagrangians submanifolds
\[\left\{\Delta^s \subset \widehat{W}\times \widehat{W}| s \in B_\delta^d\right\},\]
satisfying the following:
\begin{itemize}
	\item $B_\delta^d$ is a closed $d$-dimensional ball with radius $\delta$,
	\item $\delta$ is a small positive number depending on $\epsilon$, and
	\item for any $s \in B_\delta^d$, $\Delta^s$ is obtained by a perturbation of $\Delta$. 
\end{itemize}
The fixed collection will be called {\em Lagrangian tomograph}.

After fixing a Lagrangian tomograph, by using a {\em Crofton type} inequality, one can show that 
\begin{gather}
	\label{eqn goal 1}
	\text{the exponential growth rate of  } \int_{s \in B^d_\delta} \left|\Delta^s \cap \Gamma_{H_n^s}\right| \leq \text{  the exponential growth rate of the volume of  } \Gamma_{H_n}.
\end{gather}

By \eqref{eqn goal 1}, in order to prove Theorem \ref{thm vs topological entropy formal}, it is enough to prove that 
\begin{gather}
	\label{eqn goal 2}
	b_\epsilon(n) \leq \left|\Delta \cap \Gamma_{H_n^s}\right|.
\end{gather}
More precisely, if \eqref{eqn goal 2} holds, then one can show that the exponential growth of $b_\epsilon(n)$ with $n$ is smaller or equal to the exponential growth of the volume of $\Gamma_{H_n}$.
Let us recall that the later growth rate low-bounds $h_{top}(\phi_{H_1})$.
Thus, we have 
\begin{gather}
	\label{eqn goal 3}
	T \cdot h(\epsilon) \leq h_{top}(\phi_{H_1}) = T \cdot h_{top}(\mu),
\end{gather}
for any $\epsilon$.
One can completes the proof of Theorem \ref{thm vs topological entropy formal} by taking limit of the inequality in \eqref{eqn goal 3} as $\epsilon \searrow 0$.  

For the full proof, we need to fix a Lagrangian tomograph, and we need to prove the Crofton type inequality implying \eqref{eqn goal 2}. 
The former will be done in Section \ref{subsection Lagrangian tomograph} and the later will be done in Section \ref{subsection Crofton type inequality}.

\subsection{Lagrangian tomograph}
\label{subsection Lagrangian tomograph}
Let us start this subsection by recalling/setting notations.

\begin{itemize}
	\item As mentioned before, $\{H_n\}_{n \in \mathbb{N}}$ is a fixed sequence of Hamiltonian functions.  
	\item For any $n \in \N$, $H_n$ is linear at infinity.
	Especially, outside of $W_{r_0}$, $H_n \equiv nTr-nC$ for fixed constants $T$ and $C$.  
	\item We recall that, for any $t >1$, 
	\[W_t = W \cup \left(\partial W \times [1, t]\right) \subset \widehat{W}.\]
	\item We also use the following notation:
	\[r_1 := r_0 + 1, \text{  and  } d \text{  is a fixed constant such that  } d \geq \dim \widehat{W}.\]
	\item Throughout Section \ref{subsection Lagrangian tomograph}, $\epsilon$ is a fixed positive small number. 
\end{itemize}

Let us fix a positive real number $\delta$ depending on $\epsilon$ such that 
\begin{gather}
	\label{eqn condition on delta}
	\delta(r_1 +1 + 3 \delta) < \epsilon \text{  and  }  \delta < \frac{1}{2}.
\end{gather}

Now, we define the following notation:
\begin{definition}
	\label{def Lagrangian tomograph}
	\mbox{}
	\begin{enumerate}
		\item We fix a collection of Hamiltonian functions 
		\[\left\{\boldsymbol{g_i} : \widehat{W} \to \mathbb{R} | i = 1, \dots, d\right\},\]
		such that 
		\begin{enumerate}
			\item[(A)] if $p \in W_{r_1}$, then $\{d_pg_i | i = 1, \dots, d\}$ generates $T_p^*\widehat{W}$,
			\item[(B)] $g_i \equiv 0$ on $\widehat{W} \setminus W_{r_1+\delta}$, and 
			\item[(C)] $g_i$ is sufficiently small, in particular,
			\begin{gather}
				\label{eqn conditions on g_i}
				|g_i| < 1 \text{  on  } \widehat{W}, \left\vert \frac{\partial g_i}{\partial r} \right\vert\ < 1, \left\vert\alpha(X_{g_s})\right\vert < 1 \text{  on  } \partial W \times [1, \infty)_r.
			\end{gather}
		\end{enumerate}
		\item For $s= (s_1, \dots, s_d) \in \mathbb{R}^d$, we define $\boldsymbol{g_s}$ as follows:
		\[g_s := \sum_{i=1}^d s_i g_i.\]
		\item Then, the desired collection of Lagrangians, {\bf Lagrangian tomograph}, is given as 
		\[\left\{\Delta^s:=\Gamma_{-g_s} \subset \widehat{W}\times \widehat{W} | s \in B_\delta^d \right\},\]
		where $B_\delta^d \subset \mathbb{R}^d$ is a closed ball of radius $\delta$ centered at the origin $\mathbf{0}= (0, \dots, 0)$.
	\end{enumerate}
\end{definition}
\begin{remark}
	\label{rmk constrction of Lagrangian tomograph}
	We note that the original idea of constructing Lagrangian tomograph is given in \cite[Section 5.2.3]{Cineli-Ginzburg-Gurel21}, and Definition \ref{def Lagrangian tomograph} (3) is a slight modification of the original idea. 
	The difference between two constructions are the following: 
	The original idea is using Weinstein Neighborhood Theorem \cite{Weinstein71}, but Definition \ref{def Lagrangian tomograph} is using a graph of a Hamiltonian flow. 
	Because of the difference between two constructions, there is a slight difference between \cite[Section 5.2.2]{Cineli-Ginzburg-Gurel21} and Section \ref{subsection Crofton type inequality}.
\end{remark}

In Section \ref{subsection the proof of Theorem vs topological entropy}, we will relate the barcode $B(H_n^s)$ and $\left|\Delta^s \cap \Gamma_{H_n}\right|$, where $H_n^s$ is a Hamiltonian function defined as follows:
\begin{definition}
	\label{def H_n^S}
	We set a Hamiltonian function $\boldsymbol{H_n^s}$ by setting as
	\begin{gather}
		\label{eqn tomograph}
		H_n^s := H_n + g_s \circ \phi_{H_n}^{-t}.
	\end{gather}
\end{definition} 
In the rest of this subsection, we prove two Lemmas \ref{lem Hamiltonian flow} and \ref{lem energy bound} which describe properties of $H_n^s$. 

\begin{lem}
	\label{lem Hamiltonian flow} 
	The Hamiltonian flow of $H_n^s$ is the composition of those of $H_n$ and $g_s$, i.e., 
	\[\phi_{H_n^s}^t = \phi_{H_n}^t \circ \phi_{g_s}^t.\]
\end{lem}

\begin{lem}
	\label{lem energy bound}
	Let us assume the following:
	\begin{enumerate}
		\item $s \in \mathbb{R}^d$ satisfies that $\| s \| < \delta$.
		\item $H_n^s$ is a non-degenerate Hamiltonian function, or equivalently, $\Delta$ and $\Gamma_{H_n^s}$ intersect transversally.
		\item An intersection point $(x,x) \in \Delta \cap \Gamma_{H_n^s}$ satisfies that 
		\[\mathcal{A}_{H_n^s}\left(O_{H_n^s}^{-1}(x,x)\right) < nC-\delta (r_1+1+2\delta),\]
	\end{enumerate}
	Then, $x \in W_{r_1 - \delta}$.
\end{lem}
We note that $\mathcal{A}_{H_n^s}$ is defined in Definition \ref{def action functional on CF(H)} and $O_{H_n^s}$ is defined in Lemma \ref{lem intersection points}.

Lemma \ref{lem Hamiltonian flow} describes the Hamiltonian flow of $\phi^t_{H_n^s}$, which we need to prove Lemma \ref{lem energy bound}. 
The main Lemma of this subsection, Lemma \ref{lem energy bound}, plays a key role in the proof of Theorem \ref{thm vs topological entropy formal}.

\begin{proof}[Proof of Lemma \ref{lem Hamiltonian flow}]
	We prove a more general statement: If $F, G: \widehat{W} \to \mathbb{R}$ are Hamiltonian functions, and $H$ is defined as 
	\[H := F + g \circ \phi^{-t}_F,\]
	then $\phi_H = \phi_F \circ \phi_G$.
	
	It is enough to show that the derivative of an isotopy $\phi^t_F \circ \phi^t_G$ agrees with the Hamiltonian vector field of $H$.
	In other words, we show that, for any $x \in \widehat{W}$ and $t_0$, 
	\begin{gather}
		\label{eqn goal}
		\omega \left(\tfrac{d}{dt}\bigg\vert_{t=t_0} \left( \phi^t_F \circ \phi^t_G\right), \relbar \right) = dH(\relbar).
	\end{gather}
	
	In order to prove Equation \eqref{eqn goal}, first we observe that
	\begin{equation*}
		\tfrac{d}{dt}\bigg\vert_{t=t_0} \left( \phi^t_F \circ \phi^t_G\right) = X_F\left(\phi^{t_0}_F \circ \phi^{t_0}_G(x)\right)  + (\phi^{t_0}_F)_*\left(X_G\left(\phi^{t_0}_G(x)\right)\right).
	\end{equation*}
	Thus, we have 
	\begin{align*}
		\omega \left(\tfrac{d}{dt}\bigg\vert_{t=t_0} \left( \phi^t_F \circ \phi^t_G\right), \relbar \right) &= \omega\left(X_F\left(\phi^{t_0}_F \circ \phi^{t_0}_G(x)\right), \relbar\right) + \omega\left((\phi^{t_0}_F)_*\left(X_G\left(\phi^{t_0}_G(x)\right)\right), \relbar\right) \\ 
		&= dF_{\phi^{t_0}_F \circ \phi^{t_0}_G(x)}(\relbar) + \left((\phi^{t_0}_F)^* \omega\right)_{\phi^{t_0}_F \circ \phi^{t_0}_G(x)}\left(X_G\left(\phi_G^{t_0}(x)\right), (\phi^{-t_0}_F)_*(\relbar)\right) \\
		&= dF_{\phi^{t_0}_F \circ \phi^{t_0}_G(x)}(\relbar) + \omega_{\phi^{t_0}_G(x)}\left(X_G\left(\phi^{t_0}_G(x)\right), (\phi^{-t_0}_F)_*(\relbar)\right) \\
		&= dF_{\phi^{t_0}_F \circ \phi^{t_0}_G(x)}(\relbar) + dG_{\phi^{t_0}_G(x)}\left((\phi^{-t_0}_F)_*(\relbar)\right) \\
		&= dF_{\phi^{t_0}_F \circ \phi^{t_0}_G(x)}(\relbar) + \left((\phi^{-t_0}_F)^*dG\right)_{\phi^{t_0}_F \circ \phi^{t_0}_G(x)}(\relbar) \\
		&= \left(dF + d(G\circ \phi^{-t_0}_F)\right)_{\phi^{t_0}_F \circ \phi^{t_0}_G(x)} \\
		&=dH.
	\end{align*}
\end{proof}

\begin{proof}[Proof of Lemma \ref{lem energy bound}]
	We prove Lemma \ref{lem energy bound} by contrary, i.e., we show that $\mathcal{A}_{H_n^s}\left(O_{H_n^s}^{-1}(x,x)\right) \geq nC -\delta (r_1+1)$ if $x \notin W_{r_1 -\delta}$. 
	
	Let
	\[p: \mathbb{R} / \mathbb{Z} \to \widehat{W}, p(t):= \phi^t_{H_n^s}(x).\]
	denote the $1$-periodic orbit corresponding to $(x,x)$, i.e., $O_{H_n^s}\left(p(t)\right) = (x,x)$.  
	We are interested in the $r$-coordinate of $p(t)$. 
	We note that $p(t)$ is a flow of the Hamiltonian vector field $X_{H_n^s}$.
	In the proof of Lemma \ref{lem Hamiltonian flow}, we observed that 
	\[p'(t) = X_{H_n^s} = X_{H_n} + (\phi_{H_n}^t)_*(X_{g_s}).\]
	On $\partial W \times [1,\infty)_r$, we can see that the $\partial r$-directional part of $X_{H_n^s}$ is given as 
	\begin{align*}
		dr\left(X_{H_n^s}\right) &= \left(dr \wedge \alpha + r d\alpha\right) \left(-R, X_{H_n^s}\right) \\
		&= \omega(-R, X_{H_n^s}) \\
		&= -\left(d H_n^s\right)(R) \\
		&= - dH_n(R) - dg_s\left((\phi^t_{H_n})_*(R)\right) \\
		&= - dg_s(R) \\
		&= - dr (X_{g_s}). 
	\end{align*}
	Thus, we have $\left\vert dr\left(p'(t) \right)\right \vert < \delta$ for all $t\in [0,1]$, and then the $r$-coordinate of $p(t)$, $r\left(p(t)\right)$, satisfies that 
	\[r\left(p(t)\right) \in \left[r\left(p(0)\right) - \delta, r\left(p(0)\right) + \delta\right].\]
	We would like to point out that \eqref{eqn conditions on g_i} proves the second-last inequality of the above.
	
	Now, we divide $\mathcal{A}_{H_n^s}\left(O_{H_n^s}^{-1}(x,x) =p(t)\right)$ into three terms in \eqref{eqn action computation 1} as follows: 
	\begin{align}
		\notag
		\mathcal{A}_{H_n^s}\left(p(t)\right) &= \int_0^1 \lambda\left(p'(t)\right) dt - \int_0^1 H_n^s\left(p(t)\right) dt \left(\because \text{  by definition}\right) \\
		\notag
		&= \int_0^1 r\left(p(t)\right) \alpha \left(X_{H_n} + (\phi^t_{H_n})_*(X_{g_s})\right)dt - \int_0^1 H_n\left(p(t)\right) dt - \int_0^1 g_s\left(\phi^{-t}_{H_n}\left(p(t)\right)\right)dt \\
		\label{eqn action computation 1} 
		&= \int_0^1\left[ r\left(p(t)\right) \alpha \left(X_{H_n}\right) - H_n\left(p(t)\right)\right] dt + \int_0^1 r\left(p(t)\right) \alpha \left((\phi^t_{H_n})_*(X_{g_s})\right)dt- \int_0^1 g_s\left(\phi^{-t}_{H_n}\left(p(t)\right)\right)dt.
	\end{align}
	
	The first term of \eqref{eqn action computation 1} is easily computed. 
	To be more precise, we would like to point out that since $x \notin W_{r_1 - \delta}$, the $r$-coordinate of $x$ is larger than $r_1 - \delta$. 
	Then, we have 
	\[r_0 < r_1- 2 \delta < r\left(p(t)\right),\]
	because $p(t)$ is the flow of $X_{H_n^s}$ and $p(0)=x$.
	It means that for any $t \in [0,1]$, 
	\[H_n\left(p(t)\right) = nT r\left(p(t)\right) - n C,\]
	because $H_n = nTr -nC$ on $\partial W \times [r_0, \infty)$. 
	
	Now, the first term in \eqref{eqn action computation 1} becomes 
	\begin{align*}
		\int_0^1\left[ r\left(p(t)\right) \alpha \left(X_{H_n}\right) - H_n\left(p(t)\right)\right] dt &= \int_0^1\left[r\left(p(t)\right) \alpha \left(-nTR\right) - nTr\left(p(t)\right) \right]dt +nC \\
		&= nC.
	\end{align*} 
	
	The second term of \eqref{eqn action computation 1} is bounded from below by $- \delta(r_1 +2 \delta)$ because of the following reason:
	First, we note that $r(x) < r_1 + \delta$.
	If not, by the condition (B) of Definition \ref{def Lagrangian tomograph} (1), $H_n^s = H_n$ at $p(t)$ for all $t$. 
	Thus, $p(t)$ is a Reeb orbit of period $nT$.
	Since $nT$ is not a period of a Reeb orbit, it is a contradiction. 
	It implies that
	\[r\left(p(t)\right) \leq r_1 + \delta + \delta =r_1 +2 \delta.\]
	We also note that $\phi_{H_n}^t$ preserves $\alpha$ on $\partial W \times [r_0, \infty)_r$, i.e., $(\phi_{H_n}^t)^*(\alpha) = \alpha$. 
	Thus, the second term in \eqref{eqn action computation 1} satisfies that
	\begin{align*}
		\int_0^1 r\left(p(t)\right) \alpha \left((\phi^t_{H_n})_*(X_{g_s})\right)dt &= \int_0^1 r\left(p(t)\right) (\phi^t_{H_n})^*\alpha (X_{g_s})dt \\
		&= \int_0^1 r\left(p(t)\right) \alpha (X_{g_s}) \\
		& \geq - \delta(r_1 + 2 \delta).
	\end{align*}
	
	Finally, the last term in \eqref{eqn action computation 1} satisfies that 
	\begin{align*}
		- \int_0^1 g_s\left(\phi^{-t}_{H_n}\left(p(t)\right)\right)dt \geq - \max \{|g_s|\} > - \left\| s \right\| > - \delta. 
	\end{align*}
	
	The above arguments about three terms in \eqref{eqn action computation 1} give a lower bound of $\mathcal{A}_{H_n^s}\left(O_{H_n^s}^{-1}(x,x)\right)$ as follows:
	\[\mathcal{A}_{H_n^s}\left(O_{H_n^s}^{-1}(x,x)\right) \geq nC -\delta(r_1+ 2 \delta) - \delta = nC - \delta (r_1+1 +2 \delta).\]
	It completes the proof.
\end{proof}

\subsection{Crofton type inequality}
\label{subsection Crofton type inequality}
In this subsection, we prove \eqref{eqn goal 1} which we call the {\em Crofton type} inequality.  
We note that the main idea of the Crofton type is well-known to experts, as mentioned in \cite[Section 5.2.2]{Cineli-Ginzburg-Gurel21}.
We prove the Crofton type inequality \eqref{eqn goal 1} for completeness of the current paper.
For more general arguments about the idea, see \cite[Section 5.2.2]{Cineli-Ginzburg-Gurel21} and references therein. 

In order to state our Crofton type inequality formally, we need preparations.
First, we fix positive constants $\epsilon$ and $\delta$ satisfying \eqref{eqn condition on delta}.
Then, we have a fixed Lagrangian tomograph $\left\{\Delta^s | s \in B_{\delta}^d\right\}$.

We note that because non-degenerate Hamiltonian functions are generic in the space of Hamiltonian functions, for almost all $s \in B_{\delta}^d$, $\Delta \pitchfork \Gamma_{H_n^s}$. 
Let 
\[N_n(s):= \left\vert\Delta \cap \Gamma_{H_n^s} \cap \left(W_{r_1 - \delta} \times W_{r_1-\delta}\right)\right\vert.\]
Then, $N_n(S)$ is finite for almost all $s \in B_{\delta}^d$, and $N_n(S)$ is an integrable function on $B_{\delta}^d$. 

Let $\Psi$ be a map defined as follows:
\[\Psi: B_{\delta}^d \times \widehat{W} \to \widehat{W} \times \widehat{W}, (s,x) \mapsto \left(x, \phi_{-g_s}(x)\right).\]
We also set 
\[E:= \Psi^{-1}\left(W_{r_1} \times W_{r_1}\right).\] 

We note that $B_{\delta}^d \subset \mathbb{R}^d$.
Thus, the Euclidean metric on $\mathbb{R}^d$ induces a volume form on $B_{\delta}^d$.
Let $ds$ denote the induced volume form on $B_{\delta}^d$. 

We also note that if $(s,x) \in E$, and if $\|s\|$ is sufficiently small, then $\Psi$ is a submersion at $(s,x)$. 
To be more precise, let us mention that 
\[D\Psi|_{(0,x)}(-\partial s_i, 0) = X_{g_i}(x),\]
where $X_{g_i}$ denotes the Hamiltonian vector field of $g_i$.
Because of the condition (A) of Definition \ref{def Lagrangian tomograph} (1), $D\Psi|_{(0,x)}$ is surjective if $x \in W_{r_1}$. 
Since $W_{r_1}$ is a compact set, there is a $\delta >0$ such that 
\begin{gather}
	\label{eqn condition of delta}
	\text{if $s$ satisfies $\|s\|< \delta$, then $\Psi$ is a submersion at $(s,x)$ for all $x \in W_{r_1}$.}
\end{gather} 
\begin{remark}
	\label{rmk difference with CGG}
	A difference between \cite[Section 5.2.2]{Cineli-Ginzburg-Gurel21} and Section \ref{subsection Crofton type inequality} happens when we prove that $\Psi$ is a submersion. 
	The difference arises from the fact that our construction of Lagrangian tomograph is different from that of \cite{Cineli-Ginzburg-Gurel21}.
	See Remark \ref{rmk constrction of Lagrangian tomograph}.
\end{remark}

Now, let us assume that $\delta$ is small enough so that \eqref{eqn condition of delta} holds. 
Then, $\Psi|_E$ is a proper submersion and Ehresmann's fibration Theorem \cite{Ehresmann} proves that $\Psi$ is a locally trivial fibration. 
Thus, there exists a metric $g_E$ on $E$ such that the restriction of $D\Psi$ to the normals to $\Psi^{-1}(x_1, x_2)$ for $(x_1, x_2) \in W_{r_1} \times W_{r_1}$ is an isometry. 

Also, since $\widehat{W}$ is a symplectic manifold, there exists a metric $g$ compatible to $\omega$.
And, $g$ implies the product metric on $\widehat{W} \times \widehat{W}$.

Now, we give the formal statement of the Crofton type inequality. 
\begin{lem}
	\label{lem Crofton type inequality}
	Let $\epsilon$ be a given positive number, and let $\delta$ be a number satisfying \eqref{eqn condition on delta} and \eqref{eqn condition of delta} for the given $\epsilon$.
	Then, there is a constant $A$ depending only on $\Psi, ds$, the fixed metric $g$ on $\widehat{W}$, and the fixed metric $g_E$ on $E$, such that the following inequality holds:
	\[\int_{B_\delta^d} N_n(S) ds \leq A \cdot \mathrm{Vol}\left(\Gamma_{H_n} \cap (W_{r_1} \times W_{r_1})\right),\]
	where $\mathrm{Vol}(\relbar)$ means the volume induced by the metric $g$ on $\widehat{W}$.
\end{lem}
\begin{proof}
	For a given $s \in B_{\delta}^d$, we set a Hamiltonian isotopy $\Phi_s: \widehat{W} \times \widehat{W} \to \widehat{W} \times \widehat{W}$ as 
	\[\Phi_s(x,y) = (\phi_{g_s}(x),y).\]
	It is easy to check that $\Phi_s$ is a Hamiltonian isotopy corresponding to the following Hamiltonian function: 
	\[\widehat{W} \times \widehat{W} \to \mathbb{R}, (x,y) \mapsto g_s(x).\]
	
	Let assume that $x \in \widehat{W}$ satisfies that 
	\[(x,x) \in \Delta \cap \Gamma_{H_n^s} \cap \left(W_{r_1 -\delta} \times W_{r_1 -\delta}\right).\]
	The first step of the proof is to show that 
	\[\Phi_s(x,x) \in \Gamma_{-g_s} \cap \Gamma_{H_n} \cap \left(W_{r_1} \times W_{r_1}\right).\]
	
	By definition, $\Phi_s(x,x) = \left(\phi_{g_s}(x), x\right)$. 
	Let $y := \phi_{g_s}(x)$. 
	Then, one can easily see that 
	\[\phi_{-g_s}(y)=\phi_{g_s}^{-1}(y)=x= \phi_{H_n^s}(x) = \left(\phi_{H_n} \circ \phi_{g_s}\right)(x) = \phi_{H_n}(y).\]
	The first and second equality are trivial by definitions, the third equality holds because of $(x,x) \in \Gamma_{H_n^s}$, the fourth equality holds because of Lemma \ref{lem Hamiltonian flow}, and the last equality is trivial. 
	Thus, $\Phi(x,x) \in \Gamma_{-g_s} \cap \Gamma_{H_n}$.
	
	Moreover, because $\left\| s \right\| < \delta$ and because of the condition (A) of Definition \ref{def Lagrangian tomograph}. (1),  
	\[r\text{-coordinate of  } \phi_{g_s}(x) \leq r\text{-coordinate of  } x + \delta.\]
	Since $x \in W_{r_1- \delta}$,	$y=\phi_{g_s}(x) \in W_{r_1}$.
	
	Since $\Phi_s$ is an isotopy, we have 
	\[N_n(S) \leq \left|\Gamma_{-g_s} \cap \Gamma_{H_n} \cap \left(W_{r_1} \times W_{r_1}\right)\right|=:N'(s).\]
	Now, it is enough to prove that 
	\begin{gather}
		\label{eqn Crofton 0}
		\int_{B_{\delta}^d} N'(s)ds \leq A \cdot \mathrm{Vol}\left(\Gamma_{H_n} \cap (W_{r_1} \times W_{r_1})\right).
	\end{gather}
	
	We recall that $B_{\delta}^d$ (resp.\ $\widehat{W}$) carries a metric. 
	Thus, the product space $B_{\delta}^d \times \widehat{W}$ and its subset $E$ carry the product metric. 
	We note that the product metric does not need to be the same as $g_E$. 
	
	We consider the map $\pi: E \hookrightarrow B^d_\delta \times \widehat{W} \to B^d_\delta$, projecting to the first factor. 
	And, for simplicity, let $\Sigma:= \Psi^{-1}\left(\Gamma_{H_n} \cap (W_{r_1} \times W_{r_1})\right)$ and let $\mathrm{Vol}_1(\relbar)$ denote the volume with respect to the product metric on $E$.  
	Then, we have 
	\begin{gather}
		\label{eqn Crofton 1}
		\int_{B_{\delta}^d} N'(s) ds = \int_{B_{\delta}^d}  \left|\left(s \times \widehat{W}\right) \cap \Sigma \right|ds = \int_{\Sigma}\pi^* ds \leq \mathrm{Vol}_1(\Sigma).
	\end{gather}
	
	Now, let $\mathrm{Vol}_2(\relbar)$ denote the volume function on the space of subsets of $E$ with respect to the metric$g_E$.
	Then, Fubini theorem implies that 
	\begin{gather}
		\label{eqn Crofton 2}
		\mathrm{Vol}(\Sigma) = \int_{\Gamma_{H_n} \cap (W_{r_1} \times W_{r_1})} \mathrm{Vol}_2\left(\Psi^{-1}(x)\right)dx \leq \max_{x \in \psi(E)} \mathrm{Vol}_2\left(\Psi^{-1}(x)\right) \cdot \mathrm{Vol}\left(\Gamma_{H_n} \cap (W_{r_1} \times W_{r_1})\right).
	\end{gather}
	
	Finally, we note that $E$ is a compact set. 
	Thus, there is a constant $A_0$ depending only on $g_E$ and the product metric on $E$ such that 
	\begin{gather}
		\label{eqn Crofton 3}
		\mathrm{Vol}_1 (\Sigma) \leq A_0 \cdot \mathrm{Vol}_2(\Sigma).
	\end{gather}
	
	Now, we combine \eqref{eqn Crofton 1}--\eqref{eqn Crofton 3} and we have 
	\[\int_{B_{\delta}^d} N'(s) ds \leq \mathrm{Vol}_1(\Sigma) \leq A_0 \cdot \mathrm{Vol}_2(\Sigma) \leq A_0 \cdot \max_{x \in \psi(E)} \mathrm{Vol}_2\left(\Psi^{-1}(x)\right) \cdot \mathrm{Vol}\left(\Gamma_{H_n} \cap (W_{r_1} \times W_{r_1})\right).\]
	It proves \eqref{eqn Crofton 0} and completes the proof.
\end{proof}

\subsection{The proof of Theorem \ref{thm vs topological entropy formal}}
\label{subsection the proof of Theorem vs topological entropy} 
In this subsection, we prove Theorem \ref{thm vs topological entropy}, or equivalently, Theorem \ref{thm vs topological entropy formal}. 
Theorem \ref{thm vs topological entropy formal} compares $\hbar^{SH}(W, \lambda)$ and $h_{top}(\mu)$ where $\mu$ denotes the time $1$ Reeb flow.
As mentioned in Section \ref{subsection sketch of the proof}, our strategy is to show Lemma \ref{lem topological entropy}, i.e., that $h_{top}(\phi_{H_1}|_{W_1})$ and $h_{top}(\mu^T)$ agree to each other first, then to compare $\hbar^{SH}(W, \lambda)$ and $h_{top}(\phi_{H_1}|_{W_1})$. 

\begin{lem}
	\label{lem topological entropy}
	In the above setting, the Hamiltonian flow $\phi_{H_1}|_{W_{r_1}}$ and the Reeb flow $\mu$ have the same topological entropy, i.e., 
	\[h_{top}\left(\phi_{H_1}|_{W_1}\right) = T\cdot h_{top}(\mu).\]
\end{lem}
\begin{proof}
	Before starting the proof, we would like to point out that for convenience, every restriction of $\phi_{H_1}$ on a subset of $\widehat{W}$ will be denoted by $\phi_{H_1}$ without mentioning the domain of the restriction if there is no chance of confusion.
	
	We recall that $H_1$ is convex radial.
	Thus, we have 
	\[H_1 = \begin{cases}
		0 \text{  on  } W \\ 
		h(r) \text{  on  } \partial W \times [1,\infty)_r
	\end{cases}.\]
	Thus, $\phi_{H_1}|_W$ is the identity map, and $\phi_{H_1}|_{[1,\infty) \times \partial W}$ is a flow of $h'(r) R$. 
	
	We note that the following facts about topological entropy are well-known:
	\begin{itemize}
		\item If $\phi: X_1 \cup X_2 \to X_1 \cup X_2$ satisfies that $\phi(X_i) \subset X_i$, then 
		\[h_{top}\left(\phi\right) = \max\{h_{top}\left(X_1\right), h_{top}\left(X_2\right)\}.\]
		\item If $\phi:X \to X$ and $Y \subset X$ satisfy that $\phi(Y)$, then
		\[h_{top}\left(\phi\right) \geq h_{top}\left(\phi|_Y\right).\]
	\end{itemize}
	See \cite[Section 1.6]{Gromov2} for more details.   
	
	From the above listed properties of the topological entropy, it is easy to check that  
	\[h_{top}\left(\phi_{H_1}|_{W_{r_1}}\right) = \max\left\{h_{top}\left(\phi_{H_1}|_W\right), h_{top}\left(\phi_{H_1}|_{[0,r_1] \times \partial W}\right)\right\}.\]
	Since every identity map has the zero topological entropy, we have 
	\[h_{top}\left(\phi_{H_1}|_{W_{r_1}}\right) = h_{top}\left(\phi_{H_1}|_{[0,r_1] \times \partial W}\right).\]
	Moreover, 
	\begin{gather}
		\label{eqn topological entropy 1}
		h_{top}\left(\phi_{H_1}|_{[0,r_1] \times \partial W}\right) \geq h_{top}\left(\phi_{H_1}|_{\{r_1\} \times \partial W}\right) = h_{top}\left(\mu^T\right).
	\end{gather}
	
	Now, we recall two classical results from \cite{Bowen71a}.
	\begin{itemize}
		\item \cite[Corollary 18]{Bowen71a}. Let $X$ and $Y$ be two compact subsets. 
		For a pair of continuous maps $\phi:X \to X$ and $\pi: X \to Y$ such that $\pi \circ \phi = \pi$, the following holds:
		\[h_{top}\left(\phi\right) = \sup_{y \in Y} h_{top}\left(\phi|_{\pi^{-1}(y)}\right).\]
		\item \cite[Proposition 21]{Bowen71a}. If $\phi^t$ is a time $t$ flow of a vector field on a compact space $X$ equipped with a metric structure, for any $t >0$, 
		\[h_{top}\left(\phi^t\right) = t h_{top}\left(\phi^1\right).\]  
	\end{itemize}
	
	Now, we consider the map 
	\[\pi: [1,r_1] \times \partial W \to [1,r_1].\]
	It is easy to check that $\pi \circ \phi_{H_1} = \pi$ and  $\phi_{H_1}|_{\pi^{-1}(r)} = \mu^{h'(r)}$ for all $r \in [1,r_1]$. 
	Thus, from \cite[Vorollary 18 and Proposition 21]{Bowen71a}, one can conclude that 
	\[h_{top}\left(\phi_{H_1}\right) := \sup_{r \in [1,r_1]} \left\{h_{top}\left(\phi_{H_1}|_{\{r\} \times \partial Y}\right) = h'(r) h_{top}\left(\mu\right) | r \in [1,r_1]\right\}.\]
\end{proof}

Now, we prove Theorem \ref{thm vs topological entropy formal}
\begin{proof}[Proof of Theorem \ref{thm vs topological entropy formal}]
	We start the proof by pointing out that the following equality is well-known: 
	\begin{gather}
		\label{eqn proof 3}
		\lim_{n \to \infty} \frac{1}{n} \log \mathrm{Vol}\left(\left\{\left(x,\phi_{H_1}^n(x)\right) | x \in W_{r_1} \right\}\right) \leq h_{top}\left(\phi_{H_1}|_{W_{r_1}}\right)
	\end{gather}
	By the definition of $\Gamma_{H_n}$, 
	\[\mathrm{Vol}\left(\left\{\left(x,\phi_{H_1}^n(x)\right) | x \in W_{r_1} \right\}\right) = \mathrm{Vol}\left(\Gamma_{H_n} \cap (W_{r_1} \times W_{r_1})\right).\]
	Because of Lemma \ref{lem topological entropy}, the inequality in \eqref{eqn proof 3} becomes 
	\begin{gather}
		\label{eqn proof 1}
		\lim_{n \to \infty} \frac{1}{n} \log \mathrm{Vol}\left(\Gamma_{H_n} \cap (W_{r_1} \times W_{r_1})\right) \leq T h_{top}(\mu).
	\end{gather} 
	
	Now, from \eqref{eqn proof 3} and Lemma \ref{lem Crofton type inequality}, we have
	\begin{gather}
		\label{eqn proof 4}
		\limsup_{n \to \infty} \frac{1}{nT} \int_{B^d_\delta} N_n(S) ds \leq h_{top}(\mu). 
	\end{gather}
	In order to prove Theorem \ref{thm vs topological entropy formal}, we relate the left-handed side of \eqref{eqn proof 4} with $b_\epsilon(n), \hbar_{\epsilon}$, and $\hbar$ defined in \eqref{eqn new notations}.
	
	Let $\epsilon$ be a fixed positive small real number. 
	Then, we will fix a positive real number $\delta$ such that
	\begin{enumerate}
		\item[(i)] the inequalities in \eqref{eqn condition on delta} hold, i.e., 	$\delta(r_1 +1 + 3 \delta) < \epsilon$ and $\delta < \frac{1}{2}$, and
		\item[(ii)] if $\left\|s \right\| \leq \delta$, then $\|H_n - H_n^s\|_{\infty, W_{r_1}}$ is sufficiently small so that the interleaving distance between $B(H_n)$ and $B(H_n^s)$ is less than $\frac{\epsilon}{2}$ by Lemma \ref{lem sup norm}.
	\end{enumerate} 
	We note that since $\delta$ is small enough to satisfy (i) and (ii), one can apply the results from Sections \ref{subsection Lagrangian tomograph} and \ref{subsection Crofton type inequality}, especially, Lemmas \ref{lem energy bound} and \ref{lem Crofton type inequality}.
	
	As we did before, let $B_\delta^d \subset \mathbb{R}^d$ be a closed ball of radius $\delta$ so that all $s \in B_\delta^d$ satisfies that $\left\|s \right\| \leq \delta$.
	
	Now, we recall that $b_{2\epsilon}(n)$ is defined as the number of bars in $B(H_n)$ satisfying conditions (A-0) and (B-0) below. 
	\begin{enumerate}
		\item[(A-0)] The length of each of the bars is larger than $2\epsilon$. 
		\item[(B-0)] The left end point of each of the bars is smaller than $nC-2\epsilon$.
	\end{enumerate}
	We note that the number of bars in $B(H_n)$ satisfying (A-0) and (B-0) is equal to $b_{2\epsilon}(n)$.
	For the second condition (B-0), see Remark \ref{rmk independency on the sequence} (2).
	
	We weaken the conditions (A-0) and (B-0) as follows: 
	\begin{enumerate}
		\item[(A-1)] The length of each of the bars is larger than to $2\epsilon$.
		\item[(B-1)] The left end point of each of the bars is smaller than $nC-2\epsilon + \delta$.
	\end{enumerate}
	Let $b_1$ be the number of bars in $B(H_n)$ satisfying (A-1) and (B-1), then 
	\[b_{2\epsilon}(n) \leq b_1,\]
	since (A-1) and (B-1) are weaker than (A-0) and (B-0).
	
	Now, we count the number of bars in $B(H_n^s)$ under the assumption that $\Delta \pitchfork \Gamma_{H_n^s}$. 
	We note that the assumption holds for almost all $s \in B_\delta^d$ 
	Let $b_2$ be the number of bars in $B(H_n^s)$ satisfying (A-2) and (B-2) stated below. 
	\begin{enumerate}
		\item[(A-2)] The lengths of each of the bars is larger than to $\epsilon$.
		\item[(B-2)] The left end point of each of the bars is smaller than $nC-2\epsilon +\delta +\tfrac{\epsilon}{2}$.
	\end{enumerate}
	Then, since $B(H_n)$ and $B(H_n^s)$ are $\frac{\epsilon}{2}$-interleaved, or equivalently, $d_{bot}(B(H_n^s),B(H_n)) <\frac{\epsilon}{2}$, we have 
	\[b_{2\epsilon}(n) \leq b_1 \leq b_2.\]
	
	Again, we can weaken the condition (B-2) as follows: 
	\begin{enumerate}
		\item[(A-3)] The length of a bar is larger than to $\epsilon$.
		\item[(B-3)] The left end point of a bar is smaller than $nC-\epsilon +\delta$.
	\end{enumerate}
	We note that (A-3) is the same as (A-2). 
	Then, let $b_3$ be the number of bars in $B(H_n^s)$ satisfying (A-3) and (B-3). 
	Since (B-2) is stronger than (B-3), 
	\[b_{2\epsilon}(n) \leq b_1 \leq b_2 \leq b_3.\]
	
	We recall that by Lemma \ref{lem svd}, the differential map $\partial: \left(CF(H_n^s), \mathcal{A}_{H_n^s}\right) \to \left(CF(H_n^s), \mathcal{A}_{H_n^s}\right)$ admits a singular value decomposition. 
	Or equivalently, there is an orthogonal basis 
	\[\mathbb{B}:= \{x_1, \dots, x_N, y_1, \dots, y_M, z_1, \dots, z_M\}\]
	of the non-Archimedean normed vector space $\left(CF(H_n^s), \mathcal{A}_{H_n^s}\right)$ such that 
	\begin{gather*}
		\partial x_i = 0 \text{  for all  } i = 1, \dots, N, \partial z_j = y_j \text{  for all  } j = 1, \dots, M, \\
		\mathcal{A}_{H_n^s}(z_1) - \mathcal{A}_{H_n^s}(y_1) \leq \mathcal{A}_{H_n^s}(z_2) - \mathcal{A}_{H_n^s}(y_2) \leq \dots \leq \mathcal{A}_{H_n^s}(z_M) - \mathcal{A}_{H_n^s}(y_M).
	\end{gather*}
	
	We note that $x_i$ (resp.\ the pair $(y_j, z_j)$) corresponds to a bar in $B(H_n^s)$ having the infinite length (resp.\ length $\mathcal{A}_{H_n^s}(z_j) - \mathcal{A}_{H_n^s}(y_j)$) and the left end point of the corresponding bar is $\mathcal{A}_{H_n^s}(x_i)$ (resp.\ $\mathcal{A}_{H_n^s}(y_j)$.)
	Thus, $b_3$ is the number of $i \in \{1, \dots, N\}$ and $j \in \{1, \dots, M\}$ satisfying conditions (A-3') and (B-3').
	\begin{enumerate}
		\item[(A-3')] $\mathcal{A}_{H_n^s}(z_j) - \mathcal{A}_{H_n^s}(y_j) > \epsilon$.
		\item[(B-3')] $\mathcal{A}_{H_n^s}(x_i), \mathcal{A}_{H_n^s}(y_j) < nC-\epsilon +\delta$.
	\end{enumerate}
	
	For the next step, we forget the condition (A-3') and weaken (B-3') as follows:
	Since $\delta (r_1 +1 + 3 \delta) < \epsilon$, we have 
	\[nC- \epsilon + \delta < nC - \delta (r_1 +1 + 2 \delta).\]
	Thus, the (B-4) below is weaker than (B-3).
	\begin{enumerate}
		\item[(B-4)] $\mathcal{A}_{H_n^s}(x_i), \mathcal{A}_{H_n^s}(y_j) < nC - \delta (r_1 +1 + 2\delta)$.
	\end{enumerate}
	
	Let $b_4$ be the number of $i \in \{1, \dots, N\}$ and $j \in \{1, \dots, M\}$ satisfying (B-4). 
	Then, we have 
	\[b_{2\epsilon}(n) \leq b_1 \leq b_2 \leq b_3 \leq b_4.\]
	
	Now, we recall that $CF(H_n^s)$ admits an orthogonal basis other than $\mathbb{B}$. 
	The other orthogonal basis is the set of $1$-periodic orbits of $X_{H_n^s}$, i.e., $\mathcal{P}_1(H_n^s)$. 
	Then, by Lemma \ref{lem one to one map}, there is a one-to-one map preserving the action values between $\mathbb{B}$ and $\mathcal{P}_1(H_n^s)$. 
	
	Let $b_5$ denote the number of $\gamma \in \mathcal{P}_1(H_n^s)$ satisfying the following (B-5):
	\begin{enumerate}
		\item[(B-5)] $\mathcal{A}_{H_n^s}(\gamma) < nC - \delta (r_1 +1+2\delta)$.
	\end{enumerate}
	Then, the one-to-one map between two orthogonal bases of $CF(H_n^s)$ shows that 
	\[b_{2\epsilon}(n) \leq b_1 \leq b_2 \leq b_3 \leq b_4 \leq b_5.\]
	
	Finally, let us consider the map $O_{H_n^s} : \mathcal{P}_1(H_n^s) \to \Delta \cap \Gamma_{H_n^s}$.
	Then, one can observe that, for any $\gamma \in \mathcal{P}_1(H_n^s)$ satisfying (B-5), there exists $x \in W_{r_1 - \delta}$ such that $O_{H_n^s}(\gamma) = (x,x) \in \Delta \cap \Gamma_{H_n^s}$, by Lemma \ref{lem energy bound}. 
	In other words, 
	\[(x,x) \in \Delta \cap \Gamma_{H_n^s} \cap \left(W_{r_1 -\delta} \times W_{r_1 -\delta}\right).\]
	
	We recall that $N_n(S)$ is defined as 
	\[N_n(S):= \left|\Delta \cap \Gamma_{H_n^s} \cap \left(W_{r_1 -\delta} \times W_{r_1 -\delta}\right)\right|.\]
	Thus, 
	\begin{gather}
		\label{eqn proof 5}
		b_{2\epsilon}(n) \leq b_1 \leq b_2 \leq b_3 \leq b_4 \leq b_5 \leq N_n(S).
	\end{gather}
	
	We note that $\epsilon$ is fixed, $\delta$ is also fixed based on the choice of $\epsilon$, and inequalities in \eqref{eqn proof 5} hold for any $n \in \mathbb{N}$.
	Thus, by combining \eqref{eqn proof 4} and \eqref{eqn proof 5}, for any $\epsilon$, we have
	\begin{gather*}
		\hbar_{2\epsilon} = \limsup_{n \to \infty} \frac{1}{nT} \log b_{2\epsilon}(n) \leq \limsup_{n \to \infty} \frac{1}{nT} \int_{B^d_\delta} N_n(S) ds \leq h_{top}(\mu).
	\end{gather*}
	Also, by taking the limit on the both sides of an inequality $\hbar_{2\epsilon} \leq h_{top}(\mu)$, we have
	\[\hbar:=\lim_{\epsilon \searrow 0} \hbar_\epsilon \leq h_{top}(\mu).\]
	In other words, Theorem \ref{thm vs topological entropy formal} holds. 
\end{proof}

\bibliographystyle{amsalpha}
\bibliography{references}
\end{document}